\documentclass[11pt]{amsart}
\usepackage[hmarginratio=1:1]{geometry}              

\let\oldtocsection=\tocsection
\let\oldtocsubsection=\tocsubsection
\let\oldtocsubsubsection=\tocsubsubsection
\renewcommand{\tocsection}[2]{\hspace{0em}\oldtocsection{#1}{#2}}
\renewcommand{\tocsubsection}[2]{\hspace{1em}\oldtocsubsection{#1}{#2}}
\renewcommand{\tocsubsubsection}[2]{\hspace{2em}\oldtocsubsubsection{#1}{#2}}

\usepackage{amssymb}
\usepackage{bbm}
\usepackage{amsthm}
\usepackage[normalem]{ulem}
\usepackage[all,cmtip,rotate]{xy}
\usepackage[linktocpage]{hyperref}
\usepackage{mathtools}
\usepackage{stmaryrd }
\usepackage{tikz}
\usepackage{tikz-cd}
\usetikzlibrary{arrows,decorations.markings, decorations.pathreplacing, calc}
\newcommand{\midarrow}{\draw[postaction={decorate}]}

\tikzset{partial ellipse/.style args={#1:#2:#3}{insert path={+ (#1:#3) arc (#1:#2:#3)} }}
\newcommand\vertex[1]{\fill #1 circle (.05)}

\usepackage{xcolor}
\definecolor{darkorange}{rgb}{1.0, 0.55, 0.0}
\definecolor{forestgreen}{rgb}{0.13, 0.55, 0.13}

\newtheorem{theorem}{Theorem}
\numberwithin{theorem}{section}

\newtheorem{lemma}[theorem]{Lemma}

\newtheorem{proposition}[theorem]{Proposition}

\newtheorem{corollary}[theorem]{Corollary}

\theoremstyle{remark}
\newtheorem{remark}[theorem]{Remark}
\theoremstyle{definition}
\newtheorem{definition}[theorem]{Definition}
\newtheorem{example}[theorem]{Example}

\newcommand{\R}{\mathbb{R}} 
\newcommand{\Eu}{\mathbb{E}} 
\newcommand{\Z}{\mathbb{Z}}  
\newcommand{\inv}{^{-1}} 
\DeclareMathOperator{\lk}{lk}
\DeclareMathOperator{\st}{st}
\DeclareMathOperator{\UL}{lk^+} 
\DeclareMathOperator{\UF}{dlk} 
\DeclareMathOperator{\Sing}{split}
\DeclareMathOperator{\id}{id}
\DeclareMathOperator{\GL}{GL}
\DeclareMathOperator{\SL}{SL}
\DeclareMathOperator{\SO}{SO}

\DeclareMathOperator{\Out}{Out}
\DeclareMathOperator{\Min}{Min}
\DeclareMathOperator{\Br}{br} 
\DeclareMathOperator{\spn}{span}
\DeclareMathOperator{\Isom}{Isom}

\newcommand{\G}{\Gamma}             
\newcommand{\AG}{A_\G}                 
\newcommand{\GW}{$\G$-Whitehead} %
\newcommand{\WP}{{\mathcal P}}     
\newcommand{\WQ}{{\mathcal Q}}     
\newcommand{\WR}{{\mathcal R}}      

\newcommand{\Ca}{\mathcal{C}} 
\newcommand{\eqH}{[\![H]\!]} 
\newcommand{\eqK}{[\![K]\!]} 
\newcommand{\eqL}{[\![L]\!]} 

\newcommand{\bPi}{\Pi}  
\newcommand{\lkp}{\lk_\bPi}
\newcommand{\stp}{\st_\bPi}
\newcommand{\dlkp}{\UF_\bPi}

\newcommand{\SG}{{\Sigma}_\Gamma}                    
\newcommand{\tS}{{\mathcal{T}}_\Gamma}              
\newcommand{\OG}{\mathcal{O}_{\Gamma}}            
\newcommand{\SaG}{\mathbb{S}_{\Gamma}}           
\newcommand{\Sa}{\mathbb{S}}                                 
\newcommand{\SP}{\Sa^\bPi}                                     
\newcommand{\SPp}{{\Sa}^\Omega}                           
\newcommand{\USP}{\widetilde\Sa^\bPi}                     
\newcommand{\F}{\mathcal{F}}                                    
\newcommand{\Fp}{\mathcal{G}}                                  
\newcommand{\E}{\mathcal{E}}                                    
\newcommand{\EP}{\mathbb{C}^\bPi}			 
\newcommand{\Sv}{\mathbb K_v}              
\newcommand{\Slkv}{\mathbb K_{\lk(v)}}   
\newcommand{\Snlv}{\mathbb K_{\UF(v)} }
\newcommand{\Sulv}{\mathbb K_{\UL(v)} }
\newcommand{\USv}{\widetilde{\mathbb K}_v} 
\newcommand{\USlkv}{\widetilde{\mathbb K}_{\lk(v)}} 
\newcommand{\USlkm}{\widetilde{\mathbb K}_{\lk(m)}} 
\newcommand{\USnlv}{\widetilde{\mathbb K}_{\UF(v)} }

\newcommand{\UH}{\widetilde H} 

\newcommand{\re} {\mathbbm r} 
\newcommand{\dsx}{\mathbbm x} 

\newcommand\redvertex[1]{\fill[red] #1 circle (.075)} 

\newcommand{\pr}{\text{pr}} 
\newcommand{\iso}{\cong}   
\newcommand{\bdry}{\partial} 

\newcommand{\overbar}[1]{\mkern 5mu\overline{\mkern-5mu#1\mkern-5mu}\mkern 5mu}

\author[Corey Bregman, Ruth Charney and Karen Vogtmann]{Corey Bregman,
Ruth Charney and
 Karen Vogtmann}

 \title{Outer space for  RAAGs}

\begin{document}

\thanks{Bregman was partially supported by NSF grant DMS-1906269.}
\thanks{Charney was partially supported by NSF grant DMS-1607616.}

\begin{abstract}
For any right-angled Artin group $\AG$ we construct a finite-dimensional space $\OG$ on which the group $\Out(\AG)$ of outer automorphisms of $\AG$ acts with finite point stabilizers.  We prove that $\OG$ is contractible, so that the quotient is a rational classifying space for $\Out(\AG)$.  The space $\OG$ blends features of the symmetric space of lattices in $\R^n$ with those of  Outer space for the free group $F_n$.  Points in $\OG$ are locally CAT(0) metric spaces that are homeomorphic (but not isometric) to certain locally CAT(0) cube complexes, marked by an isomorphism of their fundamental group with $\AG$.    \end{abstract}

\maketitle

\tableofcontents

\section{Introduction}\label{sec:intro}  A lattice $\Lambda$ in a semi-simple Lie group $G$ acts discretely on the symmetric space $ G/K$, and a very well-developed  theory shows that the algebraic structure of $\Lambda$ is intimately connected to the geometric structure of $G/K.$   The   study of surface mapping class groups by Thurston, Harvey and Harer among others borrowed ideas from this classical subject, using Teichm\"uller space as a substitute for the symmetric space, and this point of view  proved to be   extremely fruitful.
An analog of symmetric spaces and Teichm\"uller spaces called ``Outer space" was later produced for the  purpose of studying the group of outer automorphisms of a free group \cite{CuVo}.  The study of this group, space and action has frequently been guided by Thurston's ideas, but there are some respects in which  $\Out(F_n)$ more closely resembles a lattice than a  mapping class group. For example, mapping class groups are automatic \cite{Mosher}, while for $n\geq3$,  $\Out(F_n)$~\cite{BriVogGeom} and $\GL(n,\Z)$~\cite{ECHLPT} are not.

In this paper we study outer automorphism groups of  right-angled Artin groups, a class which includes both $\Out(F_n)$ and the most basic lattice, $\GL(n,\Z)=\Out(\Z^n).$
Recall that a {\em right-angled Artin group (RAAG)} is defined by a presentation with a finite set of generators  together with  relations specifying that some of the generators commute.  A convenient way of expressing this is to draw a graph $\G$ with one vertex for every generator and one edge connecting each pair of commuting generators; the resulting RAAG is denoted $\AG$. In recent years RAAGs and their automorphism groups have played a prominent role in geometric group theory and low-dimensional topology.  RAAGs are linear groups and they arise naturally as subgroups of many other groups such as mapping class groups, Coxeter groups, and more general Artin groups \cite{CLM, CrPa, DaJa, Kob}.  Conversely, while all subgroups of free (or free abelian) groups are themselves free (or free abelian), a surprisingly diverse array of groups can be realized as subgroups of RAAGs, including surface groups and many 3-manifold groups \cite{BeWi, DSS, HaWi}.  The fact that the fundamental group of every closed hyperbolic 3-manifold virtually embeds in a RAAG was central to Agol's proof of the virtual Haken Conjecture \cite{AGM}, the final step in Thurston's program to classify 3-manifolds.  The diversity of subgroups has also made RAAGs a fertile source of counterexamples for a variety of conjectures \cite{BeBr, CrKl}.

To date, outer automorphism groups of RAAGs have primarily been studied from an algebraic point of view.  See, for example, \cite{ChVo1, ChVo2, Day, DayWade}.
As the case of mapping class groups and $\Out(F_n)$ clearly demonstrates, geometric approaches to studying such groups can be very effective.  In this paper we focus on constructing an analog of outer space for RAAGs that will allow us to apply similar methods to the study of $\Out(\AG)$.
Some initial steps in this direction appear in previous papers.   In \cite{CCV}, Charney and Vogtmann, together with Crisp, constructed a candidate outer space for two-dimensional RAAGs (those for which $\G$ contains no triangles), but there is no apparent way to generalize this to higher dimensions.  Then in \cite{CSV}, together with Stambaugh, they constructed a contractible space $K_\G$ with a proper action of a certain subgroup of $\Out(\AG)$.  This subgroup, denoted $U(\AG)$, is made up of ``untwisted" outer automorphisms of $\AG$ that behave more like automorphisms of free groups.  In particular, it excludes transvections between commuting pairs of generators.
In the current paper, we use the space $K_\G$ as a starting point to build an outer space for the full outer automorphism group.

Outer space for free groups, $CV_n$, can be described as a space of marked metric graphs with fundamental group $F_n,$ where the marking specifies an isomorphism of $\pi_1$ with $F_n$.  Similarly, the symmetric space $Q_n=\SO(n)\backslash \SL(n,\R)$ can be described as the space of marked flat tori with fundamental group $\Z^n$, where the marking gives an isomorphism of $\pi_1$ with $\Z^n$. Thus the basic objects in $Q_n$ (tori) are all homeomorphic but have different flat metrics, while the basic objects in $CV_n$ (graphs)  have different homeomorphism types as well as different metrics.  These different homeomorphism types, however, all have a common quotient, an $n$-petaled rose, obtained by collapsing any maximal tree.   For a general RAAG,  there is a canonical construction of a CAT(0) cube complex $\SaG$ with fundamental group $\AG$, known as the  {\em Salvetti complex}, which has a $k$-torus for each $k$-clique in $\G$.  In the new outer space,  this complex plays the role of the $n$-petaled rose.   The basic objects in our outer space $\OG$ are locally CAT(0) metric spaces $(Y,d)$ containing contractible subspaces (analogous to maximal trees) that can be collapsed to produce a quotient homeomorphic to $\SaG$.  Each $(Y,d)$ is made up of a collection of (intersecting) flat tori marked by the free abelian groups generated by cliques in $\G$.  A point in $\OG$ consists of one of these metric spaces $(Y,d)$ marked by an isomorphism of $\pi_1(Y)$ with $\AG$.

More precisely, the spaces $Y$ are homeomorphic (but not isometric) to non-positively curved cube complexes called {\em $\G$-complexes}, which were previously introduced in \cite{CSV}.  Marked $\G$-complexes form a partially ordered set whose geometric realization is the simplicial complex $K_\G$ mentioned above.  In $K_\G$, $\G$-complexes are viewed as combinatorial objects, not as metric objects, and the markings are of restricted type, allowing only an action of the subgroup $U(\AG)$.  In the current paper, $\G$-complexes are endowed with locally CAT(0) metrics that make the interior of each ``cube" isometric to a Euclidean parallelotope.  We call this a \emph{skewed} $\G$-complex.  The objects $(Y,d)$ in $\OG$ are isometric to skewed $\G$-complexes.  The markings are  arbitrary, and objects are equivalent if they  are isometric by a map that commutes with the marking, up to free homotopy.  As in the special cases of $\GL(n,\Z)$ acting on $Q_n$ and $\Out(F_n)$ acting on $CV_n$, $\Out(\AG)$ acts on $\OG$ by changing the marking. The main theorem states,

\begin{theorem}\label{thm:main} For any right-angled Artin group $\AG$, the space $\OG$ is finite-dimensional, contractible and  the action of the group $\Out(\AG)$ has finite point stabilizers.\\
\end{theorem}
We now give a brief outline of the proof. The proof begins with the space $K_\G$ which, as noted above, was shown in \cite{CSV} to be contractible.
The passage from  $K_\G$ to $\OG$ involves several intermediate steps. First, we embed $K_\G$ into a new space $\SG$ by endowing $\G$-complexes with metrics making ``cubes"  into orthotopes, i.e. orthogonal products of intervals of various lengths; these are called {\em rectilinear} $\G$-complexes.
 In the case of a free group, this corresponds to embedding the spine of outer space into the full outer space $CV_n$.
 As in the case of $K_\G$, the action on $\SG$ is restricted to the subgroup $U(\AG)$.  This is a result of allowing only certain types of markings, called {\em untwisted markings.}  It is easy to show that $K_\G$ is a deformation retract of $\SG$, so that $\SG$ is contractible.

Next, we allow the orthotopes in a $\G$-complex to skew so that they become parallelotopes.  This is done in a controlled manner, resulting in an  \emph{allowable parallelotope structure} which is still locally CAT(0).  The collection of skewed $\G$-complexes  with untwisted markings is denoted $\tS$.  We show that there is a deformation retraction of $\tS$ onto $\SG$ defined by straightening the parallelotopes, so $\tS$ is also contractible.

The action on $\tS$ is still restricted to the subgroup $U(\AG)$.
To get a space on which all of $\Out(\AG)$ acts we must allow for transvections between commuting elements; these are called {\em twists}.  To see how this is done, consider the case of a marked metric torus $T^n$.  One can think of a change of marking as either a change in the isomorphism $\pi_1(T^n) \to \Z^n$, or as a change in the shape of the parallelotope whose quotient is $T^n$. To reconcile these viewpoints in the case of a skewed $\G$-complex, we put an equivalence relation on the points in $\tS$.  Namely, two skewed $\G$-complexes with specified markings are equivalent if they are isometric by a map that takes one marking to the other (up to homotopy), where this map need \emph{not} preserve the combinatorial structure.  Then up to equivalence, we can accomplish twists   by adjusting the skewing of appropriate tori in the $\G$-complex.

The points in the new outer space $\OG$ are equivalence classes of points in $\tS$, thus there is a natural surjection $ \tS \to \OG$.  The proof of Theorem \ref{thm:main} consists in showing that this map is a fibration with contractible fibers.
The key problem is understanding to what extent the combinatorial structure on  a marked skewed $\G$-complex is determined by its  metric.  For this, we divide the hyperplanes into two classes, twist-minimal and twist-dominant, and show that the twist-minimal hyperplanes are completely determined by the metric.  The twist-dominant hyperplanes on the other hand, depend on the shapes of the parallelotopes and can vary within a fiber. To show contractibility, we encode the allowable skewings by a vector in a Euclidean space and prove that the set of points in a fiber corresponds to a convex subspace of this Euclidean space.

Theorem~\ref{thm:main} is a first step towards a more geometric study of $\Out(\AG)$.  It leads to many natural questions, a few of which we now discuss briefly.

 The dimension of $\OG$ can be computed (with some effort) by looking at the graph $\G$.   As is the case for symmetric spaces and Teichm\"uller spaces, the action of $\Out(\AG)$ on $\OG$  is not cocompact, and this dimension is quite a bit larger than the virtual cohomological dimension ({\sc vcd}) of $\Out(\AG).$ An algebraic algorithm for computing this {\sc vcd} has been established by Day--Sale--Wade \cite{DaySaleWade}.   For both $\GL(n,\Z)$ and $\Out(F_n)$ there is an equivariant deformation retract (a ``spine") of dimension equal to the {\sc vcd}, and it would be interesting to find an analogous spine for $\OG$.   The construction of such a spine might be fairly subtle, as it was shown in \cite{MillardV} that the dimension of $K_\G$, though often equal to the {\sc vcd} of $U(\AG),$ is sometimes strictly larger.  As an aside, we remark that no natural spine has yet been constructed for the action of the mapping class group of a closed surface on Teichm\"uller space.

Much of the work on $\Out(F_n)$ and $\GL(n,\Z)$ (as well as surface mapping class groups) depends on understanding the structure of the associated space at or near infinity, e.g. by adding a ``boundary" that compactifies either the space or its quotient, and studying the action on this boundary.    Thurston compactified Teichm\"uller space  by embedding it into the space of projective  length functions for the fundamental group of the surface,  Outer space can be compactified by embedding it in the space of projective length functions on $F_n$, and  the symmetric space $Q_n$  embeds into the space of projective length functions on $\Z^n$. Vijayan  \cite{Anna}   initiated a study of length functions on $\AG$, which was further developed by and Beyrer and Fioravanti \cite{BeyrerFioravanti}, who used length functions to compactify the ``untwisted" outer space $K_\G$ of \cite{CSV}.  A different way of understanding the structure at infinity is by ``bordifying" the space, which compactifies not the space but rather the quotient.  There are bordifications of  $Q_n$ (defined in much more generality by Borel and Serre \cite{BoSe}) and $CV_n$ (defined by Bestvina and Feighn \cite{BeFe}), who used them to prove that the respective groups are virtual duality groups in the sense of Bieri and Eckmann.  Is there an analogous bordification of $\OG$?   The question is  subtle, as  Br\"uck and Wade showed that $\Out(\AG)$ is not always a virtual duality group \cite{BruckWade}.

A space is a {\em classifying space for proper $G$-actions} if fixed point sets of finite subgroups are contractible. Such a space is called an $\underbar{$E$}G$-space. These are useful, for example, for studying centralizers of finite-order elements.  In  addition, we recall that the Baum-Connes Conjecture relates the topological K-theory of the reduced $C^*$-algebra of $G$ to an appropriate equivariant homology theory evaluated at $\underbar{$E$}G$. Both $Q_n$ and $CV_n$ are classifying spaces for proper actions, so it is natural to ask whether $\OG$ is likewise for $\Out(\AG)$.

Finally, both symmetric spaces and Outer space for free groups can be equipped with useful metrics (though the most intensively studied metric structure on Outer space is an asymmetric metric). A geometric approach often gives a
simpler, more natural explanation for algebraic features of the group.  Is there a good metric on $\OG$?  How do geodesics in this metric behave?

The paper is organized as follows.  In Section~\ref{sec:preliminaries}, we establish basic terminology, recall the construction of the space $K_\G$ and embed it into a space $\Sigma_\G$.
  In Section~\ref{sec:blowups}, we establish some basic properties of $\G$-complexes which will be needed later in the paper.
  In Section~\ref{sec:TMHyperplanes}, we introduce the notion of twist-dominant and twist-minimal hyperplanes and investigate the extent to which these notions depend on the choice of $\G$-structure and the marking.
In Section~\ref{sec:twistedSP}, we define an allowable parallelotope structure on a $\G$-complex and show that the resulting path metric is locally CAT(0).
Finally, in Section~\ref{sec:tS}, we prove that the space $\tS$ of skewed $\G$-complexes deformation retracts to $\Sigma_\G$, hence is contractible.
In Section~\ref{sec:OG}, we define our outer space $\OG$ and show that the natural map $\tS\to \OG$ is a fibration with contractible fibers.

The authors would like to thank the referees for their careful reading of the paper and many helpful comments.

\section{Preliminaries}\label{sec:preliminaries}

We fix a finite simplicial graph $\G=(V,E)$ throughout the paper, and denote by $\AG$  the associated  right-angled Artin group.  In this section we give a brief account of the contents of \cite{CSV}. We refer the reader to \cite{CSV} for further details.

\subsection{Graph terminology}   For $v\in V$, the \emph{link}, $\lk(v)$, is the full subgraph of $\G$ spanned by vertices adjacent to $v$, and the \emph{star}, $\st(v)$, is the full subgraph spanned by $\lk(v)$ and $v$.  If $W\subset V$, then $\lk(W)=\cap_{w\in W} \lk(w)$ and $\st(W)$  is the full subgraph spanned by  $\lk(W)\cup W$.

Define  $v\leq w$ to mean $\lk(v)\subseteq \st(w)$. This can happen in one of two ways: either $\lk(v)\subseteq\lk(w)$, in which case we write $v\leq_f w$, or $\st(v)\subseteq\st(w)$, in which case we write $v\leq_t w$. These are mutually exclusive unless $v=w$.

The following elementary lemma puts a restriction on the star- and link-orderings.

\begin{lemma}\label{lem:twist_and_fold}  If $u\leq _tv\leq_fw$ then either $v=u$ or $v=w$.
\end{lemma}
\begin{proof} Suppose $u\neq v$. Since $u\in \lk(v)$ and $\lk(v)\subseteq\lk(w)$, $u\in \lk(w)$.  Since $\st(u)\subseteq\st(v)$, this implies $w\in \lk(v)$. Hence $v\leq_t w$,  which is impossible unless $v=w$.
\end{proof}

If $v\leq_* w$ and $w\leq_* v$ we say that $v$ and $w$ are \emph{equivalent}, and write $w\sim_* v$, where $*=f, t$ or $\emptyset$.  The notation $v \leq_* w$ is justified by the fact that the induced relation  on equivalence classes $[v]$ is a partial ordering.  It will often be important to be more specific, so if $\lk(v)=\lk(w)$ we say that $v$ and $w$ are {\em fold-equivalent}, and if $\st(v)=\st(w)$ we  say $v$ and $w$ are {\em twist-equivalent.} 

For each $v\in V$ we divide the elements of $V_{\geq v}=\{u|u\geq v\}$ into two groups, namely
\begin{itemize}
\item  $\lk^+(v)= \{u| u\geq v \hbox{ and } u\in \lk(v)\}= \{u \in V | u \geq_t v, u\neq v\}$
\item $\UF(v)   =\{u| u\geq v \hbox{ and } u\notin \lk(v)\}  =\{u \in V| u\geq_f v\}$
\end{itemize}
See Figure~\ref{fig:lklk} for an example.
Observe that $\UF(v)$ is equal to the ``double link" $ \lk(\lk(v))$, i.e. every vertex in $\UF(v)$ commutes with every vertex in $\lk(v)$. Also observe that  if $u,u'\geq_tv$ then $u$ is connected to  $u'$, so  $\{v\}\cup \UL(v)$ is a clique.

 The following distinction will be critical when we define the points in our new outer space.
\begin{definition} A vertex $v\in\G$ is \emph{twist-dominant} if there is some $u\neq v$ with $v\geq_t u$, and is {\em twist-minimal} otherwise.
\end{definition}

 Note that elements of $\lk^+(v)$ are all twist-dominant, while elements of $\UF(v)$ may be either twist-dominant or twist-minimal.

 \begin{figure}
\begin{center}\begin{tikzpicture}[scale=.80]
 \coordinate (v) at (0,0); \coordinate (u1) at (2,1); \coordinate (u2) at (2,-1); \coordinate (u3) at (2,0);
 \coordinate (w1) at (5,1.5); \coordinate (w2) at (5,0); \coordinate (w3) at (5,-1.5);
  \coordinate (x) at (4,2.5); \coordinate (y) at (7,1); \coordinate (z) at (7,-1);
 \vertex{(v)};\vertex{(u1)};\vertex{(u2)};\vertex{(u3)};
 \vertex{(w1)};\vertex{(w2)};\vertex{(w3)};
 \vertex{(x)};\vertex{(y)};\vertex{(z)};
\draw (v) to (u1) to (w1)   to (u2) to (v);
\draw (u1) to (u3) to (w3) to (u2) to (u3);
\draw (w1) to (u2);
\draw (u1) to (w2);
\draw (y) to (w2) to (z);
\draw (u1) to (x);
\draw (v) to (u3) to (w2);
\draw (u3) to (w1);\draw (w2) to (w3);\draw(w2) to (u2);
\node [left](v) at (v) {$v$};
\node [above](u1) at (u1) {$u_1$};
\node [below](u2) at (u2) {$u_2$};
\node [above left](u3) at (u3) {$u_3$};
\node [above](w1) at (w1) {$w_1$};
\node [above](w2) at (w2) {$w_2$};

\end{tikzpicture}
\end{center}
\caption{ $\lk(v)=\{u_1,u_2,u_3\}$,  $\UL(v)=\{u_3\}$, $\UF(v) = \{v,w_1,w_2\}$
}\label{fig:lklk}
\end{figure}

\subsection{Salvetti complexes}  For a simplicial graph $\G$ the {\em Salvetti complex\,}  $\SaG$ is a  cube complex with one $k$-cube for each $k$-clique in $\G$; in particular it has a single 0-cube (for the empty clique) and a 1-cube  for each vertex (=1-clique) of $\G$. The $2$-skeleton of
$\SaG$ is the standard presentation complex for $\AG$, so $\pi_1(\SaG)\iso \AG$.   The addition of higher-dimensional cubes guarantees that $\SaG$ satisfies Gromov's link condition, i.e. all links are flag.  Therefore if all cubes of $\SaG$ are identified with standard Euclidean cubes $[0,1]^k$ then the induced path metric on $\SaG$ is non-positively curved (locally CAT(0)) and its universal cover $\widetilde{\Sa}_\G$ is CAT(0).  In Figure~\ref{fig:salvetti} we show a simple example of a graph $\G$ and its Salvetti complex.  In this example the Salvetti is made of two tori glued along a circle labeled $b$ plus a loop labeled $d$ at the basepoint. In the right-hand picture we have cut open the tori.
\begin{figure}
\begin{center}
 \begin{tikzpicture}
  \begin{scope}[decoration={markings,mark = at position 0.5 with {\arrow{stealth}}}]
  \coordinate (a) at (-1,0); \coordinate (b) at (0,1); \coordinate (c) at (1,0);\coordinate (d) at (0,-1);
 \vertex{(a)};\vertex{(b)};\vertex{(c)};\vertex{(d)};
 \draw (a) to (b) to (c);
 \node [above](A) at (a) {$a$};
  \node [above](B) at (b) {$b$};
    \node [above](C) at (c) {$c$};
      \node [above](D) at (d) {$d$};
 \begin{scope}[xshift=10cm, scale=.9]
 \midarrow [thick, blue] (-1,-2) to (1,-2); \node [above] (b1) at (0,-2) {$b$};
  \midarrow [thick, blue] (-1,0) to (1,0); \node [above] (b2) at (0,0) {$b$};
   \midarrow [thick, blue] (-1,2)  to (1,2); \node [above] (b3) at (0,2) {$b$};
      \midarrow [thick, red] (1,-2) to (1,0);  \node [right] (abot) at (1,-1) {$a$};
 \midarrow [thick, red] (-1,-2) to (-1,0);  \node [right] (atop) at (-1,-1) {$a$};
       \midarrow[thick, forestgreen] (1,0) to (1,2);\node [right] (ctop) at (1,1) {$c$};
 \midarrow [thick, forestgreen] (-1,0) to (-1,2);     \node [right] (cbot) at (-1,1) {$c$};
\midarrow [thick] (1,0) .. controls (2.35,1.25) and (2.35,-1.25) .. (1,0);
    \node [right] (dtop) at (2,0) {$d$};
 \end{scope}
 \node (is) at (8,0) {$=$};
 \begin{scope}[ xshift=4cm, scale=.35, yshift=1cm]
\draw (0,0) ellipse (4cm and 3cm);  
\draw (0,-2) [partial ellipse=55:125:2.5cm and 2.5cm];
\draw (0,2) [partial ellipse=225:315:2.5cm and 2.5cm];
\begin{scope}[xshift=5.5cm].  
\draw (0,0) ellipse (4cm and 3cm);
\draw (0,-2) [partial ellipse=55:125:2.5cm and 2.5cm];
\draw (0,2) [partial ellipse=225:315:2.5cm and 2.5cm];
\end{scope}
 \draw [thick, blue] (2.75,0) [partial ellipse=0:180:1.25cm and .5cm];
  \draw [thick, blue, densely dotted] (2.75,0) [partial ellipse=180:360:1.25cm and .5cm];
  \node [above](b) at (2,.3) {$b$}; 
\draw [thick, red] (5.5,.5) ellipse (2.75cm and 1.75cm);  
 \node [below] (a) at (9,1) {$a$};  
   \midarrow [red, thick] (8.25,.5) to (8.25,.6);
\draw [thick, forestgreen] (0,.5) ellipse (2.75cm and 1.75cm); 
 \node [below] (c) at (-3.4,1) {$c$};  
   \midarrow [forestgreen, thick] (-2.75,.5) to (-2.75,.6);
\draw [thick] (2.75,.5) .. controls (0,6) and (5.5,6) .. (2.75,.5);  
\fill (2.75,.5) circle (.15);  

\node [above](d) at (2,4.25) {$d$};
  \midarrow [blue, thick] (2,.4) to (1.95,.39);
    \midarrow [thick] (2.55,4.6) to (2.54,4.6);
 \end{scope}
 \end{scope}
\end{tikzpicture}
\caption{A graph $\G$ and its Salvetti $\SaG.$}\label{fig:salvetti}
\end{center}
\end{figure}

Throughout this paper we will assume familiarity with the language of locally CAT(0) and CAT(0) cube complexes, including hyperplanes, minsets, etc. as can be found, e.g. in \cite{BH}.

\subsection{$\G$-Whitehead partitions}\label{sec:GWpartitions}
\subsubsection{Definition and examples}
Let  $V\cup V^{-1}$ be the generators of $\AG$ and their inverses, and let $m$ be a vertex of $\G$.    A  {\em  \GW\  partition $\WP$ based at $m$} is a partition of  $V\cup V^{-1}$  into three parts $P^+, P^-$ (called the {\em sides} of $\WP$) and $\lk(\WP)$, where
\begin{itemize}
\item $\lk(\WP)$ consists of all generators  that commute with  $m,$ and their inverses.
\item The sides of $\WP$ form a thick partition of  $V\cup V^{-1} \setminus \lk(\WP)$ (recall that a partition is {\em thick} if it has at least two elements on each side).
\item  $m$ and $m^{-1}$ are in different sides of $\WP$
\item If $v\neq w$ are in the same component of $\G\setminus\st(m)$, then $v, v^{-1}, w$ and $w^{-1}$ are all in the same side of $\WP$.
\end{itemize}
A more succinct way to define a \GW\ partition $\WP$ based at $m$ is by forming a graph  $\G^\pm$  with one vertex for each element of $V\cup V^{-1}$ and an edge between distinct vertices $x$ and $y$ whenever $x$ and $y$ commute but are not inverses. If we let $\lk^\pm(m)$ be the link of $m$ in $\Gamma^\pm$ and $\mathcal C(m)=\{m,m\inv, C_1,\ldots, C_k\}$ be the components of $\G^\pm\setminus \lk^\pm(m)$, then
\begin{itemize}
\item $\lk(\WP)$ consists of vertices in $\lk^\pm(m)$, and
\item  the sides of $\WP$ form a thick partition of $\mathcal C(m)$ that separates $m$ from $m\inv$.
\end{itemize}
 The components $C_1,\ldots,C_k$ are called {\em $m$-components.} Thus $m$ together with any proper subset of $m$-components gives one side of a valid  \GW\ partition based at $m$.

A \GW\ partition $\WP$ based at $m$ determines an automorphism  $\phi(\WP,m)$ of $\AG$ called a {\em \GW\ automorphism}.  Examples of \GW\ automorphisms include partial conjugations and elementary folds; for details see ~\cite{CSV}.  Different bases for $\WP$ give different automorphisms, but the partition $\WP$ itself does not depend on the choice of base, and we will often not specify a base. Note that  a \GW\ partition is completely determined by giving one of its sides.

\begin{example} The following are  three examples of \GW\ partitions for the graph $\G$ depicted in Figure~\ref{fig:salvetti},
\begin{itemize}
\item $\WP=(P^+|P^-|\lk(\WP))=(\{b,d\}|\{b\inv,d\inv\}|\{a,a\inv, c,c\inv\})$
\item $\WQ=(Q^+|Q^-|\lk(\WQ))=(\{a,d\}\,|\{a\inv,d\inv,c,c\inv\}|\{b,b\inv\})$
\item $\WR=(R^+|R^-|\lk(\WR))=(\{a,c,d\}\,|\{a\inv,d\inv,c\inv\}|\{b,b\inv\})$
\end{itemize}
Here $\WP$ is based at $b$, $\WQ$ is based at $a$, and $\WR$  can be based at either $a$ or $c$.
\end{example}

\subsubsection{Properties}

\begin{definition}If $v$ and $v^{-1}$ are in different sides of $\WP$, we say $\WP$ {\em splits} $v$.
Define $\Sing(\WP)$ to be the set of vertices of $\G$ that are split by $\WP$, and
$$
\max(\WP) = \{ v\in V \mid v \textrm{ is a maximal element in } \Sing(\WP)\},
$$
where maximality is with respect to the relation $\leq$ defined above.
For a vertex $v\in V$, it is convenient to also define $\max(v)=\{v\}.$
\end{definition}
\begin{lemma}\label{lem:base}
 If $\WP$ is based at $m$ and $\WP$ splits $v$ then  $v \leq_f m$.
\end{lemma}
\begin{proof} Since  $\WP$ splits $v,$  $v$ is not in the link of $m$.  Suppose $w$ is in the link of $v$. Since $\WP$ splits $v,$   $v$ and $w$ are not in the same component of $\G-\st(m)$, so $w$ must be in the link of $m$.  This shows $v\leq_f m$.
   \end{proof}
It follows that the elements of $\max(\WP)$ are precisely the bases of $\WP$,  and they are all fold-equivalent.
\begin{lemma}\label{lem:tdominant}
 If $\WP$ splits a twist-dominant vertex $v$, then $\max(\WP)=\{v\}$.
\end{lemma}
\begin{proof} Let $m\in \max(\WP)$.  By Lemma~\ref{lem:base}, $v\leq_f m$. Since $v$ is twist-dominant there is a $w\neq v$ with $v\geq_tw$. But then $v=m$ by Lemma~\ref{lem:twist_and_fold}.
\end{proof}

We extend our orderings on vertices of $\G$ to \GW\ partitions by declaring $\WP\leq_*\WQ$ for $*=f$ or $t$ if for some (and therefore any) $v\in \max(\WP)$ and $w\in \max(\WQ)$ we have $v\leq_* w$.

\subsubsection{Adjacency, compatibility, consistency}
\begin{definition}\label{def:adjacent} Let $\WP$ and $\WQ$ be \GW\ partitions.  We say $\WP$ and $\WQ$ are  {\em adjacent}  
if $\max(\WP)\subset \lk(\WQ)$.  A vertex $v$ is {\em adjacent} to $\WP$ if $v\in\lk(\WP)$, and $v$ and $w$ are {\em adjacent} if they are adjacent in $\Gamma$.
\end{definition}
Since all elements of $\max(\WP)$ have the same link, $\max(\WP)\subset \lk(\WQ)$ if and only if $\max(\WQ)\subset \lk(\WP),$ i.e. the definition is symmetric.

\medskip

{\bf Warning}:  In \cite{CSV} we said ``$\WP$ and $\WQ$ {\em commute}" instead of ``$\WP$ and $\WQ$ are {\em adjacent}."  There are two reasons for changing the terminology here.  First,  two partitions based at the same vertex $v$ do not ``commute" in the sense of \cite{CSV} even though the generator $v$ certainly commutes with itself; this caused confusion for several readers.  The second reason is that the definition of ``commute" written in \cite{CSV} is not actually the one used in the proofs of the lemmas: we mistakenly added a condition in the definition requiring that the twist-equivalence classes of $\max(\WP)$ and $\max(\WQ)$ be different.   The proofs of all lemmas in that paper about commuting partitions are correct, however, if one replaces ``commuting" by the definition of ``adjacent" given above. 

 \begin{definition}Let $\WP$ and $\WQ$ be  distinct  \GW\ partitions. \begin{enumerate}  
\item{} $\WP$ and $\WQ$   are {\em compatible} if either  $\WP$ and $\WQ$ are adjacent 
or they have  sides $P^\times$ and $Q^\times$ with $P^\times\cap Q^\times=\emptyset$.
\item{} Sides $P^\times$ of $\WP$ and $Q^\times$  of $\WQ$ are {\em consistent} if  either   $\WP$ and $\WQ$ are adjacent
 or $P^\times \cap Q^\times\neq\emptyset.$
\end{enumerate}
\end{definition}
If $\WP$ and $\WQ$ are compatible but  are not adjacent,  
then exactly three of the four possible choices of pairs of sides are consistent, by  Lemma 3.6 of \cite{CSV}. (If they  are adjacent,
then any choice of sides is consistent.)

Define an involution $P^\times\mapsto \overbar{P^\times}$ that switches sides of $\WP$, i.e. $\overbar{P^+}=  P^-$ and $\overbar{P^-}=  P^+.$

\begin{lemma}\label{lem:noncommuting} If  $\WP$ and $\WQ$ are compatible but not adjacent  
and $P^\times\cap Q^\times=\emptyset$,  then $P^\times\cap \lk(\WQ)=\emptyset$  so $P^\times\subset \overbar{Q^\times}$; similarly  $Q^\times\subset \overbar{P^\times}$.
\end{lemma}

\begin{proof} This is Lemma 3.4 of  \cite{CSV}. It is illustrated in Figure~\ref{fig:nonadjacent}
\end{proof}

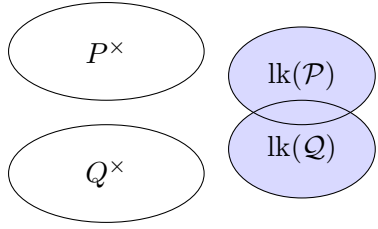
\begin{figure}
\begin{center}
   \begin{tikzpicture}[scale=.65]
   \coordinate (cQ) at (0,0);
      \coordinate (cP) at (0,2.5);
      \coordinate (clP) at (4,2.0);
      \coordinate (clQ) at (4,.5);
 \draw  (cP) ellipse (2cm and 1cm); \node (P) at (cP) {$P^\times$};
  \draw   (cQ) ellipse (2cm and 1cm); \node (P) at (cQ) {$Q^\times$};
\fill [blue!15](clP) ellipse (1.5cm and 1cm); \node (lkP) at (clP) {$\lk(\WP)$};
\draw [fill=blue!15](clQ) ellipse (1.5cm and 1cm); \node (lkQ) at (clQ) {$\lk(\WQ)$};
\draw (clP) ellipse (1.5cm and 1cm);
  \end{tikzpicture}
 \caption{(Lemma~\ref{lem:noncommuting}): Non-adjacent partitions $\WP$ and $\WQ$ have disjoint sides  $P^\times$ and $Q^\times$ that are also disjoint from $\lk(\WP)\cup \lk(\WQ)$. }\label{fig:nonadjacent}
\end{center}
\end{figure}

\subsection{Blowups}
In this section  we fix a collection $\bPi=\{\WP_1,\ldots,\WP_k\}$   of pairwise-compatible \GW\ partitions and construct a locally CAT(0) cube complex $\SP$  with fundamental group $\AG$, whose edges are labeled either by a partition in $\bPi$ or by a vertex of $\G$.

\begin{definition}  A choice of sides for a set of  \GW\ partitions is {\em consistent} if each pair is consistent.    A consistent choice of sides $P_i^\times$ for all $\WP_i\in \bPi$ is a {\em region}.
\end{definition}

\begin{lemma} Any consistent choice of sides for a subset of $\bPi$  can be extended to a region.
\end{lemma}
\begin{proof} This is Lemma 3.9 in \cite{CSV}.  It follows easily by induction on $k$, the number of partitions.
\end{proof}

 Regions will form the vertices of our cube complex.  To describe the higher-dimensional cubes, it is convenient to define a graph $\Gamma_\bPi$ that realizes our notion of  ``adjacency" for partitions in $\bPi$:
\begin{definition}\label{def:GammaPi}
 Let $\Pi$ be a collection of pairwise-compatible \GW-partitions.  Then $\Gamma_\Pi$ is the (simplicial) graph with
\begin{itemize}
\item one vertex for each element of $V\cup \Pi$
\item an edge between $A$ and $B$ whenever  $A$ and $B$ are   adjacent according to Definition~\ref{def:adjacent}., i.e. $\max(A)\subset\lk(B)$.
\end{itemize}
The link of a vertex   $A\in\Gamma_\bPi$ will be denoted $\lkp(A),$ the star by $\stp(A)$  and the double link by $\dlkp(A)$.   
\end{definition}

Every   $v\in V\cup V^{-1}$ is in $P_i^+, P_i^-$ or $\lk(\WP_i)$ for each $i$.  If $v\not\in\lk(\WP_i)$ define  the {\em $v$-side} of $\WP_i$ to be the side containing $v$.  Then the set of $v$-sides for those $\WP_i$ that  are not adjacent to 
$v$ form a consistent set and can be extended to a region. Any such region is called a {\em terminal region for  $v$}.

\begin{definition}
The {\em blowup} $\SP$ is a cube complex with one vertex for each region $\re=\{P_1^\times,\ldots,P_k^\times\}$.  The edges of $\SP$ are constructed as follows: 
\begin{itemize}
\item If two regions  differ only by changing the side of $\WP_i$ we connect them by an (unoriented) edge labeled $\WP_i$. 
 \item
 If $\re$ is a terminal region for $v$,  then   the region $\re^{*v}$ obtained by switching sides of all $\WP_i$ that split $v$ is a terminal region for $v^{-1},$ and we connect the two   by an oriented edge  labeled  $v$ that goes from $\re^{*v}$ to $\re$. 
\end{itemize}
  Higher-dimensional cubes are attached whenever   a set of edges forms the 1-skeleton of a cube  whose labels span a clique in $\Gamma_\bPi$.
 \end{definition}
From the definition, we see immediately that:
\begin{itemize}
\item There is an edge labeled $v$ terminating at the vertex $\re=\{P_1^\times,\ldots, P_k^\times\}$ if and only if for each $i$, either $v\in P_i^\times$ or $v\in\lk(\WP_i)$.    If no $\WP_i$ splits $v$, then an edge labeled $v$ in $\SP$ is a loop at $\re$.
\item There is an edge labeled $\WP_j$ with one endpoint at  $\re=\{P_1^\times,\ldots, P_k^\times\}$ if and only if for each $i\neq j$, either  $\WP_i$ and $\WP_j$ are adjacent  or some side of $\WP_j$ is contained in $P_i^\times$.  In particular, if $\WP_i$ and $\WP_j$   are not adjacent 
  then both $P_j^\times\cap P_i^\times$ and $\overline P_j^\times\cap P_i^\times$ are non-empty.  An edge labeled $\WP_i$ is never a loop.
  \end{itemize}

 In Figures~\ref{blowupone}--\ref{blowupthree} we show three blowups of $\SaG$ for the graph $\G$ shown in Figure~\ref{fig:salvetti}.  As before,  edges with the same color in the right-hand diagram are identified.   In Figures~\ref{blowupone} and \ref{blowuptwo}  the blowups are two tori identified along a circle, with an extra edge attached. In Figure~\ref{blowupthree} the blowup is  two tori identified along a cylinder, with an extra edge attached. The structure of blowups will be explored in much more detail in Section~\ref{sec:blowups}.

\begin{figure}.  
\begin{center}
 \begin{tikzpicture}
 \begin{scope}[decoration={markings,mark = at position 0.5 with {\arrow{stealth}}}]
  \coordinate (c) at (-1,1); \coordinate (a) at (0,1); \coordinate (d) at (1,1);\coordinate (b) at (2,1);
 \vertex{(a)};\vertex{(b)};\vertex{(c)};\vertex{(d)};
  \coordinate (ci) at (-1,-0); \coordinate (ai) at (0,0); \coordinate (di) at (1,0);\coordinate (bi) at (2,0);
 \vertex{(ai)};\vertex{(bi)};\vertex{(ci)};\vertex{(di)};
 \node [above](A) at (a) {$a$};
  \node [above](B) at (b) {$b$};
    \node [above](C) at (c) {$c$};
      \node [above](D) at (d) {$d$};
\node [below](AI) at (ai) {$a^{-1}$};
  \node [below](BI) at (bi) {$b^{-1}$};
    \node [below](CI) at (ci) {$c^{-1}$};
      \node [below](DI) at (di) {$d^{-1}$};
      \draw [blue] (.5,1.1) ellipse (1cm and .5cm); \node (Q) at (.5,1.4) {$\mathcal Q$};

  \begin{scope}[xshift=10cm, scale=.9]
 \midarrow [thick, blue] (1,-2) to (-1,-2); \node [above] (b1) at (0,-2) {$b$};
  \midarrow [thick, blue] (1,0) to (-1,0); \node [above] (b2) at (0,0) {$b$};
   \midarrow [thick, blue] (1,2)  to (-1,2); \node [above] (b3) at (0,2) {$b$};
   \midarrow [thick] (1,-1) to (-1,-1);  \node [above] (b4) at (0,-1) {$b$};
         \midarrow [thick, red] (1,-2) to (1,-1);  \node [left] (abot) at (1,-1.5) {$a$};
           \draw [thick, red] (1,-1) to (1,0);  \node [left] (abot) at (1,-.5) {$\WQ$};
         \midarrow [thick, red] (-1,-2) to (-1,-1);  \node [left] (abot) at (-1,-1.5) {$a$};

           \draw [thick, red] (-1,-1) to (-1,0);  \node [left] (abot) at (-1,-.5) {$\WQ$};
       \midarrow[thick, forestgreen] (1,0) to (1,2);\node [left] (ctop) at (1,1) {$c$};
 \midarrow [thick, forestgreen] (-1,0) to (-1,2);     \node [left] (cbot) at (-1,1) {$c$};
 \midarrow  [thick] (1,0) .. controls (2.5,0) and (2.5,-1) ..   (1,-1);
    \node (dtop) at (1.75,.25) {$d$};
  
 \end{scope}
  \node (e) at (8.25,0) {$=$};
 \end{scope}
 \begin{scope}[xshift=4.5cm, scale=.33]]
 \draw (0,0) ellipse (4cm and 3cm); 
\draw (0,-2) [partial ellipse=55:125:2.5cm and 2.5cm];
\draw (0,2) [partial ellipse=225:315:2.5cm and 2.5cm];
\begin{scope}[xshift=5.5cm]. 
\draw (0,0) ellipse (4cm and 3cm);
\draw (0,-2) [partial ellipse=55:125:2.5cm and 2.5cm];
\draw (0,2) [partial ellipse=225:315:2.5cm and 2.5cm];
\end{scope}
\begin{scope}[decoration={markings,mark = at position 0.5 with {\arrow{stealth}}}]
 \draw [thick, blue] (2.75,0) [partial ellipse=0:180:1.25cm and .5cm];   
  \draw [thick, blue, densely dotted] (2.75,0) [partial ellipse=180:360:1.25cm and .5cm];
  \node [above] (b) at (1.75,.25) {$b$};
  \midarrow [blue, thick] (2.3,.45) to (2.35,.45);
 \draw [thick] (8.25,0) [partial ellipse=0:180:1.25cm and .5cm];   
  \midarrow [thick] (7.85,.45) to (7.86,.45);
   \node [above] (b) at (7.25,.25) {$b$}; 
  \draw [thick, densely dotted] (8.25,0) [partial ellipse=180:360:1.25cm and .5cm];
\draw [thick, forestgreen] (0,.5) ellipse (2.75cm and 1.75cm); 
 \midarrow [forestgreen, thick] (-2.75,.5) to (-2.75,.6);
    \node [below] (c) at (-3.4,1) {$c$};  
\draw [thick, red] (5.5,.5) ellipse (2.75cm and 1.75cm); 
\node [below] (a) at (5.5,-1.25) {$a$};
 \midarrow [red, thick] (5.5,-1.25) to (5.505,-1.25);

\draw [thick] (2.75,.5) .. controls (2.75,6) and (8.25,6) .. (8.25,.5);  
     \midarrow [thick] (5.45,4.625) to (5.55,4.625);
\node [above] (d) at (4.9,4.6) {$d$};
     \fill (2.75,.5) circle (.15);
        \fill (8.25,.5) circle (.15);
 \end{scope}
 \node [below] (Q) at (5.5,2.5) {$\WQ$};
\end{scope}
\end{tikzpicture}
\caption{The blowup $\Sa^\WQ$ for $\WQ=(\{a,d\}|\{a\inv,d\inv,c,c\inv \}|\{b,b\inv\})$}\label{blowupone}
\end{center}
\end{figure}

\begin{figure}. 
\begin{center}
  \begin{tikzpicture}
  \begin{scope}[decoration={markings,mark = at position 0.5 with {\arrow{stealth}}}]
  \coordinate (c) at (-1,1); \coordinate (a) at (0,1); \coordinate (d) at (1,1);\coordinate (b) at (2,1);
 \vertex{(a)};\vertex{(b)};\vertex{(c)};\vertex{(d)};
  \coordinate (ci) at (-1,-0); \coordinate (ai) at (0,0); \coordinate (di) at (1,0);\coordinate (bi) at (2,0);
 \vertex{(ai)};\vertex{(bi)};\vertex{(ci)};\vertex{(di)};
 \node [above](A) at (a) {$a$};
  \node [above](B) at (b) {$b$};
    \node [above](C) at (c) {$c$};
      \node [above](D) at (d) {$d$};
\node [below](AI) at (ai) {$a^{-1}$};
  \node [below](BI) at (bi) {$b^{-1}$};
    \node [below](CI) at (ci) {$c^{-1}$};
      \node [below](DI) at (di) {$d^{-1}$};
\draw [blue] (1.5,1.1) ellipse (1cm and .5cm); \node (P) at (1.6,1.4) {$\WP$};
\draw [blue] (.5,1.1) ellipse (1cm and .5cm);\node (Q) at (0.4,1.4) {$\WQ$};
  \begin{scope}[xshift=10.5cm, scale=.9]
 \midarrow [thick, blue] (0,-2) to (-1,-2); \node [above] (b1) at (-.5,-2) {$b$}; 
  \draw [thick, blue] (0,-2) to (1,-2); \node [above] (b1) at (.5,-2) {$\WP$};  
  \midarrow [thick, blue] (0,0) to (-1,0); \node [above] (b2) at (-.5,0) {$b$};  
   \draw [thick, blue] (0,0) to (1,0); \node [above] (b1) at (.5,0) {$\WP$};  
   \midarrow [thick, blue] (0,2)  to (-1,2); \node [above] (b3) at (-.5,2) {$b$}; 
    \draw [thick, blue] (0,2) to (1,2); \node [above] (b1) at (.5,2) {$\WP$}; 
   \midarrow [thick] (0,-1) to (-1,-1);  \node [above] (b4) at (-.5,-1) {$b$};  
    \draw [thick] (0,-1) to (1,-1);  \node [above] (b4) at (.5,-1) {$\WP$};  
    \midarrow [thick] (0,0) to (0,2); \node [left] (c4) at (0,1) {$c$};  
      \midarrow [thick, red] (1,-2) to (1,-1);  \node [right] (abot) at (1,-1.5) {$a$};  
 \draw [thick, red] (-1,-1) to (-1,0);  \node [left] (abot) at (-1,-.5) {$\WQ$}; 
 \draw [thick] (0,-1) to (0,0); \node [left] (Qmid) at (0,-.5) {$\WQ$};  
\draw [thick, red] (1,-1) to (1,0);  \node [right] (abot) at (1,-.5) {$\WQ$};  
 \midarrow [thick, red] (-1,-2) to (-1,-1);  \node [left] (abot) at (-1,-1.5) {$a$}; 
 \midarrow [thick] (0,-2) to (0,-1); \node [left] (amid) at (0,-1.5) {$a$};  
       \midarrow[thick, forestgreen] (1,0) to (1,2);\node [right] (ctop) at (1,1) {$c$};  
 \midarrow [thick, forestgreen] (-1,0) to (-1,2);     \node [left] (cbot) at (-1,1) {$c$};  
  \midarrow  [thick] (1,0) .. controls (3,1) and (.25,3)   ..   (0,-1);  
    \node (dtop) at (.5,1.5) {$d$};
 \end{scope}
  \node (e) at (8.5,0) {$=$};
 \end{scope}

 \begin{scope}[xshift=4.5cm, scale=.33]] 
 \draw (0,0) ellipse (4cm and 3cm); 
\draw (0,-2) [partial ellipse=55:125:2.5cm and 2.5cm];
\draw (0,2) [partial ellipse=225:315:2.5cm and 2.5cm];
\begin{scope}[xshift=5.5cm]. 
\draw (0,0) ellipse (4cm and 3cm);
\draw (0,-2) [partial ellipse=55:125:2.5cm and 2.5cm];
\draw (0,2) [partial ellipse=225:315:2.5cm and 2.5cm];
\end{scope}
\begin{scope}[decoration={markings,mark = at position 0.5 with {\arrow{stealth}}}]
 \draw [thick, blue] (2.75,0) [partial ellipse=0:180:1.25cm and .5cm]; 
  \draw [thick, blue, densely dotted] (2.75,0) [partial ellipse=180:360:1.25cm and .5cm];
  \node (b) at (1,.25) {$b$}; 
  \midarrow [blue, thick] (2.3,.45) to (2.35,.45);
 \draw [thick] (8.25,0) [partial ellipse=0:180:1.25cm and .5cm]; 
  \draw [thick, densely dotted] (8.25,0) [partial ellipse=180:360:1.25cm and .5cm];
  \node (b) at (6.5,.25) {$b$}; 
   \midarrow [thick] (7.85,.45) to (7.86,.45);
\draw [thick, forestgreen] (0,.5) ellipse (2.75cm and 1.75cm); 
 \node [above] (c) at (-3.25,0) {$c$};  
   \midarrow [forestgreen, thick] (-2.75,.5) to (-2.75,.6);
\draw [thick, dotted] (0,-.5) ellipse (2.75cm and 1.75cm); 
   \node [below] (c) at (-3.25,-.25) {$c$};  
  \midarrow [thick] (-2.75,-.6) to (-2.75,-.5);
\draw [thick, red] (5.5,.5) ellipse (2.75cm and 1.75cm); 
\draw [thick, dotted] (5.5,-.5) ellipse (2.75cm and 1.75cm); 
\midarrow [thick] (2.75,.5) .. controls (2.75,7) and (9,5) .. (9,1.5); 
\draw  [thick, densely dotted ](9,1.5) .. controls (9,.75) and (8.75,0) ..  (8.25,-.5);
\node (d) at (5,5.5) {$d$};  
\node [below] (a) at (5.5,-1.05) {$a$}; 
  \midarrow [red, thick] (5.5,-1.25) to (5.505,-1.25);
 \node [below] (a) at (5.5,-2) {$a$};
 \midarrow [ thick] (5.5,-2.25) to (5.505,-2.25);

 \end{scope}
 \node [above] (Q) at (5.5,2) {$\WQ$}; 
  \node (Q) at (5.5,1.25) {$\WQ$}; 
 \node (P) at (4.5,.25) {$\WP$};
 \node (P) at (10,.25) {$\WP$};
  \fill (2.75,.5) circle (.15);
        \fill (8.25,.5) circle (.15);
\fill (2.75,-.5) circle (.15);
        \fill (8.25,-.5) circle (.15);
 \end{scope}
\end{tikzpicture}
\caption{The blowup $\SP$ for $\bPi=\{\WP,\WQ\},$ $\WP=(\{b,d\}|\{b\inv,d\inv\}|\{a,a\inv, c,c\inv\})$}\label{blowuptwo}
\end{center}
\end{figure}

\begin{figure}
\begin{center}
\begin{tikzpicture}
  \begin{scope}[decoration={markings,mark = at position 0.5 with {\arrow{stealth}}}]
  \coordinate (c) at (-1,1); \coordinate (a) at (0,1); \coordinate (d) at (1,1);\coordinate (b) at (2,1);
 \vertex{(a)};\vertex{(b)};\vertex{(c)};\vertex{(d)};
  \coordinate (ci) at (-1,-0); \coordinate (ai) at (0,0); \coordinate (di) at (1,0);\coordinate (bi) at (2,0);
 \vertex{(ai)};\vertex{(bi)};\vertex{(ci)};\vertex{(di)};
 \node [above](A) at (a) {$a$};
  \node [above](B) at (b) {$b$};
    \node [above](C) at (c) {$c$};
      \node [above](D) at (d) {$d$};
\node [below](AI) at (ai) {$a^{-1}$};
  \node [below](BI) at (bi) {$b^{-1}$};
    \node [below](CI) at (ci) {$c^{-1}$};
      \node [below](DI) at (di) {$d^{-1}$};
\draw [blue] (1.5,1.1) ellipse (1cm and .5cm); \node (P) at (1.6,1.4) {$\WP$};
\draw [blue] (0,1.1) ellipse (1.5cm and .6cm); \node (R) at (-.4,1.4) {$\WR$};
   \begin{scope}[xshift=10cm, scale=.9] %
 \midarrow [thick, blue] (-1,2)  to (0,2); \node [above] (b3) at (-.5,2) {$b$}; 
    \draw [thick, blue] (0,2) to (1,2); \node [above] (b1) at (.5,2) {$\WP$}; 
  \midarrow [thick] (-1,.67) to (0,.67); \node [above] (b2) at (-.5,.67) {$b$};  
   \draw [thick] (0,.67) to (1,.67); \node [above] (b1) at (.5,.67) {$\WP$};  
   \midarrow [thick, blue] (-1,-.67) to (0,-.67);  \node [above] (b4) at (-.5,-.67) {$b$};  
    \draw [thick, blue] (0,-.67) to (1,-.67);  \node [above] (b4) at (.5,-.67) {$\WP$};  
        \midarrow [thick] (-1,-2) to (0,-2); \node [above] (b1) at (-.5,-2) {$b$}; 
  \draw [thick] (0,-2) to (1,-2); \node [above] (b1) at (.5,-2) {$\WP$};  
    \midarrow [thick] (0,.67) to (0,2); \node [left] (c4) at (0,1.45) {$c$};  
 \draw [thick, red] (-1,-1) to (-1,.67);  \node [left] (abot) at (-1,.15) {$\WR$}; 
 \draw [thick] (0,-1) to (0,.67); \node [left] (Rmid) at (0,.15) {$\WR$};  
\draw [thick, red] (1,-1.67) to (1,.67);  \node [left] (abot) at (1,.15) {$\WR$};  
 \midarrow [thick, red] (-1,-2) to (-1,-.67);  \node [left] (abot) at (-1,-1.2) {$a$}; 
 \midarrow [thick]         (0,-2) to (0,-.67); \node [left] (amid) at (0,-1.2) {$a$};  
 \midarrow [thick, red] (1,-2) to (1,-.67);  \node [left] (abot) at (1,-1.2) {$a$};  
       \midarrow[thick, forestgreen] (1,.67) to (1,2);\node [left] (ctop) at (1,1.45) {$c$};  
 \midarrow [thick, forestgreen] (-1,.67) to (-1,2);     \node [left] (cbot) at (-1,1.45) {$c$};  
  \midarrow  [thick] (1,.67)   .. controls (2.7,.3)  and (1.8,-.4) ..   (0,-.67);  
    \node (dtop) at (2,-.25) {$d$};
 \end{scope}
  \node (e) at (8,0) {$=$};
 \begin{scope}[xshift=5cm, scale=.33]] 
\draw (0,0) [partial ellipse=25:335:3.5cm and 4cm]; 
\draw (4,0) [partial ellipse=-155:155:3.5cm and 4cm]; 
\draw (-.2,0) [partial ellipse=45:315:1.4cm and 2cm]; 
\draw (4.2,0) [partial ellipse=-135:135:1.4cm and 2cm]; 
 \draw [thick, densely dotted] (2,1.5) [partial ellipse=0:180:1.25cm and .5cm]; 
  \draw [thick] (2,1.5) [partial ellipse=180:360:1.25cm and .5cm];
 \draw [thick, blue, densely dotted] (2,-1.5) [partial ellipse=0:180:1.25cm and .5cm]; 
  \draw [thick, blue] (2,-1.5) [partial ellipse=180:360:1.25cm and .5cm];
\draw (3.25,1.5) to (3.25,-1.5);  
\draw (.75,1.5) to (.75, -1.5);   
\begin{scope}[xshift=.5cm, yshift=1cm]
\draw [densely dotted] (2,1) to (2,-2);  
\end{scope}
\draw[thick,red] (2,1) to (2,-2); 
\midarrow [thick, forestgreen] (2,1) .. controls (-3,7) and (-5,-6) .. (2,-2); 
\midarrow [thick, red] (2,1) .. controls (7,7) and (8.5,-6.5) .. (2, -2);  
\begin{scope} [xshift=15, yshift=30]
\midarrow [densely dotted] (2,1) .. controls (-3,6) and (-5,-8) .. (2,-2); 
\midarrow [densely dotted] (2,1) .. controls (7.5,6) and (8,-9) .. (2, -2);  
\end{scope}
 \fill (2,1) circle (.15);
        \fill (2,-2) circle (.15);
        \fill (2.5,2) circle (.15);
        \fill (2.5,-1) circle (.15);
\midarrow [thick] (2,1) .. controls (6,1) and (6,-1) .. (3.25,-1);
\draw [thick, densely dotted] (3.25,-1) to (2.5,-1);
\end{scope}
\end{scope}
\end{tikzpicture}
\caption{The blowup $\SP$ for $\bPi=\{\WP,\WR\}$,  $\WR=(\{a,c,d\}|\{a\inv,d\inv,c\inv\}|\{b,b\inv\})$}\label{blowupthree}
\end{center}
\end{figure}

\subsection{Collapsing hyperplanes}

\begin{definition} Let $H$ be a hyperplane in a  cube complex $X$.  The closure of the set of cubes that intersect $H$ is called the {\em hyperplane carrier $\kappa(H)$}, and the {\em hyperplane collapse} associated to $H$ is the map $c_H$ on $X$ that collapses   $\kappa(H)$ to  $H$.
\end{definition}

Recall from \cite{CSV} that hyperplanes in $\SP$ are characterized by the fact that the set of edges they intersect is exactly the set of edges with a given label $A\in V\cup\bPi$.  We say the hyperplane is {\em labeled} by $A$. 

\begin{proposition}[\cite{CSV} Theorem  4.6]  If $\WP\in \bPi$,  and $H_{\WP}$ is the hyperplane in $\SP$ labeled by $\WP$, then  the image of $\SP$ under $c_{H_\WP}$ is isomorphic to $\Sa^{\bPi-\WP}$.
\end{proposition}
The {\em standard collapse} $c_\pi\colon \Sa^\bPi\to \Sa^\emptyset =
\SaG$ is the map that collapses all hyperplanes whose  labels are in $\bPi$.

\subsection{Untwisted Outer space $\SG$}\label{subsec:untwisted}

Recall that the {\em untwisted subgroup} $U(\AG)\leq \Out(\AG)$ is the subgroup generated by  \GW\ automorphisms, graph automorphisms and inversions.
By work of  Laurence and Servatius \cite{Lau, Ser}, $U(\AG)$ together with automorphisms $v\mapsto vw$ for  $v\leq_t w$ (called {\em  twists}) generate the full group $\Out(\AG)$. In this section we recall the main theorem of \cite{CSV} and use it to define a contractible space $\SG$ on which $U(\AG)$ acts properly.  We first recall the space $K_\G$ studied in \cite{CSV}.

\begin{definition} A cube complex $X$ is a {\em $\G$-complex}    if it is isomorphic  to a blowup $\SP$ for some $\bPi$.  A {\em $\G$-complex collapse} $c\colon X\to\SaG$ is the composition of an isomorphism $X\iso\SP$ with the standard collapse  $\SP\to \SaG$.

\begin{example} If $\G$ has no edges, then a $\G$-complex is a connected graph with no univalent or bivalent vertices and no separating edges, and a $\G$-complex collapse contracts a maximal tree to a point.
\end{example}

A {\em marked $\G$-complex} is an equivalence class of pairs $(X,g)$ where
\begin{itemize}
\item   $X$ is a $\G$-complex
\item  $g\colon X\to \SaG$ is a  homotopy equivalence\item $(X',g')\sim ({X},g)$ if there is a cube complex isomorphism $i\colon X'\to{X}$ with $g\circ i\simeq g'$.
\end{itemize}
\end{definition}
A marking $h\colon X\to\SaG$  is {\em untwisted} if  the composition of a homotopy inverse $h^{-1}$ with some (and hence any)  $\G$-complex collapse  induces an element of the untwisted subgroup $U(\AG)$.

If a hyperplane collapse $c_H\colon X'\to {X}$ is a homotopy equivalence, we set $$ (X', h\circ c_H)>(X,h).$$   This induces a partial order on $\G$-complexes with untwisted markings.  The {\em  spine $K_\G$} is the geometric realization of the resulting poset, i.e. it is a simplicial complex, where a $k$-simplex is a  $\G$-complex  with an untwisted marking together with a chain of $k$ hyperplane collapses, each of which is a homotopy equivalence to another $\G$-complex with an untwisted marking.

\begin{theorem}[\cite{CSV}]The  spine $K_\G$ is contractible, and   $U(\AG)$ acts properly and cocompactly on $K_\G$.
\end{theorem}

We now define the space $\SG$ by viewing the cubes of a $\G$-complex $X$ as metric objects, each isometric to an orthogonal product of intervals of various lengths, i.e. an {\em orthotope}.   The result is a locally CAT(0) complex  $X$ which we will call a {\em rectilinear} $\G$-complex.  All edges dual to the same hyperplane in $X$  have the same length, called the {\em width} of the hyperplane.  A point in $\SG$ is then a marked  rectilinear $\G$-complex   $(X, h)$, where $h$ is untwisted and the cube complex isomorphism in the definition of  the equivalence relation must be an  isometry on each orthotope.  In the case $\G$ has no edges, the spine $K_\G$ is the same as the spine of (reduced) Outer space, as originally defined in \cite{CuVo}, and $\SG$ is reduced Outer space itself.

The spine $K_\G$ embeds in $\SG$ as follows: the image of a vertex $[(X,h)]$ of $K_\G$ is determined by the property that all edges of $X$ have length $1/n,$ where $n$ is the number of hyperplanes in $X$.   The image of each higher-dimensional simplex is the linear span of its vertices.

\begin{proposition}  $K_\G$ is a deformation retract of $\SG$.
\end{proposition}

\begin{proof}
$\SG$ contains the set $P\SG$ of marked metric $\G$-complexes $[(X,h)]$ for which the sum of the hyperplane widths in $X$ equals $1$.  Note that the image of our embedding of $K_\G$ into $\SG$ is contained in $P\SG$.  The map  $\SG\to P\SG$ that scales all edge lengths simultaneously is a deformation retraction.

The subspace $P\SG$ decomposes into a union of open simplices, one for  each marked $\G$-complex $[(X,h)]$, of dimension one less than  the number of hyperplanes in $X$.  The points in this simplex are obtained by varying the widths of the hyperplanes while keeping the sum equal to $1$.    For each such simplex, consider the barycentric subdivision of its closure, and let $K[(X,h)]$ be the subcomplex of this barycentric subdivision spanned by the barycenters of faces that are actually contained in $P\SG$.  It is easy to see that $K[(X,h)]$ is equal to the image of $K_\G$ under the embedding described above, and is a deformation retract of $\SG$.
\end{proof}

\begin{corollary}\label{cor:SGcontractible} The space $\SG$ is contractible.
\end{corollary}

\section{Combinatorial and metric structure of blowups}\label{sec:blowups}
 Throughout this section we  fix a compatible set $\Pi$ of \GW-partitions. To prove our main theorem we will have to understand the structure of  the blowup $\SP$ in some detail.  We gather some facts about blowups here.

\subsection{Basics}
The following basic features of  blowups $\SP$ are either part of Theorem 3.14 of \cite{CSV} or follow immediately from the definition of $\SP$.
\begin{enumerate}
\item\label{fact:NPC} $\SP$ is a locally CAT(0) cube complex, i.e. the path metric induced by making each $k$-cube isometric to $[0,1]^k$ is locally CAT(0).
\item\label{fact:treelike} The subcomplex $\EP\subset \SP$ consisting of cubes all of whose edge labels are in $\bPi$ is CAT(0) and locally convex, and it contains all vertices of $\SP$.
\item\label{fact:collapse} The standard collapse map $c_\pi$ maps all of $\EP$ to the single vertex in $\SaG$.
 \item\label{fact:hyplabel} The set of edges of $\SP$ with a given label $A\in V\cup\bPi$ is the set of edges that intersect a single hyperplane, which we will call $H_A$.  All hyperplanes in $\SP$ are of this form.
 \item\label{fact:hypedges} Each hyperplane $H_A$ inherits a cube complex structure from $\SP$ whose edges are labeled by the   elements of $V\cup \bPi$ that are adjacent 
 to $A$, i.e. by elements in $\lk_\bPi(A)$.
\item\label{fact:onelabel} There is at most one edge with a given label at any vertex of $\SP$.
\end{enumerate}

Another way to define the subcomplex $\EP$ is to observe that  the set of sides of the partitions in $\bPi$ form a {\em pocset}, that is, a partially ordered set with an order reversing involution $P \mapsto \overline{P}$ such that  pairs $P, \overline{P}$ are unrelated; this follows from  Lemma~\ref{lem:noncommuting}.  Any pocset  satisfying suitable finiteness conditions  gives rise to a CAT(0) cube complex  (see, e.g., \cite{Sageev}), and $\EP$ is isomorphic to the cube complex associated to the pocset of sides of $\bPi$.

 \subsection{Adjacent  
 labels}\label{subsec:linked}

 In this section  show that there is a unique cube in $\SP$ for every maximal clique in the graph $\Gamma_\bPi,$ i.e. any maximal set of pairwise adjacent elements of $V\cup \bPi$.   

We begin with existence, for which the following definition is useful.

\begin{definition} Let $\WP\in\bPi$.  For  $v\in V\cup V^{-1} \setminus \lk(\WP)$ the {\em $v$-side of $\WP$} is the side containing $v$. For $\WQ\in \bPi\setminus \{\WP\}$  not adjacent to $\WP$ 
the {\em $\WQ$-side of $\WP$} is the side containing  some side of $\WQ$ (There is a unique such side by Lemma~\ref{lem:noncommuting}). 
\end{definition}

Stated in terms of hyperplanes, $H_\WP$ splits the subspace $\EP$ into two components. If  $v\not\in \lk(\WP)$ 
the $v$-side of $H_\WP$ is the side containing the terminal vertex of some (hence every)  edge labeled $v$.  If $\WQ$ and $\WP$  are distinct and not adjacent, then $H_\WQ$ does not intersect $H_\WP$ and the $\WQ$-side of $H_\WP$ is the side containing $H_\WQ$.

  \begin{proposition}\label{prop:cube} Let $\bPi$ be a compatible set of  $k$ \GW-partitions, and let $\mathcal A=\{A_1,\ldots, A_\ell\}$ be the vertices of a maximal clique in $\Gamma_\Pi$.  Then there is a cube in $\SP$ with edge labels $\{A_1,\ldots,A_\ell\}$.
\end{proposition}

\begin{proof}
 Let   $\mathcal A\cap V=\{v_1,\ldots v_r\}$ and $\mathcal A\cap\Pi=\{\WQ_1,\ldots,\WQ_s\}$, so
\begin{itemize}
\item $\Pi=\{\WQ_1,\ldots,\WQ_s,\WP_1,\ldots,\WP_t\}$ with $s+t=k$ and
\item $\mathcal A=\{v_1,\ldots, v_r,\WQ_1,\ldots,\WQ_s\}$ with $r+s=\ell$
\end{itemize}
For any choices of sides $Q_i^\times$ of $\WQ_i$  for  $i=1,\ldots,s$ and exponents $v_j^\times=v_j$ or $v_j^{-1}$  for  $j= 1,\ldots,r$, we will find a region ${\re}$  which is a terminal region for each $\WQ_i$ and $v_j^\times$.  These $2^\ell$ regions (some of which may coincide, as we will see) form the vertices of an $\ell$-dimensional cube in $\SP$ with edges labeled by the elements of $\mathcal A$.

 To define the region associated to $\{v_1^\times,\ldots,v_r^\times,Q_1^\times,\ldots,Q_s^\times\}$ we will start with the sides $Q_i^\times$.  We then need to choose a side of each $\WP_i$.
Since $\mathcal A$ is a maximal clique, for each $\WP_i$ there is some $A_j$ not adjacent to $\WP_i$.  If $A_j$ is a vertex $v_j$, let $P_i^\times$ be the $v_j^\times$-side of $\WP_i$, and if $A_j$ is a \GW-partition $\WQ_j$, let $P_i^\times$ be the $\WQ_j$-side of $\WP_i$.
To see that $P_i^\times$ does not depend on the choice of $A_j$, observe that if $\WP_i$ is not adjacent to either $A_j$ or $A_k$  
then the fact that $w_j\in\max(A_j)$ and $w_k\in\max(A_k)$ are joined by an edge in $\G$ implies that all of $\{w_j,w_j\inv,w_k,w_k\inv\}$ are on the same side of $\WP_i$, so the $A_j$-side of $\WP_i$ is the same as the $A_k$-side of $\WP_i$.

 Now let $\re=\{Q_1^\times,\ldots,Q_s^\times,P_1^\times,\ldots,P_t^\times\}$.  To see that this is a region, we must show that any two elements either belong to adjacent partitions or intersect non-trivially.
Each pair $\WQ_i,\WQ_j$ is adjacent.  If $\WQ_j$ is not adjacent to $\WP_i$ we have chosen the $\WQ_j$-side $P_i^\times$ of $\WP_i$.  Since $P_i^\times$ contains an entire side of $\WQ_j$, it intersects both sides of $\WQ_j$ non-trivially.  If $\WP_i$ and $\WP_j$ are not adjacent, let $A_k$ be an element of $\mathcal A$ that is not adjacent to $\WP_i$. We argue by contradiction:  suppose  that $P_i^\times\cap P_j^\times=\emptyset$. If $A_k$ is a vertex $v_k$ then $v_k^\times\in P_i^\times$, so $v_k^\times\not\in P_j^\times$.  Since $v_k^\times\not\in\lk(\WP_j)$   by Lemma~\ref{lem:noncommuting}, this contradicts our choice of $P_j^\times$.
 If $A_k$ is a partition $\WQ_k$ and $Q_k^\times\subset P_i^\times$ then  $\max(\WQ_k)\not\subset \lk(\WP_j)$ and neither  side of $\WQ_k$ 
  is contained in $P_j^\times$, again contradicting our choice of $P_j^\times$.

 The region $\re$ is a terminal region for each $v_i^\times$.  If we use   $(v_i^\times)^{-1}$ instead of $v_i^\times$ we get another region, terminal for $(v_i^\times)^{-1}$.   These two regions may be the same if $v_i$ and $v_i^{-1}$ are on the same side of each $\WP_j$, in which case the edge labeled $v_i$ is a loop.   Switching sides of any $\WQ_i$ gives another region, with an edge labeled $\WQ_i$ joining the two (this edge is never a loop). Thus we have the 1-skeleton of an $\ell$-dimensional cube in $\SP$, which is filled in since all of the edge-labels are adjacent.  
\end{proof}
 
\begin{corollary} \label{fact:commhyp} Two hyperplanes $H_A$ and $H_B$ intersect if and only if $A$ and $B$  are adjacent. .
\end{corollary}

\begin{proof}
If $A$ and $B$  are adjacent 
it follows from Proposition~\ref{prop:cube} that  there is a square with sides labeled $A$ and $B$, so the hyperplanes $H_A$ and $H_B$ intersect at the midpoint of that square.  Conversely, if $H_A$ and $H_B$ intersect, there is a pair of edges dual to these hyperplanes that bound a square, so  $A$ and $B$ must  be adjacent 
since by the construction of $\SP$, we only fill in squares when labels  are adjacent. 
\end{proof}

 \begin{remark} 
 Corollary~\ref{fact:commhyp}  says that $\Gamma_\Pi$ is the {\em crossing graph} for $\SP$ as defined in \cite{Sageev}.
 \end{remark}

 \begin{proposition}\label{fact:parallelcubes} Any cubes $c,c'$ in $\SP$ with the same edge labels are parallel,  i.e. $S^\Pi$ contains a subcomplex isomorphic to $c\times [0,n]$ for some $n\in\Z$, with $c=c\times \{0\}$ and $c'=c\times \{n\}$.
\end{proposition}
 \begin{proof}
 If $c$ and $c'$ share a vertex then they must be equal, so  we may assume that they are disjoint.  Recall that $\EP$ is CAT(0), hence connected, and contains every vertex of $\SP$.   Let $p$ be a minimal length edge-path from $c$ to $c'$ that is contained in $\EP$. The CAT(0) property implies that $p$ crosses each hyperplane at most once. The first edge of $p$ is labeled by some partition $\WP$. Since $p$ has minimal length, $\WP$ is distinct from all of the edge labels of $c$.   Let $\re=\{P^\times,\ldots\}$ be the initial vertex of $p$, where $P^\times$ is a side of $\WP$,  and let $\re'$ be the terminal vertex.

Suppose now that  some edge label $B$  of $c$ is not adjacent to $\WP$.   If $B=\{v\}$ then $P^\times$ contains $v$, so both $\re$ and $\re'$ use this side.  The first edge of the path $p$ switches sides of $\WP$, i.e. crosses the hyperplane $H_\WP$, so in order to reach $\re'$ it must cross $H_\WP$ again, contradicting the assumption that it is the shortest path.  If $B=\WQ$ then  the side of $\WQ$ that appears in $\re$  is neither contained in $P^\times$ nor contains $P^\times$ (since there are edges labeled both $\WP$ and $\WQ$ at $\re$).  Since changing sides of $\WQ$ is allowed at $\re'$, it follows that the side of $\WP$ at $\re'$ must also be equal to $P^\times$.  As before, the initial edge of the path $p$ crosses $H_\WP$, so in order to reach $\re'$ it must cross $H_\WP$ again, contradicting the assumption that it is shortest.

We conclude that  B and $\WP$ are adjacent 
for all edge labels $B$ of $c$, so there is a cube $c\times e_B \subset \SP$.  The  side $c''$ of this cube  opposite from $c$ is closer to $c'$, and we can continue to build a product neighborhood $c\times [0,n]$ until we reach $c'$.
\end{proof}

\begin{corollary}\label{fact:maxcube}
 For every maximal collection $\{A_1,\ldots,A_k\}$ of pairwise  adjacent 
 labels, there is a unique maximal cube in $\SP$ with those edge-labels.
 \end{corollary}

 \begin{proof}
  Existence is Proposition~\ref{prop:cube}.  Uniqueness follows from Proposition~\ref{fact:parallelcubes}, since the existence of two distinct parallel cubes implies that $\{A_1,\ldots, A_k\}$ is not maximal.
\end{proof}

\subsection{Characteristic cycles}\label{sec:CharCycles}

  \begin{definition}\label{def:CharCycle} Let $v$ be a vertex of $\G,$  and $e_v$ an edge of $\SP$ labeled  by $v$.  Choose a minimal length edge-path $p(e_v)$ in $\EP$ from the terminal vertex $\tau(e_v)$ to the initial vertex $\iota(e_v)$.  The  loop $\chi_v = p(e_v) \cup e_v$   is called  a {\em characteristic cycle} for $v$.
 \end{definition}

\begin{figure}
\begin{center}
\begin{tikzpicture}
\fill [red!10] (3,0) ellipse (4.5cm and 1.25cm);
 \begin{scope}[decoration={markings,mark = at position 0.5 with {\arrow{stealth}}}]
 \draw [red, thick] (0,0) .. controls (0,2) and (5.5,2.5) .. (5.5,.5);
 \midarrow [red, thick] (2.75,1.75) to (2.70,1.75);
  \node [above] (ev) at (2.75,1.75) {$e_v$};
\end{scope}
\draw[red, thick] (0,0) to (1,0) to (1,-1) to (2,-1) to (2.5,-.5) to (3.5,-.5) to (3.5,.5) to (4.5,.5) to (5.5,.5);
\draw (1,0) to (1.5,.5) to (2.5,.5) to (2,0) to (1,0);
\draw (1,-1) to (1.5,-.5) to (2.5,-.5) to (2.5,.5);
\draw (1.5,-.5) to (1.5,.5); \draw (2,-1) to (2,0);
\draw (3.5,-.5) to (5.5,-.5) to (5.5,.5); \draw (4.5,-.5) to (4.5,.5);
\draw [blue] (0,0) to (1,0) to (2.5,-.5) to (3.5,-.5) to (5.5,.5);
\node [left](iv) at (0,0) {$\tau(e_v)$};
\node [right](tv) at (5.5,.5) {$\iota(e_v)$};
\node [red, above] at (4,.5) {$p(e_v)$};
\node [blue] at (5,-.1) {$\rho(e_v)$};
\node at (6.25,-.5) {$\EP$};
\end{tikzpicture}
\caption{The  local geodesic $\beta_v=e_v \cup \rho(e_v)$ and a characteristic cycle $p(e_v)\cup e_v$ containing $e_v$.} \label{fig:kappa}
\end{center}
\end{figure}

Since $\EP$ is contractible, the standard  collapse map takes $\chi_v$ to the loop in $\SaG$ representing $v$.
By the construction of $\SP$, a characteristic cycle for $v$ has one edge labeled $v$ and one edge labeled $\WP$ for each partition $\WP \in\bPi$  that splits $v$.
Such a path crosses the same hyperplanes as a locally geodesic loop $\beta_v$ representing $v$.  
(see Figure~\ref{fig:kappa}).   
Since  $v$ is not adjacent to  
any other  label on an edge crossed by $\chi_v$,   $e_v$ must lie in  $\beta_v$.
Similarly, any edge $e_{A}$ in $\chi_v$ for which $v \in  \max(A)$ must lie in $\beta_v$.

\subsubsection{Characteristic cycles and partitions}\label{sec:partitions} In this section we give a more detailed description of characteristic cycles $\chi_v$ in terms of partitions that split $v$. This depends on the following observation.

\begin{lemma}\label{lem:commute}   Suppose $\WP$ and $\WQ$ are compatible and both split $v$. Let $P$ and $Q$ be the $v$-sides of $\WP$ and $\WQ$.   If $\WP$ and
$\WQ$  are not adjacent 
then either $P\subset Q$ or $Q\subset P$. If  $\WP$ and $\WQ$ are adjacent, then  $P\not\subset Q$ and $Q\not\subset P$.
\end{lemma}
\begin{proof}   $P\cap Q$ contains $v$ so is not empty, and  $\overline P\cap \overline Q$  contains $v\inv$ so is not empty.  Since $\WP$ and $\WQ$ are compatible,   either   they are adjacent or
$P\subset Q$ or $Q\subset P$ by Lemma~\ref{lem:noncommuting}.  If  $\WP$ and $\WQ$ are adjacent
then $P$ intersects $\lk(Q)$, so $P\not\subset Q$, similarly $Q\not\subset P$.
\end{proof}

Now fix a vertex $v\in\G$, and for each $\WP\in\Pi$ that is not adjacent to $v$, let $P$ denote the $v$-side of $\WP$ and $\overline P$ the side that does not contain $v$ (note that $v^{-1}$ may be in $P$ or in $\overline P$).   Let $\WP_1,\ldots,\WP_k$ be the partitions in $\bPi$ that have $v$ as a maximal element (i.e. are based at $v$).  
By Lemma~\ref{lem:commute} the $v$-sides $P_i$ are nested, i.e. after possibly reordering we have  $P_1\subset\ldots\subset P_k$ (see Figure~\ref{fig:nest}).  For notational convenience, set $P_0=\{v\}$ and   $\WP_0=\{P_0|\overline P_0|\lk^\pm(v)\}$, and let $P_{k+1}=\overline P_0\setminus{\{v\inv\}}$.  The differences $dP_i=P_{i+1}\setminus P_i$ for $i=1,\ldots,k$ are called the {\em pieces} of the nest.

 If $\WQ\in\bPi$ is not adjacent to $v$ and does not split $v$, then Lemma~\ref{lem:noncommuting} implies that  some side of $\WQ$ is contained in either $P_i$ or $\overline P_i$ for each $i$; since $\WQ$ does not split $v$, this must be the side that does not contain $v$, which we have called $\overline Q$. We conclude that    $\overline Q$ is contained in some piece $dP_i$   of the nest. 

Let $\bPi_v$ denote the set of partitions $\bPi$ that split $v$.  Note that in addition to the partitions $\WP_i$, $\bPi_v$ may contain partitions that split $v$ but do not have $v$ as a maximal element;  such partitions may  be adjacent to 
each other.  A characteristic cycle $\chi_v$ has one edge for each element of $\bPi_v$, so in particular one edge for each $\WP_i$.  Let $\mathcal S_i$ be the consistent set of sides
$$
 \begin{cases}
 Q  &\hbox{ if } \WQ\in\bPi_v,  Q\supseteq P_i\\
\overline Q &\hbox{ if } \WQ\in\bPi_v,  Q\subsetneqq P_i\\
Q &\hbox{ if } \WQ\in\bPi\setminus \bPi_v \hbox{ is not adjacent to } $v$.
\end{cases}
$$
\begin{figure}. 
\begin{center}
  \begin{tikzpicture}
 \node (vinv) at (8,.5) {$v\inv$};
\draw [rounded corners] (0,0) to (0,1) to (1.4,1) to (1.4,0)   --cycle;
\node (P1) at (1,.7) {$P_1$};
\node (m) at (.4,.5) {$v$};
\draw [rounded corners] (-.1,-.1) to (-.1,1.1) to (3,1.1) to (3,-.1)   --cycle;
\node (Pi1) at (2.5,.7) {$P_{i}$};
\node (dots) at (1.8,.5) {$\ldots$};
\draw [rounded corners] (-.2,-.2) to (-.2,1.2) to (4.9,1.2) to (4.9,-.2)   --cycle;
\node (Pi) at (4.4,.7) {$P_{i+1}$};
\draw [fill=blue!20,rounded corners] (3.2,-.1) to (3.2,.5) to (4.5,.5) to (4.5,-.1)   --cycle;
\node (R) at (3.8,.2) {$\overline Q$};
\node (dots) at (5.3,.5) {$\ldots$};
\draw [rounded corners] (-.4,-.4) to (-.4,1.4) to (6.6,1.4) to (6.6,-.4)   --cycle;
\node (Pk) at (6.2,.7) {$P_{k}$};
  \end{tikzpicture}
  \caption{Partitions $\WP\in\bPi$ with $v\in \max(\WP)$ are nested.  Partitions $\WQ\in\bPi$ that don't split $v$   and are not adjacent to 
  $v$ have a side $\overline Q$ in the nest. ($\overline Q$ is the side that does not contain $v$.)}\label{fig:nest}
\end{center}
\end{figure}
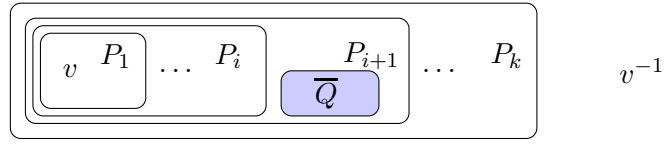
and $\overline{\mathcal S}_i$ the set obtained from $\mathcal S_i$ by replacing $P_i$ by $\overline P_i$.  Since the $P_i$ are nested,  changing sides of $\mathcal P_i$ doesn't change the fact that the relevant intersections are non-empty, so $\overline{\mathcal S}_i$ is still consistent.  Either set can be completed to a region by any consistent choice of sides of the $\WR\in \bPi$ that   are adjacent to  $v$.
One endpoint of the edge in $\chi_v$ labeled $\WP_i$  is a region that extends the set $\mathcal S_i$; call this endpoint $x_i$.   The other endpoint $\overline x_i$ of this edge is obtained by switching $P_i$ to $\overline P_i$; this extends $\overline{\mathcal S}_i$ (see Figure~\ref{fig:charcycle}).

We can now describe an arbitrary characteristic cycle $\chi_v$ in terms of partitions (refer to Figure~\ref{fig:charcycle}).  Start with any consistent choice $\mathcal S$ of sides of the $\WR\in\bPi$ that  are adjacent to  
$v$, and let $x_0$ be the region extending $\mathcal S$ that is given by choosing the $v$-side of every partition that  is not adjacent to  
$v$.     Define a partition $\WQ\in\bPi_v$  to be {\em innermost} if  its $v$-side $Q$ does not contain the  $v$-side of any other element of $\bPi_v$.   By  Lemma~\ref{lem:commute}, all innermost partitions in $\bPi_v$  are adjacent. 
  For the first edge of $\chi_v$ we may choose the edge labeled by any innermost $\WQ\in \bPi_v$.   For the next edge, we may choose any edge labeled by  $\WQ'\in\bPi_v$ that is  innermost in $\bPi_v\setminus \WQ$.  The following edge is labeled by any innermost element of $\bPi_v\setminus \{\WQ,\WQ'\}$ etc., and the loop is closed by an edge labeled $e_v$.

\begin{figure} 
\begin{center}\begin{tikzpicture}
\node (v) at (-.5,0) {$v$}; \node (vinv) at (6.5,.05) {$v^{-1}$}; 
\coordinate (p10) at (.5,0);\coordinate (p11) at (1.5,0); 
\coordinate (p20) at (4,0);\coordinate (p21) at (5,0); 
\coordinate (q0) at (2,1);\coordinate (q1) at (3,.5); %
\coordinate (q3) at (2.5,-.5);\coordinate (q4) at (3.5,-1);
\redvertex{(p10)} ; \redvertex{(p11)}; \redvertex{(p20)}; \redvertex{(p21)};
\redvertex{(q0)}; \redvertex{(q1)}; \redvertex{(q3)}; \redvertex{(q4)};
\node [above] (xi) at (p10) {$x_{i}$}; 
\node [above] (xibar) at (p11) {$\overline x_{i}$}; 
\node [above] (xi1) at (p20) {$x_{i+1}$}; 
\node [above] (xibar1) at (p21) {$\overline x_{i+1}$}; 
 \node  (Pi0) at (.75,-.25) {$P_i$};
     \node  (Pi4) at (4.05,-1.5) {$P_{i+1}$};
     \begin{scope}[decoration={markings,mark = at position 0.5 with {\arrow{stealth}}}]
 \midarrow [red] (6.5,.25) .. controls (6,3.5)  and (0,3.5) .. (-.5,.25); 
 \node (ev) at (3,3) {$e_v$};
 \end{scope}
 \node (dots) at (0,0) {$\ldots$};
\draw [rounded corners] (-.9,-.5) to (-.9,.5) to (1,.5) to (1,-.5)   --cycle;
\node (dots) at (5.5,0) {$\ldots$};
\draw [rounded corners] (-1.4,-1.75) to (-1.4,2) to (4.5,2) to (4.5,-1.75)   --cycle; 
\draw [blue, rounded corners] (-1.1,-.75) to (-1.1,.75) to (1.5, 1.25) to (2.55,1.25) to (1.5,-.75)   --cycle;
\draw [blue, rounded corners] (-1.2,-1) to (-1.2,1) to  (1.5,1.65) to (2.85,1.65) to (3.6,.5)
to (2.8,-1)   --cycle;
\draw [blue, rounded corners] (-1,-.6) to (-1,.7) to (1.25,.7) to (4.25,-.75) to (3.5,-1.5)   --cycle;
\draw [red] (p10) to (p11) to (q0) to (q1) to (p20) to (p21);
\draw [red]   (p11) to (q3) to (q4) to (p20);
\draw[red] (q1) to (q3);
\end{tikzpicture}
\end{center}
\caption{A characteristic cycle for $v$ has one edge $e_v$ and one edge for each partition that splits $v$. The partitions based at $v$ are nested.  Partitions that split $v$ but are not based at $v$ are indicated in blue; these have $\max>_f v$.  The blue partitions are adjacent 
if and only if they cross.}\label{fig:charcycle}
\end{figure}
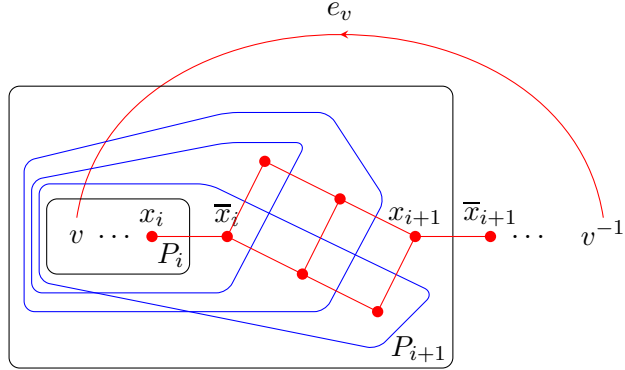

 If no two partitions that split $v$ are adjacent, the description of characteristic cycles in terms of partitions is particularly simple, since then the $v$-sides of all elements of $\bPi_v$ are nested so  any  characteristic cycle $\chi_v$  consists of an edge-path dual to the nest   plus an edge $e_v$ connecting its endpoints. This characteristic cycle is a local geodesic in $\SP$.  In particular we record 

  \begin{lemma}\label{lem:tdChi}  If $v$ is twist-dominant, then any characteristic cycle $\chi_v$ for $v$ is an edge-path in $\SP$ labeled by $v$ and the partitions that split $v$. Furthermore $\chi_v$ is a local geodesic in $\SP$.  
\end{lemma}

\begin{proof}  If $v$ is twist-dominant then  all partitions that split $v$ have $\max=\{v\}$, so none of them are adjacent.   
\end{proof}

\subsubsection{Characteristic cycles and minsets}

Since $\SP$ is locally CAT(0) its universal cover $\USP$ is CAT(0).  We will label edges and hyperplanes in $\USP$ with the same   label  as their their images in $\SP$.  The group $\AG$ acts on $\USP$ via deck transformations (preserving labels), using the identification of $\pi_1(\SP)$ with  $\AG$ induced by the  standard collapse map $c_\pi\colon \SP\to\SaG$. The following lemma uses standard CAT(0) methods to investigate the relation between characteristic cycles and this action.

\begin{lemma}\label{lem:CharCycles}  Let $A\in \bPi\cup V$ be a label and $v\in  \max(A)$. Then
\begin{enumerate}
\item  The minset of $v$ in $\USP$ decomposes as a product $\alpha_v \times \UH_v$ where
$\alpha_v$  is an axis for $v$ containing an edge $\tilde e_v$ and $\UH_v$ is the dual hyperplane.
\item For each edge in $\SP$ labeled $A$ there is a unique edge $e_v$ such that $e_A$ and $e_v$ are contained in a local geodesic $\beta_v$, hence every characteristic cycle for $v$ containing $e_v$ contains $e_A$, and vice versa.
\item The carrier $\kappa(H_A)$ lies in the image of $\Min(v)$ and the induced cubical structures on $H_A$ and $H_v$ are isomorphic.
\item If $w$ is adjacent to $A$
then the carrier of $H_A$ contains a characteristic cycle for $w$.
\end{enumerate}
\end{lemma}

\begin{proof}  Consider the minset of $v$ in the universal cover $\USP$.  By standard properties of CAT(0) spaces, $\Min(v)$ decomposes as an orthogonal product $\alpha_v \times Y$ where $Y$ is a convex subspace of  $\USP$ and $\alpha_v$ is an axis for $v$.  The image of $\alpha_v$ under the projection $\USP\to\SP$ is a local geodesic $\beta_v$. By the comments after Definition~\ref{def:CharCycle} we may assume $\beta_v$ contains an edge $e_v$, and thus that  $\alpha_v$ contains a lift $\tilde e_v$ of $e_v$.   We conclude that $\UH_v$ must contain a copy of $Y.$

 Conversely, we claim that every edge dual to $\UH_v$ lies on an axis for $v$, so by convexity this copy of $Y$ contains $\UH_v$.  Suppose  ${\tilde e}'_v$ is another edge dual to $\UH_v$,  separated from $\tilde e_v$ by a square whose other label is $A\in \lk_\bPi(v)$.  Let $\chi_v$ be a characteristic cycle for $v$ containing $e_v$.  Since every edge-label $B$ on $\chi_v$ splits $v$, we have $\lk_\bPi(B)\supseteq \lk_\bPi(v)\ni A$, so $\SP$ contains an annulus $\chi_v\times e_A$.  The boundary of this annulus is two characteristic cycles, one containing $e_v$ and one containing the image $e'_v$ of $\tilde e'_v$, so these two characteristic cycles are homotopic, and  correspond to two different axes for $v$, one containing   $\tilde e_v$ and one containing   $\tilde e'_v$.  Since any two edges dual to $H_v$ can be connected by a sequence of squares, this proves (1).

As observed above, for any $A\in \bPi$ with $v\in  \max(A)$, the local geodesic $\beta_v$ containing $e_v$ also contains a (unique) edge labeled $e_A$.  It follows that the axis through $\tilde e_v$ contains a lift of $e_A$, hence the dual hyperplane $\UH_A$ contains a subspace parallel to $\UH_v$.  Since every edge that  is adjacent to  
$A$ is also  adjacent to  
$v$, these two hyperplanes must, in fact, be isomorphic.  Thus the carrier of $\UH_A$ lies entirely in the minset of $v$ and every edge dual to $\UH_A$ lies on an axis containing an $e_v$ edge.  This proves (2) and (3).

For (4), note that since $w$  is adjacent to  
$v$, $e_v$ and $e_w$ span a cube in the carrier of $H_v$.  Let $\chi_w$ be a characteristic cycle containing $e_w$.  The label on every edge of this cycle  is also adjacent to  $v$, so the entire characteristic cycle is contained in the carrier of $H_v$.
\end{proof}

\begin{corollary}\label{lem:splits} If an edge $\widetilde e$ of $\USP$ is in $\Min(v)$ then its image in $\SP$ is labeled either by $v$, by some partition that splits $v$, or some label that is adjacent to $v$. 
\end{corollary}
\begin{proof}
By Lemma~\ref{lem:CharCycles} (1) $\Min(v)\iso \alpha_v\times \widetilde H_v$, and we may assume $\alpha_v$ is a lift of the local geodesic $\beta_v$ described in Section~\ref{sec:CharCycles}.   An edge $e_A$ of $\SP$ can only be in  $\beta_v$ if $A=v$ or $A$ splits $v$.  (Warning:  splitting $v$ does not guarantee that $e_A$ will be contained in $\beta_v$ unless $max(A)=v$.)  The hyperplane $H_v$ is parallel to a subcomplex with all labels adjacent to $v$. 
\end{proof}

\subsection{Subcomplexes of  $\SP$ associated to a generator}\label{subcomplexes}

Fix a compatible set $\bPi$  of \GW-partitions.  We will use the graph $\Gamma_\Pi$ defined in Definition~\ref{def:GammaPi}, with vertices  $V\cup\bPi,$ to describe certain subcomplexes of $\SP$ associated to a generator $v\in V$. We remark that $\Gamma_\bPi$ can be used to encode the fold relation:  $A\leq_f B$ if and only if $\lkp(A)\subseteq \lkp(B)$.  However, it does {\em not} encode the twist relation; this will be explored further in Section~\ref{sec:TMHyperplanes}.

\begin{definition} Given a set vertices   $\Lambda$ of $\Gamma_\Pi$, the  {\em span of $\Lambda$}, denoted $\spn(\Lambda),$ is the subcomplex of $\SP$ consisting of those cubes with all edge labels  in $\Lambda$.
\end{definition}

\begin{example}
 $\spn(\Pi)=\EP$.
 \end{example}

 Now fix $v \in V,$  let $e_v$ be an edge labeled by $v$ and let $H_v$ be the hyperplane in $\SP$ dual to $e_v.$ The  carrier $\kappa(H_v)$ is a product
$$\kappa(H_v) = e_v \times \Slkv,$$
where $\Slkv$ is the connected component of $\spn(\lk_\Pi(v))$ that contains the terminal vertex $x$ of $e_v$. 

 Since $v\in\UF_\Pi(v),$ some connected component of $\spn(\UF_\Pi(v))$ contains $x$.  Denote this component by $\Snlv$.
Since every vertex  of $\lk_\Pi(v)$ is  linked to every vertex of $\UF_\Pi(v)$, the product of these two subcomplexes is also a subcomplex of $\SP$:
$$\Sv = \Snlv \times \Slkv.$$

\begin{example} \label{ex:decomposition} If $v$ is twist-dominant then $\UF(v)=\{v\}$, so $\UF_\Pi(v)$  consists of $v$ and partitions based at $v$.  These are precisely the labels in any characteristic cycle for $v$ (see the discussion at the end of Section 3.3.1),   so the characteristic cycle $\chi_v$ containing $x$ is one component of $\spn(\UF_\Pi(v)$) Thus,
$$\Sv =\chi_v \times \Slkv \cong \chi_v \times H_v,$$
and $\Sv$ is equal to the image in $\SP$ of the minset of $v$ in $\USP$.
\end{example}

If $v$ is twist-minimal then  $\Sv$ can be considerably larger and more complicated than the image of $\Min(v)$.  However, the following lemma holds for any $v \in V$.

\begin{proposition}\label{prop:generators}  The subcomplex  $\Slkv$ contains at least one characteristic cycle for every $u \in \lk(v)$, and $\Snlv$ contains at least one characteristic cycle for every $w \in \UF(v)$.
\end{proposition}

\begin{proof}  Let $e_v$ be an edge in $\SP$ labeled $v$, and let $x$ be the terminal vertex of $e_v$.  Then $\Sv$ contains $\kappa(H_v)$, so the first statement follows from Lemma \ref{lem:CharCycles}(4).

 For the second statement, let $w\in \UF(v)$ and recall that the labels on a characteristic cycle $\chi_w$ consist of $w$ and all partitions $\WP\in\Pi$ that split $w$. If $\WP$ is based at $m$ and splits $w$, then $\lk(m)\supseteq \lk(w)\supseteq \lk(v)$, so  $m\in\UF(v)$.  This shows that all characteristic cycles $\chi_w$ are contained in $\spn(\UF_\Pi(v)$).   It remains to check that the component of $\spn(\UF_\Pi(v))$ containing $x$ also contains a characteristic cycle for $w$.  For this it suffices to find an edge  $e_w$ in the same component as $x$.

 Let $e_w$ be an edge labeled $w$ whose terminal vertex $y$  has minimal distance in $\EP$ to $x$. (Recall that $\EP$ contains all vertices and is CAT(0).)  If $y=x$ we are done; otherwise connect $y$ to $x$ by a minimal length edge path $p$ in $\EP$.
We claim that this edge path lies entirely in $\spn(\UF(v))$.  

 To see this, let $\WP_1,\ldots,\WP_k$ be the successive labels on the path $p$ (all of these labels are partitions). Since the path has minimal length, each $\WP_i$ occurs only once. The vertex $y$ is a terminal region for $w$,  $x$ is a terminal region for $y$, and the two regions differ by changing the sides of each $\WP_i$ on the path, say from $P_i$ to $\overline P_i$.  

 If $\WP_i$ is not in $\lkp(w)$ it is not in $\lkp(v)$ either, so $v$ and $w$ must be in different sides of $\WP_i$, specifically $w\in P_i$ and $v\in\overline P_i$.  Since each $\WP_i$ is a \GW-partition, this means $v$ and $w$ are in different components of $\Gamma\setminus \st(\WP_i)$ for all $i$.  But $\lk(v)\subset\lk(w)$, so this can only happen if $\WP_i\in \UF_\bPi(v).$   Thus we will be done if we can show that no $\WP_i$ is adjacent to $w$ 

Suppose to the contrary that some partition along the path is in $\lk_\bPi(w)$; let $\WP_i$ be the first such partition.  We first claim that $\WP_i$ is adjacent to  $\WP_{i-1}$. If not, then there is a unique pair of sides   of $\WP_i$ and  $\WP_{i-1}$ with empty intersection.  Since $P_{i-1}\cap P_i$, $\overline P_{i-1}\cap  P_i$ and $\overline P_{i-1}\cap \overline P_i$ all correspond to vertices of the path $p$, the empty intersection must be $ P_{i-1}\cap  \overline P_{i}$.   Since $\WP_{i-1}$ is not in $\lkp(w)$, $w\in P_{i-1}$, as observed in the previous paragraph.  But  Lemma~\ref{lem:noncommuting},   implies that  $\lk(\WP_i),$  which contains $w$, does not  intersect $P_{i-1}$, giving a contradiction. 

  Since $\WP_i$ is adjacent to $\WP_{i-1}$ we can re-route the path $p$ to obtain a new path with the same edge labels that crosses $\WP_i$ before it crosses $\WP_{i-1}$. Repeating the argument  we can arrange that $\WP_{i}$ labels the first edge of the path, so this edge has an endpoint at $y$.  Filling in a square with edge labels $w$ and $\WP_{i}$, we obtain an edge labeled $w$ that is closer to $x$, contradicting our original choice of $e_w$.  
\end{proof}

Now let $\USv \subset \USP$ be the connected component of the lift of $\Sv$ containing an axis for $v$.  This decomposes as a product
 $\USv = \USnlv \times \USlkv.$

\begin{corollary}\label{cor:ContainsAxes}   
$\USv$ is preserved by the action of the special subgroup $A_{\UF(v)} \times A_{\lk(v)}$, and $\USnlv $ contains an axis for every element of the group $A_{\UF(v)} $.   If $\alpha_v \subset \USnlv$ is the axis for $v$, then $\alpha_v \times \USlkv$ is the minset of $v$. 
 \end{corollary}
 
\begin{proof}  First note that the subcomplexes $\Snlv$ and $\Slkv$ are locally convex in $\SP$.   This follows from the fact that a cube lies in one of these subcomplexes if and only if its edges all lie in that subcomplex.
By general properties of CAT(0) spaces, a locally convex embedding of a subspace lifts to a globally convex embedding on universal covers and induces an injective map on fundamental groups.

It follows from Proposition \ref{prop:generators}   that under the standard collapse map, the image of $\pi_1(\Snlv)$ in $\pi_1(\Sa_\G)=\AG$ is the subgroup  $A_{\UF(v)}$ and the image of  $\pi_1(\Slkv)$ is $A_{\lk(v)}$.  Hence these subgroups preserve the lifts $\USnlv$ and $\USlkv$.  Since these subspaces are convex in $\USP$, they contain axes for each element of the corresponding subgroup.

 The last statement follows by Lemma~\ref{lem:CharCycles}(1) since $\USlkv$ is parallel to and isomorphic to $\widetilde H_v.$ 
\end{proof}

\subsection{Branch loci}\label{sec:tdomCC}
 In Section~\ref{sec:OG}   we will be given a locally CAT(0) space $X$ with fundamental group $\AG$ and will need to construct an isomorphism of $X$ with some blowup $\SP$.  We will do this using  the action of $\AG$ on the universal cover $\widetilde X$. In this section we discuss features of the action of $\AG$ on $\USP$ that will help in this task.

\begin{definition}\label{branch point} A point $x\in \Min(v)\subset \USP$ is a \emph{branch point for $v$}   if the link of $x$ in $\Min(v)$ is strictly smaller than the link of $x$ in $\USP$.   Denote the branch locus of $v$ by $\Br(v)$.
\end{definition}

 (Recall that the link of a point $x$ in a CAT(0) metric space $X$ is defined to be the boundary of a small ball centered at $x$. This is standard terminology; the reader should not confuse this with the graphical links used elsewhere in this paper.)

  If $v$ is central, then $\Min(v)=\USP$ and hence $\Br(v)=\emptyset$.
No \GW\ partition can split a central $v$, so
 in every blowup a characteristic cycle for $v$ consists of a single edge which is a loop.  For the rest of this section we assume that $v$ is not central, and  show that in this case the location of hyperplanes in $\USP$  is determined by branch loci of minsets.

\begin{proposition}\label{prop:BranchLocus1} Let $H_A$ be a hyperplane in $\SP$ with $v\in \max(A)$, and $\UH_A$ a lift of $H_A$ to $\USP.$  If $v$ is not central then each component of the boundary of $\kappa(\UH_A)$ contains a branch point of $\Min(v)$.
\end{proposition}

\begin{proof}  Let $e_A$ be an edge dual to $H_A$.  
By Lemma~\ref{lem:CharCycles} (3) we know that   $e_A$   is contained in some  characteristic cycle $\chi_v$ for $v$.  Let $e_B$ be the edge following $e_A$ in $\chi_v$, so that either $B=v$ or $B$ is a partition that splits $v$.  

 If $B$ is a partition based at $w$ and $w>_fv$, choose $u\in\lk(B)\setminus \lk(A)$.   Denote the common endpoint  of $e_A$ and $e_B$ by $x$ and let   $\bdry_x(A)$ and  $\bdry_x(B)$ be the components of $\bdry(\kappa(H_A))$  and  $\bdry(\kappa(H_B))$ respectively that contain  $x$.  Then $\bdry_x(A)\iso H_v$ is a subcomplex of $\bdry_x(B)$, but $\bdry_x(B)$ is strictly larger, since 
$\kappa(H_B)$ contains a square with edge labels $u$ and $B$, and that square is not in $\kappa(H_A)$.  Thus there is a point $x'\in\bdry_x(A)$ that is adjacent to some edge $e_C$ with $C$ not adjacent to $v$.  This means that no lift of $e_C$ is contained in $\Min(v)\iso \alpha_v\times \widetilde H_A\iso  \alpha_v\times \widetilde H_v$, i.e.   any lift $\tilde x'$ of $x'$  lying on   $\bdry(\kappa(\UH_A))$   is a branch point for $v$.

 If $v\in\max(B)$ then we need to choose our characteristic cycle carefully and look more closely at the vertex $x$. To this end, we recall the description of characteristic cycles  from Section~\ref{sec:partitions}.  
If $\WP_1,\ldots,\WP_k$ are the partitions  in $\bPi$ that are based at $v$ then the $v$-sides $P_i$ of the $\WP_i$ are nested and, 
for notational convenience  we set 
$P_0=\{v\}$,  $\WP_0=\{P_0|\overline P_0|\lk^\pm(v)\}$, and  $P_{k+1}=\overline P_0\setminus{\{v\inv\}}$, so (after possibly reordering) we have
$$ P_0\subset P_1\subset \ldots\subset P_k\subset P_k\subset P_{k+1}.$$ 
Since $A$ is based at $v$, we have $A=\WP_i$ for some $i=0,\ldots,k$  and the vertex $x$ corresponds to a region that extends the consistent set $\overline{\mathcal S}_i$  given by 
$$
\overline{\mathcal S}_i = 
 \begin{cases}
 Q  &\hbox{ if } \WQ \hbox{ splits } $v$ \hbox{ and  } Q\supsetneqq P_i\\
\overline Q &\hbox{ if } \WQ\hbox{ splits } $v$  \hbox{ and  }    Q\subseteq P_i\\
Q &\hbox{ if } \WQ  \hbox{ does not split $v$ and is not adjacent to } $v$,
\end{cases}
$$
where $Q$ is the $v$-side  of $\WQ$.

 Since $v\in\max(B)$, there is no $\WQ$ in $\bPi$ whose $v$-side $Q$ satisfies $P_i\subsetneqq Q\subsetneqq P_{i+1},$  so for any characteristic cycle the edge labeled $A=\WP_i$ is followed by an edge labeled $B=\WP_{i+1} $ (this situation is illustrated in Figure~\ref{fig:seven}). We claim that for some characteristic cycle $\chi_v$ there is an edge $e_C$ at this vertex whose label  $C$ is not adjacent to $v$ and does not split $v$, so by Lemma~\ref{lem:splits} no lift of this edge is  in $\Min(v)$.  But  $\bdry(\widetilde H_A)\subset \Min(v)$ does contain a lift of $x$, so that  is a branch point.

 Recall that if a partition $\WQ$   is not adjacent to 
$v$ and does not split $v$, then it has a side $\overline Q$ sitting  in some piece $dP_j=P_j\setminus P_{j-1}$ of the nest; call this the {\em nesting side} of $\WQ$.   We say $\WQ$ is {\em outermost} if $\overline Q$ is not properly contained in any other such  nesting side.    We call a vertex $u$  {\em outermost} if $u$ is not adjacent to $v$ and is not contained in any nesting side.

 Since $v$ is not central and  $P_{i+1}\neq P_i$, the piece $dP_{i+1}=P_{i+1}\setminus P_i$ must contain at least one outermost side or vertex; let $C$ be the corresponding label.  If $C$ is a partition, then the condition that $C$ is outermost guarantees that both sides of $C$ are consistent with $\overline{\mathcal S}_i$.  We claim we can extend $\overline{\mathcal S}_i=\mathcal S_{i+1}$ to a region that is an endpoint of an edge $e_C$, i.e. we can choose sides of all remaining $\WQ$ that are consistent with each other and with both sides of $C$.

 The remaining $\WQ$ are  those that are adjacent to $v$.  These do not split $v$ and are adjacent to every partition that does split $v$.     
 Suppose  such a $\WQ$ is not adjacent to $C$.  If $C$ is an element of $V\cup V^{-1}$ choose the side of $\WQ$ that contains $C$;  the result   is a terminal region for $C$, i.e. there's an edge labeled $C$ at the corresponding vertex.

  If $C$ is a partition, let $C^\times$ denote the nesting side.  Both sides of $\WQ$ intersect $\lk(v)$, so they cannot be contained in $C^\times$.  It follows from Lemma~\ref{lem:noncommuting} that some side $Q^\times$ must contain $C^\times$.   Note that $Q^\times$ intersects both sides of $C$, and also intersects the previously chosen side of any partition not adjacent to $\WQ$, so the complete set of chosen sides is a region.  Since $C$ is outermost, switching sides of $C$ still gives a region,  and we have found our vertex $x$.    
  \end{proof}

 \begin{figure} 
\begin{center}
  \begin{tikzpicture}
 \node (v) at (8.5,.55) {$v^{-1}$};
\node (m) at (.2,.5) {$v$};
\node (dots) at (1.25,.5) {$\ldots$}; 
\draw [rounded corners] (-.2,-.3) to (-.2,1.5) to (2.8,1.5) to (2.8,-.3)   --cycle; 
\node (Pi1) at (2.3,-.05) {$P_{i}$};
\fill [blue!20, rounded corners]  (3.2,-.3) to (3.2,.3) to (5.5,.3) to (5.5,-.3)   --cycle;
\fill [blue!30, rounded corners]  (4.2,-.3) to (4.2,.3) to (4.5,.3) to (4.5,-.3)   --cycle;
\draw [rounded corners] (3.2,-.3) to (3.2,.3) to (4.5,.3) to (4.5,-.3)   --cycle; 
\draw [rounded corners] (4.2,-.3) to (4.2,.3) to (5.5,.3) to (5.5,-.3)   --cycle; 
\draw [fill=blue!20, rounded corners] (3.2,.8) to (3.2,1.5) to (4.5,1.5) to (4.5,.8)   --cycle; 
\draw [fill=blue!30, rounded corners] (4,.9) to (4,1.4) to (4.3,1.4) to (4.3,.9)   --cycle; 
\node (dots) at (7.56,.5) {$\ldots$}; 
\draw [rounded corners] (-.4,-.4) to (-.4,1.6) to (6.6,1.6) to (6.6,-.4)   --cycle; 
\node (Pk) at (6.2,-.15) {$P_{i+1}$};
\node (R) at (3.45,1.05) {$\overline Q$};
\vertex{(5.25,1.2)} ;
\draw[red] (2,.5) to (7,.5);\redvertex{(2,.5)};\redvertex{(4.5,.5)};  \redvertex{(7,.5)};
\node[above right] at (2,.5) {$x_i$};
\node  at (5.5,.7) {$\overline x_{i}=x_{i+1}$};
\draw[red] (4.5,.5) to (3.65,1.35);\vertex{(3.65,1.35)};
\draw[red] (4.5,.5) to (5.25,1.2);\node [right] at (5.25,1.2) {$u$};
\draw[red] (4.5,.5) to (3.8,0);\vertex{(3.8,0)};
\draw[red] (4.5,.5) to (4.8,0);\vertex{(4.8,0)};
\end{tikzpicture}
 \caption{If there is no $\WQ$ with $P_i\subsetneqq Q\subsetneqq P_{i+1},$ the remaining edges at the vertex  $\overline x_{i}= x_{i+1}$  are either adjacent to
$v$ or correspond to  those $\overline Q$ and $u$  in $dP_i=P_i\setminus P_{i-1}$ that are outermost.}\label{fig:seven}
\end{center}
\end{figure}
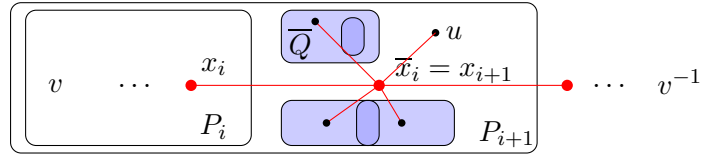

For a generator $v$,  Lemma~\ref{lem:CharCycles}(1) gives a decomposition
$\Min(v) \iso \alpha_v \times \UH_v$ where $\alpha_v$ is an axis containing an edge $\widetilde e_v$ and $\UH_v$ is the hyperplane dual to $\widetilde e_v$.   Let
$$\pr_v: \Min(v) \iso \alpha_v \times \UH_v \rightarrow \alpha_v$$
be the  (nearest point)  projection map corresponding to this decomposition.

If $\SP$ is a blowup with the standard collapse marking, an axis $\alpha_v$ in $\USP$ is transverse to some lift $\UH_A$ of a hyperplane $H_A$ if and only if $A=v$ or $A$ is a partition that splits $v$.  In either case we say $\UH_A$ {\em splits} $v$.

\begin{corollary}\label{cor:BranchImage}If $v$ is not central,  the image $\pr_v(\Br(v))$ of the branch set under projection to $\alpha_v$  is a set of discrete points and closed intervals. Each component of the complement of this image crosses exactly one hyperplane, which lifts a hyperplane  in $\SP$ labeled either by $v$ or by a partition $\WP$  based at $v$.  
\end{corollary}

\begin{proof}   First note that being a branch point is a closed condition and $pr_v$ is a closed map, so $\pr_v(\Br(v))$ is closed.  By Corollary~\ref{cor:ContainsAxes} the min set  $\Min(v)$ decomposes as  $\alpha_v \times \USlkv \subseteq \USnlv \times \USlkv$, and by Lemma~\ref{lem:CharCycles} we may assume $\alpha_v$ contains lifts  of all edges in $\SP$ labeled by some $A$ with $v\in\max(A)$. 

 Any segment of $\alpha_v$ lying in the interior of a cube $C \subset \USnlv$ of dimension $\geq 2$ lies entirely in the branch set, since $C$ is not contained in the minset.   So
the only segments of $\alpha_v$ which might not be in the image are contained in edges $\tilde e$ of $\USnlv$.  Let $\tilde e$ be such an edge, and $\widetilde H$  the hyperplane dual to $\tilde e$.  The hyperplane  $\widetilde H$ projects to a hyperplane $H_A$ in $\SP$ for some $A\in V\cup\bPi$. 

 Since $v$ splits $A$, either $v\in  \max(A)$ or any $w\in \max(A)$ satisfies  $w>_f v$.    If $v\in \max(A)$, then $\kappa(\widetilde{H})=\widetilde e\times \widetilde{H}$ lies entirely in $\Min(v)$ and hence the interior of $\kappa(\widetilde{H})$ does not contain any branch points.  By Proposition~\ref{prop:BranchLocus1}, the two boundary components of $\kappa(\widetilde{H})$ do contain branch points so  the two endpoints of $\tilde e$ do lie in  $\pr_v(\Br(v))$.

 If  $w\in \max(A)$ satisfies  $w>_f v$, then  there exists $u\in \lk(w)$  with $u \notin \lk(v)$.   By Proposition~\ref{prop:generators} $\kappa(H_A)$ contains a characteristic cycle for $u$.  It follows that  $\kappa(\UH)$ contains a square with edges labeled by $A$ and $u$.  This square does not lie in $\Min(v)$ (since $u$ is not in $\lk(v)$). This implies that closest  edge in $\Min(v)$ that is parallel to the edge labeled $A$ is also in a square with one edge outside of $\Min(v)$, so this edge is  entirely contained  in the branch locus. Since this edge is dual to $\widetilde H$, it projects to $\tilde e$, so $\tilde e$ is contained in $pr_v(\Br(v))$.  
\end{proof}

If $v$ is twist-dominant then every partition in $\bPi_v$ has $\max=\{v\}$, so every $\chi_v$ is an edge-path in $\SP$ and its lift to $\USP$  is an edge path which is an axis $\alpha_v$  for $v$. By Proposition  \ref{prop:BranchLocus1} every vertex  of $\alpha_v$ is the projection of a branch point and there are no branch points in the interior of edges.  We record these observations in the following statement.

\begin{corollary}\label{cor:BranchLocus2}  If  $v$ is twist-dominant and not central, then any lift of a characteristic cycle to $\USP$ is an axis $\alpha_v$, and $\pr_v(\Br(v))$ is precisely the set of vertices of  $\alpha_v$.
\end{corollary}

\section{Hyperplanes in $\Gamma$-complexes}\label{sec:TMHyperplanes}

  Let $(X,h)$ be a point of $\SG,$ i.e. a rectilinear $\G$-complex with an untwisted marking.  If we choose an isomorphism $X\iso \SP$ the hyperplanes of $X$ acquire labels, and we can use these labels to define what it means for a hyperplane to be twist-dominant or twist-minimal.   In this section we show that this designation is independent of the isomorphism and can be detected using the action of $\AG$ on $\widetilde X$ induced by any untwisted marking.

 To this end, let $\Ca(X)$ be the \emph{crossing graph} of $X$, i.e., the graph whose vertices are the hyperplanes of $X$, and where two vertices are connected by an edge if the corresponding hyperplanes intersect non-trivially.
 If we give $X$ the structure of a blowup, then $\Ca(X)\cong \Gamma_\Pi$.

 We defined twist and fold orderings on partitions $\WP$ by choosing an element $v\in\max(\WP)$ and using the twist and fold orderings defined in terms of $\Gamma$.  The defining graph $\Gamma$ occurs as a subgraph of $\Gamma_\Pi$, but  the corresponding subgraph in $\Ca(X)$ is not well-defined since it depends on a choice of isomorphism $\Ca(X)\cong \Gamma_\Pi$.   We do know that  both orderings are well-defined on fold-equivalence classes in $\Gamma$, so it is natural to try to define these orderings on {\em fold-equivalence classes} of $\Ca(X)$, i.e. equivalence classes of vertices with the same link.  This works well for the fold ordering, but must be modified for the twist ordering, as we will see. In the end, 
 our notions of twist-dominant and twist-minimal will be defined using both $
\Ca(X)$ and  the combinatorial structure of $X$ itself.  

\subsection{Isomorphisms of $\G$-complexes}

First we define twist and fold orderings for hyperplanes in a $\G$-complex $X$ and show that, for any  isomorphism  $X\cong \Sa^{\bPi}$  these orderings coincide with the orderings of their labels, as previously defined.   Note that the ordering of labels is well-defined on fold-equivalence classes, so we need the same to be true here.

\begin{definition} Let $H$ be a hyperplane in a   $\G$-complex $X.$   The \textit{link} $\lk(H)$ of $H$ is
the link of $H$ in $\Ca(X)$.  In other words, $\lk(H)$ is the set of hyperplanes $K\neq H$ that intersect  $H$ non-trivially.  The {\em fold-equivalence class} $\eqH$  is
$$\eqH=\{K\mid \lk(K)=\lk(H)\}. $$
We then define $\eqH\leq_f \eqK$ if $\lk(H)\subseteq \lk(K)$. 
\end{definition}

By Corollary~\ref{fact:commhyp} hyperplanes $H_A\neq H_B$ in $\SP$ intersect non-trivially if and only if their labels $A$ and $B$  are adjacent,  
so  this coincides with the notion previously defined for $A\leq_f B$.

Giving a combinatorial definition of the twist relation is more subtle, and requires us to look beyond the structure of $\Ca(X)$ to the combinatorial structure of $X$ itself.
 \begin{definition}We call a hyperplane  $H$ {\em cyclic} if $$\bigcup_{H'\in\eqH}\kappa(H')\iso H\times C,$$ where $C$ is a graph homeomorphic to $S^1$.
Define $\eqK\leq_t \eqH$ to mean $H$ is cyclic and $\lk(K)\cup\eqK \subseteq \lk(H)\cup\eqH$.
\end{definition} 
 The second condition in the definition of $\eqK\leq_t\eqH$ is the analogue of $\st(v)\subseteq \st(w)$ for the twist-relation on $\G$. However, the second condition alone does not  capture the notion of twist-dominance.  For instance, if  $\Gamma$ is a 4-cycle with vertices $a,b,c,d$, then in the Salvetti $\SaG$ we have  $\lk(H_a)\cup \![\![H_a]\!]=\{H_a,H_b,H_c,H_d\}=\lk(H_b)\cup[\![H_b]\!]$  but neither $a$ nor $b$ is twist-dominant as generators.

  Since not every fold-equivalence class $\eqH$ is cyclic, we only have $\eqK\leq_t\eqH$ when $\eqH$ is cyclic.  Nevertheless, it is transitive: if $\eqK\leq_t \eqH$ and $\eqH\leq_t \eqL$ then $\eqL$ must be cyclic so $\eqK\leq_t \eqL$. Also note that the analogue of Lemma \ref{lem:twist_and_fold} still holds for fold equivalence classes of hyperplanes: if $\eqK,$ $\eqH,$ and $\eqL$ are distinct then $\eqK\leq_t\eqH\leq_f\eqL$ is not possible.

\begin{definition}\label{def:GeneralTwistDominance} A hyperplane  $H$ is {\em twist-dominant} if there is some hyperplane $K\neq H$ with $\eqK\leq_t\eqH$; in particular $H$ must be cyclic. If $H$ is not twist-dominant, it is {\em twist-minimal}.
\end{definition}
 If $X=\SaG$, each hyperplane is labeled by a generator and the two notions of twist-dominance coincide.  Indeed, a hyperplane $H_v$ of $\SaG$ is cyclic if and only if $v$ is not fold-equivalent to another generator.  Then if there exists $w\neq v$ with $[\![H_w]\!]\leq_t [\![{H_v}]\!]=\{H_v\}$ this means $w\leq_t v$, and conversely.

\begin{lemma}\label{lem:TDSingle} Let $X$ be a $\Gamma$-complex and choose an isomorphism $X\cong \SP$.  For any hyperplane $H_B\subset \SP$, $[\![H_B]\!]$ is twist-dominant if and only if there exists $H_A$ such that $ \max(A)\leq_t  \max(B)$.
\end{lemma}
\begin{proof}
First suppose $[\![H_B]\!]$ is twist-dominant.  Then there exists $H_A$ such that $[\![H_A]\!]\leq _t [\![H_B]\!]$.  As noted above, the fold equivalence class of the hyperplane $H_A$ in $\SP$ consists of all hyperplanes $H_{A'}$ with $ \max(A)\sim_f  \max(A')$. Since $[\![H_B]\!]$ is cyclic, under the collapse map $c:\SP\rightarrow \SaG$, $[\![H_B]\!]$ maps to a single hyperplane labeled by a generator $v$. Hence, the fold equivalence class of $v$ is just $\{v\}$, and all of the hyperplanes in $[\![H_B]\!]$ have $v$ as the unique maximal element. Since $[\![H_A]\!]\leq_t[\![H_B]\!]$, this means $ \max(A)\leq_t v= \max(B)$.

Conversely, if there exists $H_A$ such that $ \max(A)\leq_t  \max(B)$, then  any $v\in \max(B)$ is twist-dominant  and $[\![H_A]\!]\cup \lk(H_A)\subseteq[\![H_B]\!]\cup \lk(H_B)$.  By Lemma \ref{lem:tdominant},   $\max(B)=\{v\}$ and  the hyperplanes with $ \max(H)=\{v\}$ coincide with the hyperplanes that split $v$, which are exactly those occurring along any characteristic cycle $\chi_v$ for $v$. It follows from Example \ref{ex:decomposition} that  $\Sv\cong \chi_v\times H_v$, but on the other hand this implies that \[\Sv=\bigcup_{H\in [\![H_v]\!]}\kappa(H).\]
This proves $[\![H_v]\!]=[\![H_B]\!]$ is cyclic as well, and therefore $[\![H_B]\!]$ is twist-dominant.
\end{proof}

Since the definitions of  cyclic  hyperplane  and the link of a  hyperplane depend only the combinatorial structure of $X$, the following is an immediate corollary.

\begin{corollary}\label{cor:independent}  Let $i\colon \SP\to \SPp$ be an isomorphism of cube complexes.  Then $i$ preserves the twist and fold ordering on edge labels.
\end{corollary}

\subsection{Untwisted markings}\label{sec:tminimal}
In this section we show that we can detect twist-minimal hyperplanes in a $\G$-complex $X$  using only the action of $\AG$  that an untwisted marking $h\colon X\to \SaG$ induces on the CAT(0) space $\widetilde X$.

We begin by recalling  some basic facts  about untwisted markings. Define  $U^0(\AG)$ to be the subgroup of $U(\AG)$ generated by inversions, folds and partial conjugations.
 
\begin{lemma}\label{lem:U0PreservesLinks}  For any $v\in V$, both $A_{\lk(v)}$ and $A_{\UF(v)}$ are invariant up to conjugacy under the action of $U^0(\AG)$.
\end{lemma}
\begin{proof}Let $\Delta\subseteq \G$ be any subgraph. We claim that the special subgroup $A_{\lk(\Delta)}$ is invariant up to conjugacy under $U^0(\AG)$, where $\displaystyle \lk(\Delta)=\cap_{v\in \Delta}\lk(v)$.  The lemma will follow by taking $\Delta=\{v\}$ and $\Delta=\lk(v)$, respectively. If $\lk(\Delta)=\emptyset$ then we set $A_\emptyset=\{1\}$ which is trivially invariant. Otherwise, assume $\lk(\Delta)\neq \emptyset$. We consider each type of generator of $U^0(\AG)$. Clearly inversions preserve $A_{\lk(\Delta)}$.  If $v\in \lk(\Delta)$ and $v<_f w$, then $\Delta\subseteq \lk(v)\subseteq \lk(w)$, hence $w\in \lk(\Delta)$ as well.  It follows that $A_{\lk(\Delta)}$ is also invariant under the fold sending $v$ to $vw$.  Finally, suppose consider a partial conjugation by $w\in V$. If $\Delta$ is not contained in $\st(w)$, then there exists $v\in \Delta\setminus \st(w)$ hence $\lk(\Delta)\setminus \st(w)$ is contained in the same component of $\Gamma\setminus\st(w)$ as $v$. Hence a partial conjugation by $w$ preserves $A_{\lk(\Delta)}$ up to conjugacy. On the other hand, if $\Delta\subseteq \st(w)$ then either $w\in \Delta$, whence $\lk(\Delta)\subset \st(w)$, or $w\in \lk(\Delta)$.  Either way, any partial conjugation by $w$ preserves $A_{\lk(\Delta)}$, and the claim is proved.
\end{proof}

 The next lemma implies that after changing the collapse map, we can always assume the marking lies in $U^0(\AG)$. 
 \begin{lemma}\label{lem:Gcollapse} Let $X$ be a $\G$-complex.  Every untwisted marking $h\colon X\to \SaG$ is a $\G$-collapse map $c$ followed by  an element of $U^0(\AG)$,
\end{lemma}
\begin{proof} By definition $h$ is untwisted if there is an isomorphism $i\colon X\iso \SP$ for some $\bPi$ so that the composition $h \circ c\inv$ (where $c=c_\pi i$ and $c^{-1}$ is a homotopy inverse for $c$) induces an element $\varphi\in U(\AG)$ on $\pi_1(\SaG)=A_\G.$ This  condition is independent of the choice of $i$. The subgroup $U(\AG)$ is generated by inversions, partial conjugations, elementary (right and left) folds and graph automorphisms.   Any product of these is equal to a product with a single graph automorphism $\sigma$ as the initial element. The automorphism $\sigma$ permutes $V$, sends a $\G$-Whitehead partition $\WP$ to $\sigma\WP$, and induces an isomorphism $\SP$ to $\Sa^{\sigma\bPi}$, so the composition of the initial $\G$-collapse map $c=c_\pi  i$ with $\sigma$ is itself a $\G$-collapse map, and the rest of the factors are inversions, partial conjugations and elementary folds.
\end{proof}

Now let $H$ be a hyperplane in a $\G$-complex $X$, fix an untwisted marking $h$ and let $\bar h\colon \SaG\to X$ be a homotopy inverse for $h$, so that   $g\in \AG=\pi_1(\SaG)$ acts on $\widetilde X$ by the deck transformation $\bar h_*(g)\in \pi_1(X)$.
Define
$$ \Sing_h(H) = \{ v \in V \mid \textrm{an axis for $\bar h_*(v)$ crosses a lift of $H$}\}.$$
and let $\max_h(H)$ denote the set of maximal elements in $\Sing_h(H)$.

 A special case is when $h=c$ is just a collapse map.  In this setting, the set $\max_c(H)$ all belong to the same fold equivalence class of $\G$. 
In the next lemma we will see that the elements of $\max_h(H)$ also all belong to the same fold-equivalence class and moreover, that this equivalence class is actually independent of the marking $h$ up to graph automorphisms.

\begin{lemma}\label{lem:Sing}   Let $X$ be a $\G$-complex, $h\colon X\to \SaG$  an untwisted marking and let $m_h$ be any element of $\max_h(H)$.
 \begin{enumerate}
\item   If  $v\in\Sing_h(H)$, then $v\leq_f m_h$.
 \item   There is a $\G$-collapse map $c$ such that $m_c \sim_f m_h$  for any $m_c \in \max_c(H).$
 \end{enumerate}
Thus the maximal elements in $\Sing_h(H)$ all lie in the fold equivalence class of $\max_c(H)$.
\end{lemma}

\begin{proof}
By Lemma~\ref{lem:Gcollapse} we can write $$h=\varphi_n\circ\ldots\circ\varphi_1\circ c$$
where $c=c_\pi i\colon X\iso \SP\to \Sa $ is a $\G$-collapse map and each $\varphi_j$ induces an inversion, partial conjugation or elementary fold.  Let $\bar h$ denote the homotopy inverse of $h$.   By Lemma \ref{lem:U0PreservesLinks}, $U^0(\AG)$  preserves the special subgroup $A_{\UF(v)}$ up to conjugacy for every $v \in V$, and by  Corollary  \ref{cor:ContainsAxes}, the subcomplex $\widetilde X_{\UF(v)}=i\inv(\USnlv) $ contains an axis for every element of this subgroup.  Thus, some translate of $\widetilde X_{\UF(v)}$ contains an axis for $\bar h_*(v)$.  It follows that every hyperplane $H$ that splits $\bar h_*(v)$ has a lift that is dual to an edge in $\widetilde X_{\UF(v)}$, or in other words, if $v \in \Sing_h(H)$, then $v \leq_f m_c$.  Thus (1) will follow immediately from (2).

We will prove (2) by induction on $n$.
By definition, if $v \in \Sing_h(H)$, then the image of an $\bar h_*(v)$-axis crosses $H$ at least once.  For the purpose of this proof, we will need to keep track of more information about how many times it crosses $H$.  Begin by choosing an orientation for $H$, or equivalently, for a dual edge to $H$.  ($H$ is orientable since $X$ is a special cube complex.)  If $p$ is an edge path in $X$ we define the \emph{net crossing number} of $p$ with  $H$ to be
$$n(p,H)=  \textrm{\# positive crossings of $H$} -  \textrm{\# negative crossings of $H$}.$$
Note that two paths that are homotopic rel endpoints have the same net crossing number with respect to any hyperplane.
For a generator $v \in V$, set $n_h(v,H) = n(p_v,H)$ where $p_v$ is some (hence any) loop in $X$ representing $\bar h_*(v)$.   This is independent of basepoint since changing basepoints conjugates $p_v$ by a path connecting the two basepoints and hence leaves the net crossing number unchanged.  In particular, $n_h(vw,H)=n_h(v,H) + n_h(w,H)$.

Note that if $n_h(v,H) \neq 0$, then $v$ necessarily lies in $\Sing_h(H)$, but the converse need not be true.  In addition to property (2), we will prove by induction that the following property holds:
\begin{quote}
(3) For some $v \in \max_h(H)$,  $n_h(v,H)  \neq  0$.
\end{quote}

For $n=0$, $h=c$, so (2) is true trivially and for any $v \in \max_h(H)$, a characteristic cycle for $v$ crosses $H$ exactly once, so $n_h(v,H) = \pm 1$.

Now set $h'=\varphi_{n-1}\circ\ldots\circ\varphi_{1}\circ c,$  with homotopy inverse $\bar h'$ and assume by induction that (2) and (3) hold for $h'$.
 If $\varphi_n$ is an inversion $v \mapsto v^{-1}$, then $c \circ \bar h$ and $c \circ \bar h'$ agree on every generator except $v$.  Furthermore $\bar h'_*(v)$  and $\bar h_*(v)=\bar h'_*(v\inv)=\bar h'_*(v)\inv$ have the same axis in $\widetilde X$, so there is no change in which hyperplanes this axis crosses hence no change in the splitting set.  Only the sign of the net crossing numbers with these hyperplanes change.

If $\varphi_n$ is a partial conjugation then $c \circ \bar h_*(v)$ is conjugate to  $c \circ \bar h'_*(v)$ for every generator $v$, so an axis for $\bar h_*(v)$  is just a translate of an axis for $\bar h'_*(v)$.  Thus, the former crosses some lift of $H$ if and only if the latter crosses some lift of $H$, and again there is no change is the splitting set or the net crossing numbers for $H$.

It remains to consider the case that $\varphi_n$ is a right fold $v \mapsto vw^{-1}$ for some $w \geq_f v$  (the case of  left folds is symmetric.) Again $c \circ \bar h_*$ and $c \circ \bar h'_*$ agree on every generator except $v$, and $c \circ \bar h_* (v)= c \circ \bar h'_*(vw)$.  So the only possible change is that after composing with $\varphi_n$, $v$ may be added to or removed from $\Sing_{h'}(H)$ and the net crossing of $v$ with $H$ may change.

Suppose $v$ is in $\Sing_h(H)$, but not in $\Sing_{h'}(H)$.  By induction, we know that $m_{h'} \sim_f m_c$.  and as observed above, $v \leq_f m_c $.  Thus, $v \leq_f m_{h'}$, so adding $v$ to $\Sing_{h'}(H)$ does not change its maximal equivalence class and (2) and (3) remain valid.

Next suppose that $v \in \Sing_{h'}(H)$.  If  $\Sing_{h'}(H)$ contains more than one maximal element with non-zero net crossing number, then  removing $v$ from $\Sing_{h'}(H)$  or changing its net crossing number will again preserve properties (2) and (3).

Thus, we need only consider the case where $v$ is the unique element of $\max_{h'}(H)$ with $n_{h'}(v,H) \neq 0$.   Since $n_{h'}(vw,H)= n_{h'}(v,H) + n_{h'}(w,H)$, either $n_{h'}(w,H)$ or $n_{h'}(vw,H)$ must also be non-zero.  In the former case, $w$ lies in $\Sing_{h'}(H)$ and since $v \leq_f w$, this contradicts our assumption that $v$ is the unique maximal element  with non-zero net crossing number.  In the latter case, since $n_h(v,H) = n_{h'}(vw,H) \neq 0$, we conclude that $v$ is also in $\Sing_h(H)$ and its net crossing number remains non-zero, so (2) and (3) still hold for $h$.  This completes the induction.
\end{proof}

\begin{remark} Suppose $v$ is twist-dominant.  Then there are no elements $w\neq v$ with $v \leq_f w$.  It follows that any element of $U^0(\AG)$ takes $v$ to a conjugate of itself or its inverse, so the image in $X$ of an axis for $\bar h_*(v)$ is the same for every marking $h$ as in the lemma.  Moreover, any hyperplane $H$ crossed by this axis has $\max_h(H)=\{v\}$.  \end{remark}

More generally, if we drop the assumption that $h^{-1}c \in U^0(\AG)$, we have

\begin{corollary}\label{cor:mh_independence}  Let $X$ be a $\G$-complex and $h, h'\colon X\to\SaG$ untwisted markings.  Then $[\max_{h'}(H)] = \sigma[\max_{h}(H)]$ for some graph automorphism $\sigma$.
\end{corollary}
\begin{proof} We have $h'=\psi h$ for some $\psi\in U(\AG)$. Write $\psi=\varphi\circ \sigma$, where $\sigma$ is a   graph automorphism and  $\varphi\in U^0(\AG)$.  Then $\max_{\sigma h}(H)=\sigma \max_{h}(H)$, so by Lemma~\ref{lem:Sing}, $[\max_{h'}(H)]=[\max_{\sigma h}(H)]=\sigma [\max_{h}(H)]$.
\end{proof}

\begin{corollary}\label{cor:tmin_characterize} Let $H$ be a hyperplane in a $\Gamma$-complex $X$ and $h\colon X\to\SaG$ an untwisted marking. Then  $H$ is  twist-minimal (respectively   twist-dominant) if and only if  $[\max_h(H)]$ is twist-minimal (respectively twist-dominant).
 \end{corollary}

 \begin{proof} Choose an isomorphism $X\iso \SP$ and let $h=c_\pi$.  Then $H=H_A$ for some $A\in V\cup\bPi$ and  $[\max_h(H_A)]=[\max(A)]$.
 \end{proof}

\begin{definition}Given an untwisted marking $h\colon X\rightarrow \SaG$ and a generator $v\in V$ we define $\Min_h(v)$ to be $\Min(\bar{h}_*(v))\subset \widetilde{X}$. Similarly, we define the branch locus $\Br_h(v)$ to be the set of points in $\Min_h(v)$ whose link in $\widetilde{X}$ strictly larger than the link in $\Min_h(v)$.
\end{definition}

If we choose an identification of $X$ with $\SP$, then in terms of this definition $\Min(v)=\Min_{c_\pi}(v)$ and $\Br(v)=\Br_{c_\pi}(v)$, where $c_\pi\colon \SP\rightarrow \SaG$ is the standard collapse map. Using Lemma \ref{lem:Sing}, we can identify when a hyperplane is contained in the minset for a general untwisted marking.

\begin{proposition}\label{prop:min_h} Let $h\colon X\rightarrow \SaG$ be an untwisted marking,  $H$ a hyperplane of $X$, and $v\in \Sing_h(H)$.  Then $v \in \max_h(H)$ if and only if there is a lift $\widetilde{H}$ contained in $\Min_h(v)$, and in this case, both components of $\partial\kappa(\widetilde{H})$ contain points in $\Br_h(v)$.
\begin{proof}
 It is easily seen that this property is preserved by graph automorphisms, so it suffices to consider the case where $hc^{-1} \in  U^0(A_\G)$ for some collapse map $c$.  Fix an identification of $X$ with  $\SP$ and let $c=c_\pi$.    Then by Lemma \ref{lem:Sing}, for any hyperplane $H$ we have $[\max_h(H)]= [\max_c(H)]$.   Consider the subspace  $\USv= \USnlv \times \USlkv$.
By Lemma \ref{lem:U0PreservesLinks}, since $U^0(A_\G)$  preserves the subgroups $A_{\UF(v)}$ and $A_{\lk(v)}$ up to conjugacy, taking a translate if necessary, we may assume that  {$\USnlv$} contains an axis for $\bar h_*(v)$ (see Corollary \ref{cor:ContainsAxes}).  Call this axis $\alpha^h_v$. Then  {$\USv$} contains $\Min_h(v)= \alpha^h_v \times {\USlkv}$.

By assumption, the axis $\alpha^h_v$ crosses some lift $\widetilde{H}$ of $H$. Let $m \in \max_c(H)$.   If $v$ is not maximal in $\Sing_h(H)$ then $v<_f m$.  In this case,  {$\UH$} is isomorphic to  {$\USlkm$} which is strictly bigger than  {$\USlkv$} hence $\widetilde{H}$ is not contained in $\Min_h(v)$.

Conversely, if $v$ is maximal, then $v \sim_f m$ and we can identify  { $\UH= \USlkv = \USlkm$}.  It follows that the entire carrier $\kappa(\widetilde{H})$ is contained in $\Min_h(v)$.
In addition,  {$\USnlv$} also contains an axis $\alpha_m$ for $m$ with respect to the marking $c$ and $\Min_c(m)= \alpha_m \times  {\USlkv}$.  Since $m \in \max_c(H)$, this minset contains $\kappa(\widetilde{H})$ and by Proposition  \ref{prop:BranchLocus1} , both components of $\partial\kappa(\widetilde{H})$ contain points in the branch locus of $\Min_c(m)$.  Since $\Min_c(m)$ and $\Min_h(v)$ are both metrically the product of  a real line with  {$\USlkv$},  any point in $\partial\kappa(\widetilde{H})$ that is branch for one of these minsets is branch for the other.  Thus
both components of $\partial\kappa(\widetilde{H})$ also contain points in $\Br_h(v)$.
\end{proof}
\end{proposition}

\section{Parallelotope structures on blowups}\label{sec:twistedSP}

In this section we consider blowups $\SP$  as metric objects, where we now allow some of the cubes in $\SP$  to be skewed in certain directions, so that edges spanning a ``cube" are no longer necessarily orthogonal.  We call these {\em skewed blowups}.

An $n$-dimensional   {\em Euclidean parallelotope}  $F$ is a metric space isometric to the image of the unit cube $[0,1]^n \subset\R^n$ under some element of $\GL(n,\R)$.  If $e$ is an edge of   $F$, the {\em midplane} $H_e$   is the convex hull of the midpoints of the edges parallel to $e$. A parallelotope $F$ is  an {\em orthotope} if any two edges at a vertex are orthogonal or, equivalently, the dihedral angle between any two midplanes  is a right angle.

\subsection{Allowable parallelotope structures}

In a blowup $\SP$ every edge $e$ has a label, \textit{i.e.} $e=e_A$ where $A\in V\cup\bPi$.    By Corollary~\ref{fact:commhyp} there is a square in $\SP$    spanned by $e_A$  and $e_B$ if and only if $A$ and $B$   are adjacent. 
We will say $A,B$ are \textit{twist-related} if $ \max(A)\leq_t  \max(B)$ or vice versa.

\begin{definition}\label{def:pre-allowable}  Let $\mathbf c$ be a maximal cube of  $\SP$ with  outgoing edges $e_1,\ldots,e_n$  at a vertex $p\in\mathbf c$.  Let $A_i$ be the label of $e_i$, choose $v_i\in \max(A_i)$ and let $\st^+(v_i) = \{v_i\} \cup \UL(v_i)$.    Given $d_{\mathbf{c}}$ a parallelotope metric on $\mathbf{c}$, we realize $d_{\mathbf c}$ via an embedding $\rho\colon\mathbf{c}\hookrightarrow \R^n$ which sends $p$ to $0$.  Regarding $\rho(e_i)$ as vectors in $\R^n$,   set
\begin{align*}
K_i &= \textrm{ the  subspace of $\R^n$ spanned by $\rho(e_k)$  with $v_k\in\st^+(v_i)$} \hbox { and } \\
  L_{ij} &= K_i \cap K_j.
  \end{align*}
The metric $d_{\mathbf c}$ on $\mathbf c$ is \emph{allowable} if whenever $v_i$ and $v_j$ are not twist-related, then
$$L_{ij}^\perp\cap K_i \hbox{ is orthogonal to   } L_{ij}^\perp\cap K_j.$$
\end{definition}

Note that if $v_i, v_j$ are not twist-related, then $\st^+(v_i) \cap \st^+(v_j) = \UL(v_i) \cap \UL(v_j)$ so in this case, $L_{ij}$ is the  subspace spanned by the $\rho(e_k)$  with $v_k \in \UL(v_i) \cap \UL(v_j)$.

\begin{lemma}\label{lem:angles} An allowable parallelotope metric is determined by edge-lengths and the angles between twist-related edges.
\end{lemma}
\begin{proof} Any parallelotope metric is determined by edge-lengths and the angles between edges, so we must show that in an allowable metric, the angles between non-twist-related edges are determined by those between twist-related edges.

Suppose $e_i$ and $e_j$ have labels that are not twist-related, so $L_{ij}$ is the span of edges $e_k$ with $v_k\in\UL(v_i)\cap \UL(v_j)$.     The proof is by induction on $\dim(L_{ij})$.  If $\dim(L_{ij})=0$ the angle between $e_i$ and $e_j$ must be $\frac{\pi}{2}$.  If $\dim(L_{ij})>0$, write $e_i=e_i^\prime + \ell_i$,
where $\ell_i$ is the orthogonal projection of $e_i$ onto $L_{ij}$ and $e_i^\prime$ is the projection onto $L_{ij}^\perp\cap K_i$. Similarly, write $e_j=e_j^\prime+\ell_j$.  Then $e_i\cdot e_j=(e_i^\prime+\ell_i)\cdot(e_j^\prime+\ell_j)=\ell_i\cdot\ell_j$. Since $\ell_i$ and $\ell_j$ are linear combinations of  the $e_k\in L_{ij}$, this dot product is determined by the dot products of these $e_k$. The dot products $e_k\cdot e_k$ are the squares of the lengths of the $e_k$, which are given. If $v_k\in \UL(v_i)\cap \UL(v_j)$ then $\UL(v_k)$ is strictly contained in $\UL(v_i)$, so for two edges $e_k\neq e_l\in L_{ij}$, the subspace $L_{kl}$ has dimension strictly smaller than $\dim(L_{ij})$.  Thus, by induction, the angle between edges in $L_{ij}$ is determined by edge-lengths and the angles between twist-related edges, so the same holds for $e_i,e_j$.
\end{proof}

\begin{definition} \label{def:allowable}An \emph{allowable parallelotope structure} $\F$ on $\SP$ is an assignment of a parallelotope metric to each cube $\mathbf c$ of $\SP$ such that
\begin{enumerate}
\item the metric on each maximal cube is allowable,
\item If $\mathbf c'$ is a face of $\mathbf c$, then the metric on $\mathbf c'$ is the restriction of the metric on $\mathbf c$, and
\item If $\max(A)=\{v\}$ is twist-dominant, then for any $B$    adjacent to  
$A$, the angle between $e_A$ and $e_B$ is equal to the angle between $e_v$ and $e_B$.
\end{enumerate}
\end{definition}
The parallelotope structure  in which every $k$-cube is isometric to the Euclidean cube $[0,1]^k$ is clearly allowable; it will be called the {\em standard structure} and denoted $\E$. If all parallelotopes in $\F$ are orthotopes, the structure $\F$ will be called {\em rectilinear}. These too are clearly allowable.

If $A$ and $B$ are  adjacent 
labels in $\bPi\cup V$, then there is at least one parallelogram $F\in \F$ with edges labeled $e_A$ and $e_B$.  If $\F$ is allowable, then condition (2) guarantees that the angle between these edges is the same for any such $F$, and we will denote this angle by $\alpha_{A,B}$.   By Lemma \ref{lem:angles}, the entire structure $\F$ is determined by the lengths of the edges $e_A$ and the angles $\alpha_{A,B}$ for twist-related $A,B$.

An allowable parallelotope structure $\F$  induces a (path) metric $d_\F$ on $\SP$.  Different parallelotope structures may induce the same metric on $\SP$; for example if $\SP=\SaG$ is an $n$-torus consisting of a single parallelotope $F$ with sides identified, then changing $F$ by any element of $\GL(n,\Z)$ results in the same metric $d_\F$.

Note that an edge-path which was convex in the standard cube complex structure $(\USP,d_\E)$ is no longer necessarily convex in the metric space $(\USP,d_\F)$. We define a {\em hyperplane}  $H_A$ in $(\SP,\F)$ to be the set of midplanes dual to edges with  label $A$.  This is the usual notion of hyperplane if $\F=\E$, but for arbitrary $\F$ lifts of hyperplanes are no longer necessarily convex in  $(\USP,d_\F)$.

\subsection{Rotating a hyperplane in $\SP$}\label{sec:rotating}

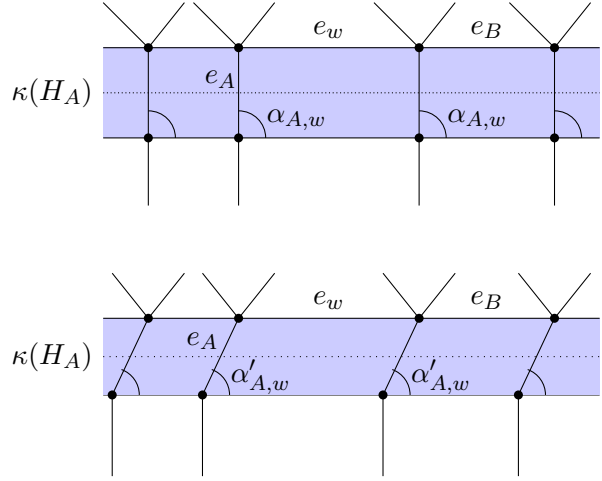
\begin{figure}
\begin{center}
    \begin{tikzpicture}[scale=1.2]
    \fill[blue!20] (-1.5,0) to (4,0) to (4,1) to (-1.5,1) to (-1.5,0);
        \draw (-1.5,0)  to (4,0);
             \draw (-1.5,1)  to (4,1);
                 \draw[densely dotted] (-1.5,.5) to (4,.5);
    \foreach \x in {-1,0,2,3.5}
    {
     \fill[black] (\x,0) circle (.05);
     \fill[black] (\x,1) circle (.05);
     \draw (\x,0) to  (\x,1);
     \draw (\x,0) to (\x,-.75);
     \draw (\x,1) to (\x-.5,1.5);
          \draw (\x,1) to (\x+.4,1.5);
          \draw ([shift=(0:.3cm)] \x,0) arc (0:90:.3);
   }
        \node (ew) at (1,1.2) {$e_w$};
               \node (eB) at (2.75,1.2) {$e_B$};
    \node (e) at (-.2,.65) {$e_A$};
     \node [left] (H) at (-1.5,.5) {$\kappa(H_A)$};
          \node [right] (alpha0) at (.2,.2) {$\alpha_{A,w}$};
                  \node [right] (alpha01) at (2.2,.2) {$\alpha_{A,w}$};
   \begin{scope} [yshift=-3cm]

  \draw (-1.5,.15)  to (4,.15);
    \fill[blue!20] (-1.5,.15) to (4,.15) to (4,1) to (-1.5,1) to (-1.5,.15);
             \draw (-1.5,1)  to (4,1);
                 \draw[dotted] (-1.5,.575) to (4,.575);
    \foreach \x in {-1,0,2,3.5}
    {
     \fill[black] (\x-.4,.15) circle (.05);
     \fill[black] (\x,1) circle (.05);
     \draw (\x-.4,.15) to  (\x,1);
     \draw (\x-.4,.15) to (\x-.4,-.75);
     \draw (\x,1) to (\x-.4,1.5);
          \draw (\x,1) to (\x+.4,1.5);
          \draw ([shift=(0:.3cm)] \x-.4,.15) arc (0:70:.3);
   }
     \node (eA) at (-.4,.75) {$e_A$};
       \node (eB) at (2.75,1.2) {$e_B$};
      \node (ew) at (1,1.2) {$e_w$};
      \node [left] (H) at (-1.5,.575) {$\kappa(H_A)$};
        \node [right] (alpha) at (-.2,.35) {$\alpha'_{A,w}$};
           \node [right] (alpha) at (1.8,.35) {$\alpha'_{A,w}$};
\end{scope}
    \end{tikzpicture}
  \caption{Rotating  $H_A$ in the direction $w\in\UL(A).$ Here  $B$ splits $w$}\label{fig:rotation}
  \end{center}
  \end{figure}
  \begin{definition} Suppose $(\SP,\F)$ is an allowable parallelotope structure.  Let $A\in \bPi\cup V$, $H_A$ the hyperplane in $\SP$ labeled $A$, $v\in \max(A)$ and $w \in \UL(v)$. Then {\em rotating $H_A$ in the direction of $w$} means changing the angle $\alpha_{A,w}$ to $\alpha'_{A,w}$, so that for every $B$ that splits $w$, the angle between the edges $e_A$ and $e_B$ is $\alpha'_{A,w}$.  More generally, {\em rotating $H_A$} means rotating it in the direction of one or more $w \in  \UL(v)$.  The length of the edge $e_A$ remains unchanged under rotation.
 \end{definition}

Rotating a hyperplane $H_A$ in an allowable paralleotope structure $\F$ gives rise to a new parallelotope structure $\F'$ which still satisfies the first two conditions for allowability.  This is because the subspaces $K_i$ in the definition of an allowable parallelotope are unchanged by the rotation.  However,  if  $A$ is twist-dominant, then to achieve the third allowability condition, one needs to do comparable rotations to every hyperplane $H_{A'}$ with $ \max(A')= \max(A)=\{v\}$.

Recall that we have a partial ordering on equivalence classes in $V$ given by $[v] \leq [w]$ if $\lk(v) \subseteq \st(w)$.  Choose a total ordering $\prec$ on $V$ consistent with this partial ordering.  Given a compatible collection of partitions $\bPi$,  we can extend this to a total order on  $\bPi \cup V$  satisfying $[\max(A)] < [\max(B)] \Rightarrow A \prec B$.

\begin{proposition}\label{prop:shearing} Every allowable parallelotope structure on $\SP$ can be obtained from an orthotope structure on $\SP$ by a sequence of rotations.
\end{proposition}

\begin{proof}  Suppose $\F$ is an allowable parallelotope structure on $\SP$ and let $\alpha_{A,B}$ denote the angle between edges $e_A, e_B$ for any  adjacent 
 pair $A,B$.  Let $\F_0$ denote the rectilinear structure with the same edge lengths as $\F$.  Using the total order $\prec$, we will rotate the hyperplanes in $\F_0$ in descending order and show inductively that after rotating $H_A$, we get a parallelotope structure on $\SP$ satisfying
\begin{itemize}
\item[(i)] The metric on each paralleotope is allowable and agrees on common faces.
\item[(ii)] For all $B,C\succeq A$ such that $ \max(B) \leq_t  \max(C)$, the angle between $e_B$ and $e_C$ equals $\alpha_{B,C}$.
\end{itemize}

Say by induction that we have rotated all the hyperplanes $H_{A'}$ with $A \prec A'$.  Rotating $H_A$ only changes the angles between $e_A$ and other edges.  By induction, condition (ii) is already satisfied whenever $B,C\succ A$.  We now rotate $H_A$ so that condition (ii) also holds when $A=B$, i.e. when $A \prec C$ and $ \max(A) \leq_t  \max(C)$.  As observed above, rotating preserves allowability of individual parallelotopes, and by definition it agrees on common faces, so condition (i) continues to hold.

 At the end of this process, when we have rotated all the hyperplanes as needed, we arrive at a parallelotope structure in which the angles between any two edges $e_A,e_B$ with $ \max(A) \leq_t  \max(B)$,  agree with those in $\F$.  This implies that this structure also satisfies the third condition for allowability.  So by Lemma \ref{lem:angles}, it must in fact be equal to $\F$.
\end{proof}

\begin{proposition} \label{prop:CAT0}  Let $\F$ be an allowable parallelotope structure on $\SP$ and suppose the induced path metric $d_\F$ is locally CAT(0).  Suppose $\F'$ is obtained from $\F$ by a hyperplane rotation. Then
\begin{enumerate}
\item  $d_{\F'}$ is also locally CAT(0).
\item Any twist-minimal hyperplane which is locally convex with respect to $d_{\F}$ remains locally convex with respect to $d_{\F'}$.
\end{enumerate}
\end{proposition}
\begin{proof}  The local geometry at a point $p$ is determined by the geometry of its link. Thus, it suffices to show that rotating a single hyperplane $H$ does not change the isometry type of links in $\SP$.  The carrier $\kappa(H)$ either has two boundary components, each isometric to $H$, or (if the dual edge is a loop in $\SP$) these boundary components may be identified to each other.  In either case, we will denote (the image of) this boundary by $\partial \kappa(H)$ and the interior by $\kappa^\circ (H)=\kappa(H)-\partial \kappa(H)$.  Setting $Y=\SP - \kappa^{\circ}(H)$, we have
$$ \SP = \kappa(H) \cup_{\partial \kappa(H)} Y.$$
Rotating $H$ changes the parallelotope structure only on cubes meeting the interior of $\kappa(H)$, leaving  those in $\partial \kappa(H)$ and $Y$ unchanged. Hence, it suffices to show that if $x$ is a vertex lying in $\partial \kappa(H)$, then the rotation does not change the induced metric on the link of $x$ in $\kappa(H)$  (though it may change the metric on individual simplices in that link).

To see this, note that if $H$ is dual to $e_A$ with $v\in\max(A)$, then the carrier of $H$ decomposes as
$\kappa(H) = e_A \times \Slkv$,
and by Proposition \ref{prop:generators}, $\Slkv$ contains a characteristic cycle for every $w \in \lk(v)$. Consider the subspace $\Sulv \subset \Slkv$ spanned by the characteristic cycles for $w \in \UL(v)$.  Elements of $\UL(v)$ commute and are twist-dominant, so this subspace is a torus with a flat metric.   Moreover, as elements of $\UL(v)$ commute with every element of $\lk(v)$ we have a further (combinatorial) decomposition $\kappa(H_A)=e_A\times \Sulv\times \mathbb{K}_{\lk(v)\setminus \UL(v)}$. The edge $e_A$ can only rotate in the direction of $\Sulv$.  Thus, viewing $e_A \times \Sulv$  geometrically as an interval cross a torus, this rotation changes only the width of the interval.  In particular, the rotation does not change the local geometry of $\kappa(H)$.  This proves (1).

For (2), let $L$ be any twist-minimal hyperplane of $\SP$ and $p$ be a point of $L$.  Then as was just shown,  the local metrics at $p$ with respect to $d_\F$ and $d_{\F'}$ are the same.  Since $L$ is twist-minimal, it is preserved setwise by rotation. Thus if $L$ was locally convex before rotation, it remains locally convex afterward.
\end{proof}

 \begin{corollary}\label{cor:CAT(0)}
If $\F$ is an allowable parallelotope structure on $\SP$ then the induced path metric $d_\F$ is locally CAT(0).
 \end{corollary}
 \begin{proof} This follows from Propositions~\ref{prop:CAT0}(1) and  \ref{prop:shearing}, since any orthotope structure is CAT(0) by Gromov's link condition.
  \end{proof}

As noted above, subcomplexes which are locally convex in $(\SP,\E)$ may no longer be  convex in a general allowable parallelotope structure $(\SP,\F)$.  The following lemma specifies two  exceptions that will be important in the sequel.

\begin{lemma}\label{lem:convex} Suppose $\F$ is an allowable parallelotope structure on $\SP.$
\begin{itemize}
\item  Let $v\in V$ be twist-dominant.  Then any lift of a characteristic cycle for $v$ in $\SP$ is convex in $(\USP,d_\F).$
\item Let $A\in V \cup \bPi$ be a label with $v\in \max(A)$.  If $v$ is twist-minimal, then any lift of the hyperplane $H_A$ is convex  in $(\USP,d_\F).$
\end{itemize}
\end{lemma}

\begin{proof} If $v$ is twist-dominant, then condition (3) in Definition \ref{def:allowable} guarantees that consecutive edges in a characteristic cycle for $v$ have angle $\pi$ in $(\SP,d_\F)$. Since $(\SP,d_\F)$ is locally CAT(0), the lift of the characteristic cycle to $\USP$ is geodesic and convex. The second statement follows from Propositions~\ref{prop:shearing} and \ref{prop:CAT0}(2).
\end{proof}

  \subsection{Straightening an allowable parallelotope structure}\label{sec:straightenSP}

 In this section we show how to straighten an allowable parallelotope structure  $\F$ on $\SP$ to obtain an orthotope structure, while maintaining allowability throughout the process.
  \begin{remark}\label{rmk:MaxOfEdge}It will be convenient to describe the straightening process in terms of what it does to the edges of $\SP$, rather than its dual hyperplanes.  In particular, if an edge $e_A$ is dual to a hyperplane $H_A$ and $m\in \max(A)$ we will say that $m$ is a \emph{maximal element of $e_A$}.\end{remark} 

We begin by straightening a single parallelotope $F\in \F$. The straightening procedure for  $F$ will depend only on the equivalence classes $[ \max(A)]$ of the edges $e_A$ in $F$.  Therefore, it suffices to describe the straightening process in the case where all edges are labeled $e_v$ for some $v \in V$.

Fix a vertex $x$ in $F$ with all angles acute or right.  Let $E$ be the set of edges emanating from $x$.  We can view $E$ as a set of $n$ linearly independent vectors in the positive orthant of $\R^n$.
Let $\prec$ be a total ordering on $V$ as described in Section~\ref{sec:rotating}.  For each edge $e_v$,  define subspaces
\begin{align*}
K_{v} &=  \textrm{span of } \{e_w \in E \mid w \in \st^+(v)\}\\
K_{v}^{\prec} &= \textrm{span of } \{e_w \in E \mid w \in \st^+(v), v \prec w \}
\end{align*}

For $e_v,e_w \in E$, set $L_{v,w}= K_v \cap K_w$.  Recall that $F$ is allowable if whenever $v, w$ are not twist-related, $K_v \cap L_{v,w}^\perp$ is orthogonal to $K_w \cap L_{v,w}^\perp$.

Now define a new basis $\{b_v\}$ for  $\R^n$ as follows.  For each $v$,  let $b_v$ be the unit normal vector to $K_v^\prec$ in $K_v^\prec \oplus \langle e_v \rangle$.  In the case where $K_v^\prec$ is empty, $b_v$ is just the unit vector in the direction of $e_v$.
With respect to this basis, we have
$$e_v = r_vb_v  + \sum_w r_{v,w} b_w$$
for some $r_v >0,  r_{v,w}  \geq 0$, where the sum is taken over all $w$ with $e_w \in  K_{v}^{\prec} $.
 In particular, this set of vectors $\{b_w \}$ is also a basis for $K_v^\prec$.

\begin{lemma}  $F$ is allowable if and only if for any two edges $e_v,e_w \in E$,  $b_v$ is orthogonal to $b_w$. That is, the vectors $\{b_v\}$ span an orthotope.
\end{lemma}

\begin{proof} Assume $F$ is allowable.  Suppose $v$ and $w$ are twist-related and say $v \prec w$.   In this case,
$K_w^\prec \oplus \langle e_w \rangle  \subset K_v^\prec$, so by definition, $b_w$ lies in $K_v^\prec$ and hence it is orthogonal to $b_v$. (This is always true, even without assuming allowability.)

So now suppose $v$ and $w$ are not twist-related.  Then
$$L_{v,w} = K_v \cap K_w = K_v^\prec \cap K_w^\prec $$
since any $u \in \st^+(v) \cap \st^+(w)$ must be strictly greater than either $v$ or $w$ with respect to the ordering $\leq_t$, and hence also with respect to $\prec$.
Since  $b_v \in (K_v^\prec)^\perp \subset L_{v,w}^\perp$, and  $b_w \in (K_w^\prec)^\perp \subset L_{v,w}^\perp$,  the allowability condition implies that $b_v$ and $b_w$ are orthogonal.

Conversely, assume the all of the $b_*$ vectors are orthogonal to each other.  For $v,w$ not twist-related, a basis for  $K_v \cap L_{v,w}^\perp$ is given by the set of $b_u$ with $v \leq_t u$ and $w \nleq_t u$, and similarly, a basis for $K_w \cap L_{v,w}^\perp$ is given by the set of $b_z$ with $w \leq_t z$ and $v \nleq_t z$.  These sets are disjoint, and any two such $b_u$ and $b_z$ are orthogonal, so  $K_v \cap L_{v,w}^\perp$ is orthogonal to $ K_w \cap L_{v,w}^\perp$ as required.
\end{proof}

Next we describe a process for straightening $F$.  For $t \in [0,1]$, set
$$e^t_v = s_t (r_vb_v  + t \sum_{w} r_{v,w} b_w)$$
where $s_t \in \R^+$ is chosen so that $\|e^t_v\|=\|e_v\|$.  Then $e^1_v = e_v$ and $e^0_v=\|e_v\| b_v$.
Let $F^t$ be the parallelotope spanned by $\{e_v^t\}$.

\begin{lemma}  If $F$ is allowable, then $F^t$ is allowable for all $t \in [0,1]$, and  $F^0$ is an orthotope.
\end{lemma}

\begin{proof}  At all times $t$, the subspaces $K_v^\prec(F^t)$  remain unchanged, that is,
$K_v^\prec(F^t)=K_v^\prec(F)$ for all $t$, and likewise for $K_v^\prec(F^t) \oplus \langle e^t_v \rangle$.
  Hence the normal vectors $b_v$ remain fixed throughout the process. By the previous lemma, $F$ is allowable if and only if all of the $b_v$ vectors are orthogonal, or equivalently $F^0$ is an orthotope.  \end{proof}

We now want to apply the straightening procedure simultaneously to all parallelotopes in $\F$.   Suppose two maximal parallelotopes $F$ and $F'$ share a face $F_0$ in $\SP$.  If $e_A$ is an edge lying in $F_0$, with $v \in \max(A)$ then for any $w$ with $v <_tw$, $F$ and $F'$ must each contain an edge with maximal element $w$.  Since $w$ is twist-dominant, the allowability condition implies that these edges both lie along an axis for $w$, hence they are parallel.   Since the straightening procedure on $e_A$ depends only on these edges, it follows that the straightening  in $F$ and $F'$ agree on this face. Moreover, the same argument applied to the edges with maximal element $w$ shows that these edges remain parallel throughout the straightening process.  Thus, we obtain a consistent straightening, $\F^t$, of the entire complex which remains allowable at all times $t$.  We call $(\SP,\F^t)$ the \emph{straightening path} for $(\SP,\F)$.

\section{The space of skewed $\G$-complexes with untwisted markings}\label{sec:tS}

We are now ready to define a space $\tS$ of skewed $\G$-complexes with untwisted markings, that serves as an intermediary between $\SG$ and the full outer space $\OG$.

\subsection{Skewed $\G$-complexes}

Let $X$ be a $\G$-complex and $\F$ a parallelotope structure on $X$.  Define $\F$ to be {\em allowable} if there is some isomorphism $ \SP\iso X$ such that the pullback  of $\F$ is an allowable parallelotope structure on $\SP$.

\begin{lemma} Allowability of a parallelotope structure on $X$ is independent of the choice of isomorphism $\SP\iso X$.
\end{lemma}
\begin{proof}
By Corollary~\ref{cor:independent} the twist-relation is independent of the  isomorphism $\SP\iso X$. Hence if $X$ is isomorphic to both $\SP$ and $\Sa^{\bPi'}$, then  Properties 1, 2 and 3 in Definition~\ref{def:allowable}  are satisfied by the pullback structure on $\SP$, if and only if they are satisfied for the pullback structure on $\Sa^{\bPi'}$, showing that allowability is also independent of the isomorphism.
\end{proof}

\begin{definition} A {\em skewed $\G$-complex} is a $\G$-complex $X$ together with an allowable parallelotope structure $\F$.
If all of the parallelotopes $F\in\F$ are orthotopes, we will call $(X,\F)$ a  {\em rectilinear $\G$-complex}, and if all parallelotopes are isometric to $[0,1]^k$ we will write $\F=\E$ and call $(X,\E)$ a  {\em standard $\G$-complex}.
\end{definition}

\subsection{Definition of  $\tS$} We now add untwisted markings to skewed $\G$-complexes to form a space $\tS$.

\begin{definition} A \textit{marked}, skewed $\G$-complex is a triple $(X,\F,h),$ where $(X, \F)$ is a skewed $\G$-complex  and $h\colon X\rightarrow \SaG$ is an untwisted homotopy equivalence, i.e.   for any $\G$-collapse map $c:X\rightarrow \SaG$, the composition $c\circ h^{-1}:\SaG\rightarrow \SaG$ induces an element of $U(A_\G)$ (where $h^{-1}$ is a homotopy inverse to $h$).  Two marked, skewed $\G$-complexes $(X,\F,h)$ and $(X',\F',h')$ are equivalent if there is a combinatorial isometry $i: (X,\F) \to (X',\F')$  (i.e. a map which preserves both the combinatorial structure and the metric on each parallelotope)  that commutes with the markings up to homotopy, i.e., $h \simeq h' \circ i$.
\end{definition}

The space $\tS$ is the space of equivalence classes of marked skewed $\G$-complexes with untwisted markings:
$$\tS = \{ \textrm{marked, skewed $\G$-complexes $(X,\F,h)~ |~h$ is untwisted} \} / \sim.$$
We will denote the equivalence class of $(X,\F,h)$ by $[X,\F,h]$.

Given a $\G$-complex $X$ and untwisted marking $h : X \to \SaG$,  let $U_{X,h}$ denote the subset of $\tS$ obtained by equipping $X$ with all possible allowable parallelotope structures,  i.e.
$$U_{X,h} = \{ [X, \F,h] \in \tS\} $$
We will call this a ``cell" in $\tS$. It comes equipped with a natural topology as a subspace of a Euclidean space determined by the parallelotopes in $\F$ and subject to the allowability conditions in Definition \ref{def:allowable}. Metrically, collapsing a hyperplane in $X$ corresponds to letting the length of the dual edges go to zero. The closure of $U_{X,h}$ thus consists of the cells $U_{X',h'}$ such that there exists a hyperplane collapse map $c: X \to X'$ with $h$ homotopic to  $h' \circ c$.  The topology on $\tS$ is therefore determined as a complex of spaces comprised of the cells  $U_{X,h}$, where cells are identified by collapse maps as just described. For a more detailed description of complexes of spaces and their properties, see \cite[Chapter 4.G]{Hat}.

\subsection{Contractibility of $\tS$}\label{Sec:contractibility} We will show that $\tS$ is contractible by finding a deformation retraction of $\tS$ onto the subspace of rectilinear marked $\G$-complexes; this is the space   $\SG$ defined in Section~\ref{subsec:untwisted}, which we know is contractible.  In other words, we want to find a way to straighten marked, skewed $\G$-complexes  in a way that maintains allowability and extends to a continuous straightening of the whole of $\tS$.

In order to straighten a skewed $\G$-complex $(X,\F)$ we choose an identification of $X$ with $\SP$ for some $\bPi$ and apply the straightening process described in Section~\ref{sec:straightenSP}.   We need to show that this is independent of the isomorphism $X\iso \SP$.  We note that the labeling on  $\SP$ was used in the straightening process only to order the edges in $K_v$.
By Corollary \ref{cor:independent}, any combinatorial isomorphism $i: \SP \to \Sa^{\bPi'}$ preserves the twist ordering $\leq_t$ on edge labels, so in fact we need only be concerned about what it does to the ordering $\prec$ within each twist equivalence class.

To address this problem, we will need to choose preferred representatives for points in  $\tS$.  Let $X$ be a $\G$-complex and $h\colon X\to \SaG$ an untwisted marking.  By Lemma~\ref{lem:Gcollapse} there exists a blow-up $\SP$ and  an isomorphism of cube complexes  $i: X \to \SP$ such that $c_\pi \circ i \circ h^{-1} \in U^0(\AG)$.   Suppose $j\colon X\to {\SPp}$ is another such isomorphism. 

\begin{lemma}\label{lem:GC2} Let $i\colon X\to \SP$ and $j\colon X\to {\SPp}$ be as above.  For any twist-dominant $v$,  $j \circ i^{-1}$ takes edges with maximal element $v$ to edges with maximal element $v$  (cf. Remark \ref{rmk:MaxOfEdge}).
\end{lemma}
\begin{proof}
For  $i\colon X\to \SP$ and $j\colon X\to {\SPp}$  as above, the composition $c_{\omega} \circ j \circ i\inv\circ c_\pi\inv$ induces an element of $U^0(\AG)$.
 Since any element of $U^0(\AG)$ takes every twist-dominant generator $v$ to a conjugate of itself,
 the map $j \circ i^{-1} : \SP \to \SPp$ takes an axis for $v$ in $\USP$ (with respect to the standard metric) to an axis for $v$ in $\widetilde{\Sa}^\Omega$.  Edges with maximal element $v$ lie on such an axis, thus they map to edges with the same maximal element.
\end{proof}

We can now define the deformation retraction $ R_t: \tS \to \SG$ as follows.  Let $(X,\F,h)$ represent a point in $\tS$ and choose a cubical isomorphism $i : X \to \SP$ as in Lemma \ref{lem:GC2}.  Using this isomorphism, we can identify paralleotope structures on $\SP$ with parallelotope structures on $X$.  Thus the straightening path for $(\SP, \F)$ gives a path in $\tS$ defined by $R_t [X,\F,h] = [X,\F^t,h]$.

\begin{lemma}\label{lem:continuity} The deformation retraction  $ R_t: \tS \to \SG$ is  well-defined and  continuous.
\end{lemma}
\begin{proof}
 The straightening path depends only on which edges in the parallelotope structure are twist-dominant.  If $i\colon X\to \SP$ and $j\colon X\to {\SPp}$ are two identifications of $X$ with blowups then by Lemma \ref{lem:GC2}, $j\circ i^{-1}$ takes twist-dominant edges to twist-dominant edges, so for each $t$ the straightening path induced by $i$ is isometric to the straightening path induced by $j$. 

 It is clear from the definition of the straightening path that $R_t$ is continuous on each cell $U_{X,h}$ of $\tS$.  It suffices to show that $R_t$ is also continuous on the closure of each cell. The closure of $U_{X,h}$ consists of all the cells $U_{X',h'}$ such that there exists a collapse map $c: X \to X'$
with $h$ homotopic to $h' \circ c$.  Since straightening paths preserve edge lengths, a path $[X,\F^t,h]$ in $U_{X,h}$ will collapse to a path in $U_{X',h'}$ when the appropriate edge lengths go to zero.  Moreover, since the straightening paths in every cell are defined using the same ordering $\prec$ on $V$, this path will agree with $R_t$ on $U_{X',h'}$.
\end{proof}

In light of Corollary~\ref{cor:SGcontractible} we conclude:

\begin{corollary}\label{cor:tScontractible}  The space  $\tS$ is contractible.
\end{corollary}

\section{Outer space $\OG$}\label{sec:OG}

\subsection{Definition of $\OG$ and the map   $\Theta\colon\tS\to \OG$}

We now  define a new space $\OG$  by forgetting the combinatorial structure  on skewed $\G$-complexes and allowing arbitrary markings.  Thus a point in $\OG$ is an equivalence class of triples $(Y, d, f)$ such that
 \begin{itemize}
 \item $(Y,d)$ is a locally CAT(0) metric space that is isometric to  $(\SP, d_\F)$ for some skewed blowup $(\SP,\F)$.
 \item $f\colon Y\to \SaG$ is a homotopy equivalence, and
 \item $(Y,d,f)\sim (Y',d',f')$ if there is an isometry $i:(Y,d)\to (Y',d')$  with $f'\circ i\simeq f$.
\end{itemize}
The full group $\Out(\AG)$ acts on the left on $\OG$ by changing the marking $f$.

\begin{proposition}
The action of $\Out(\AG)$ on $\OG$ has finite stabilizers. 
\end{proposition}
 
\begin{proof} The element of $\Out(\AG)$ induced by a homotopy equivalence $g:\SaG \to \SaG$ fixes the point $[Y,d,f]$ if and only if $f^{-1} \circ g \circ f$ is homotopic to an isometry of $(Y,d)$.  Thus, the stabilizer of a point $[Y,d,f]$ can be identified with the group of isometries of $Y$ up to homotopy.   

 Since $(\SP, \F)$ has no free faces, each $(Y,d)$ has the geodesic extension property.  It follows from \cite{BH} Lemma II.6.16 that the min set of the center of $A_\G$ is all of $Y$, so by the Flat Torus Theorem (Theorem II.7.1)  $Y$ splits as a product $Y=Y_0\times T_{\mathcal Z(\AG)}$, where $T_{\mathcal Z(\AG)}$ is a torus  of dimension equal to the rank of the center $\mathcal Z(\AG)$.  Moreover, by Theorem II.6.17 of \cite{BH},  $\Isom(Y)$ is a topological group with finitely many components, and the connected component of the identity is generated by translations of $T_{\mathcal Z(\AG)}$. As every such translation is homotopic to the identity, the group of isometries of $Y$ up to homotopy is a quotient of the group of path components of $\Isom(Y)$, hence finite as claimed. 
 \end{proof}
 
  In fact, as shown by Bregman in \cite{Breg}, the group of path components of $\Isom(Y)$ injects into $\Out(\AG)$. 

To finish the proof  of Theorem~\ref{thm:main} we need to show that $\OG$ is contractible.  To do this, we 
define a map $ \Theta\colon \tS \to \OG $
 by forgetting the parallelotope structure on $X\in \tS$ and just viewing it as a CAT(0) metric space. The remainder of this section is devoted to  proving the following theorem.

\begin{theorem} \label{thm:Theta} The map $\Theta\colon\tS \to \OG$  is a fibration with contractible fibers. Hence $\OG$ is contractible.
\end{theorem}

  Since the inclusion map  $\SG \hookrightarrow \tS$ is a homotopy equivalence by Lemma~\ref{lem:continuity}, the map
$$  \SG  \hookrightarrow  \tS  \xrightarrow{\Theta} \OG $$
that forgets the orthotope structure on $X$ is also a homotopy equivalence.  We will show in Corollary \ref{embedding} below, that this map is an embedding.

\begin{corollary}  The restriction of $\Theta$ to  $\SG$ is a homotopy equivalence $\SG \simeq \OG$.\end{corollary}

 The proof of Theorem \ref{thm:Theta} has two major components.  The first is to show that the map $\Theta$ is surjective.  This is by no means obvious since the markings in $\tS$ must be untwisted, whereas the markings in $\OG$ are unrestricted.  Finding a point in the fiber over some $(Y,d,f) \in \OG$ means finding a skewed blowup structure $(\SP,\mathcal F)$ on $Y$ such that $f^{-1}$ followed by the standard collapse map  is untwisted.  To do this we  first decompose $Y$ into parallelotopes, then identify the \GW\ partitions in the blowup structure, and finally calculate the composition  $c_\pi\circ f^{-1}$.

 The second component of the proof is to show that the fibers are contractible.  To do this, we fix a point in the fiber and describe a process of ``shearing" edges dual to a hyperplane in this $\G$-complex.  We then prove that every point in the fiber can be obtained by a series of ``zero-sum shearings" of the initial point.  This set of shearings spans a linear subspace of a Euclidean space, hence is contractible.  

\subsection{Surjectivity of $\Theta$}\label{sec:Surjectivity}

The first step in proving Theorem~\ref{thm:Theta} is to show that the inverse image of an arbitrary point in $\OG$ is non-empty.
\begin{proposition}\label{prop:surjective} $\Theta\colon \tS\to\OG$ is $U(\AG)$-equivariant and surjective.
\end{proposition}

Equivariance under the action of $U(\AG)$ is clear from the definition of $\Theta$ whereas surjectivity is not, since markings in $\OG$ can differ by any element of $\Out(\AG)$. The key is to show that an appropriate change of skewed blowup structure on a point of $\tS$  will have the effect of composing the collapse marking with a twist.  The proof of Proposition \ref{prop:surjective} will occupy the remainder of this subsection.

For skewed blowups the end result of the retraction $R_t$ defined in Section \ref{Sec:contractibility} followed by scaling the edge lengths linearly gives a continuous ``straightening map"  $s_\F\colon (\SP,\F)\to (\SP,\E)$  that sends each parallelotope to a unit cube.  The standard collapse map $c_\pi\colon\SP\to \SaG$ induces a collapse map $c_\pi^\F=c_\pi\circ s_\F$ on $(\SP, d_\F),$ called a {\em straighten-collapse} map.

\begin{definition}\label{def:realized}  An automorphism $\phi\in\Out(\AG)$ is {\em realized} by an isometry $i\colon (\SP,d_\F)\to(\SPp,d_\Fp)$ if $c^\Fp_{\omega}\circ i\circ (c^\F_\pi)\inv$ induces $\phi$ on $\pi_1(\SaG)=\AG$:  
 \begin{center}\begin{tikzcd}[]
(\SP,d_\F)   \arrow[r, "i"] \arrow[d, "c^\F_\pi"] &( \SPp,d_\Fp) \arrow[d, "c^\Fp_{\omega}"]\\
\SaG \arrow[r,"\phi"] & \SaG
 \end{tikzcd}
 \end{center}
 \end{definition}
 Note that we are not requiring $i$ to be a {\em combinatorial} isometry, just an isometry.  The realization of a combinatorial isometry is always untwisted.     In Figure ~\ref{fig:one} we illustrate an isometry between two skewed blowups  that realizes  an elementary twist $v\mapsto vw$; one should think of these blowups as giving two different parallelotope decompositions of the same space, and the isometry as the identity.   The following lemma explains in general how to realize a twist $v\mapsto vw$ in the case that $v$ is twist-minimal.

\begin{figure}
\begin{center}
  \begin{tikzpicture}
  \begin{scope}[decoration={markings,mark = at position 0.5 with {\arrow{stealth}}}]
 \coordinate (iw) at (-1.5,-1);\coordinate (mw) at (-.5,-1); \coordinate (tw) at (.5,-1);
        \midarrow [thick, blue] (-1,-2) to (1,-2); \node [above] (b1) at (0,-2) {$w$}; 
  \midarrow [thick, blue] (-1,0) to (1,0); \node [above] (b2) at (0,0) {$w$};  
   \midarrow [thick, blue] (-1,2)  to (1,2); \node [above] (b3) at (0,2) {$w$}; 
   \midarrow [thick] (iw) to (tw);  \node [above] (b4) at (mw) {$w$};  
 \draw [thick, red] (-1,0) to (iw);  \node () at (-1,-.5) {$\WP$}; 
\draw [thick, red] (tw) to (1,0);  \node () at (1,-.5) {$\WP$};  
   \midarrow [thick, red] (1,-2) to (tw);  \node () at (1,-1.5) {$v$};  
 \midarrow [thick, red] (-1,-2) to (iw);  \node  () at (-1,-1.5) {$v$}; 
       \midarrow[thick, forestgreen] (1,0) to (1,2);\node [right] (ctop) at (1,1) {$u$};  
 \midarrow [thick, forestgreen] (-1,0) to (-1,2);     \node [left] (cbot) at (-1,1) {$u$};  
  \midarrow  [thick] (tw) .. controls (.25,3) and (3,1) ..   (1,0);  
    \node (dtop) at (.5,1.5) {$z$};

\begin{scope}[xshift=4cm]. 
 \coordinate (iw) at (-1.5,-1);\coordinate (mw) at (-.5,-1); \coordinate (tw) at (.5,-1);
        \draw [densely dotted] (-1,-2) to (1,-2); 
  \draw [densely dotted] (-1,0) to (1,0);  
   \draw [densely dotted] (-1,2)  to (1,2);  
   \draw [densely dotted] (iw) to (tw);   
      \draw [densely dotted] (1,-2) to (tw);   
 \draw [densely dotted] (-1,0) to (iw);  
\draw [densely dotted] (tw) to (1,0);  
 \draw [densely dotted] (-1,-2) to (iw); 
       \draw[densely dotted] (1,0) to (1,2); 
 \draw [densely dotted] (-1,0) to (-1,2);    
 \end{scope}

\begin{scope}[xshift=6cm]. 
  \coordinate (iw) at (-1.5,-1);\coordinate (mw) at (-.5,-1); \coordinate (tw) at (.5,-1);  
   \coordinate (iwb) at (-3,-2);\coordinate (mwb) at (1,-2); \coordinate (twb) at (-1,-2); 
      \midarrow [thick, blue] (iwb) to (twb); \node [above] (b1) at (-2,-2) {$w$}; 
  \midarrow [thick, blue] (-1,0) to (1,0); \node [above] (b2) at (0,0) {$w$};  
   \midarrow [thick, blue] (-1,2)  to (1,2); \node [above] (b3) at (0,2) {$w$}; 
   \midarrow [thick] (iw) to (tw);  \node [above] (b4) at (mw) {$w$};  
    \draw [thick, red] (-1,0) to (iw);  \node   () at (-1,-.5) {$\WP$}; 
\draw [thick, red] (tw) to (1,0);  \node  () at (1,-.5) {$\WP$};  
 \midarrow [thick, red] (twb) to (tw);  \node () at (0.1,-1.5) {$v$};  
 \midarrow [thick, red] (iwb) to (iw);  \node  () at (-1.9,-1.5) {$v$}; 
       \midarrow[thick, forestgreen] (1,0) to (1,2);\node [right] (ctop) at (1,1) {$u$};  
 \midarrow [thick, forestgreen] (-1,0) to (-1,2);     \node [right] (cbot) at (-1,1) {$u$};  
  \midarrow  [thick] (iw) .. controls (-1.75,3) and (1,1) ..   (-1,0);  
\end{scope}
   \end{scope}
\end{tikzpicture}
\caption{Parallelotope structures $\F, \Fp$ on $\Sa^{\WP}$ such that $(\Sa^{\WP},d_\F)$ is isometric to $(\Sa^{\WP},d_{\Fp})$ and  $c^\Fp_{\omega}\circ i\circ (c^\F_\pi)\inv$   induces a twist $v\mapsto vw$.}\label{fig:one}
\end{center}
\end{figure}

\medskip
\begin{lemma}~\label{lem:t-minimal}   Let $\F$ be an allowable parallelotope structure  on $\SP$ and 
let $\tau\colon v\mapsto vw$ be an elementary twist.  If $v$ is twist-minimal then $\tau$ can be realized by an isometry  $i: (\SP,d_\F)\to (\SP,d_\Fp)$ for some allowable parallelotope structure  $\Fp$ on $\SP$. 
\end{lemma}
\begin{proof}   Let $\chi_w$ be a characteristic cycle for $w$.  Note that $w$ is twist dominant, so $\chi_w$ is a local geodesic. The carrier  $\kappa(H_v)$ of the hyperplane $H_v$  decomposes  combinatorially  as  a product
$$e_v\times \chi_w\times Z,$$
where $Z$ is the subcomplex of $\SP$ spanned by edges that  are adjacent to   
$v$ and don't split $w$. The orientation on $e_w$ induces an orientation on all edges of $\chi_w$.   We define a new decomposition of $e_v\times \chi_w$ by replacing each edge $e_v$  by the geodesic  from its initial vertex to its terminal vertex which cuts diagonally across all the parallelograms in in $e_v\times \chi_v$.  In a lift of $e_v\times\chi_v$ to $\USP$,  the new edge is a geodesic from the initial point of $\widetilde e_v$ to the terminal point of $w\widetilde e_v$ (this is what happened in Figure~\ref{fig:one}, where $\chi_w$ consisted of a single edge $e_w$).   Since the structure of $Z$ is unchanged, the new decomposition of $e_v\times \chi_w$   extends to a new parallelotope decomposition of   $\kappa(H_v)$ which is combinatorially isomorphic to the old one.  It does not change the metric on any parallelotope outside $\kappa(H_v)$, so extends to a new parallelotope structure $\Fp$ on $Y=\SP$.   Since $v$ is twist-minimal, skewing a single edge of a characteristic cycle is allowed, so this new parallelotope structure is allowable (Definition \ref{def:allowable}).  Note that  the identity on $Y$ is an isometry $(\SP,d_\F)\to (\SP,d_\Fp)$ but is not a combinatorial isometry $(\SP,\F)\to (\SP,\Fp)$.

The new collapse map $c^\Fp_{\pi}$ gives a new action of $\pi_1(\SaG)=\AG$ on $\widetilde Y$.  The only generator whose action has changed is $v$, whose new axis is the axis that was formerly the axis for $vw$.
\end{proof}

Notice that in the proof of  Lemma~\ref{lem:t-minimal} we skewed a single edge of a characteristic cycle for $v$.  If $v$ is twist-dominant  we cannot use that  trick to realize $\tau\colon v\mapsto vw$, since a characteristic cycle for $v$ must lift to a (straight!) axis for $v$ in  $\USP$.  Instead we will have to construct a new blowup structure $(S^\Omega,\mathcal G)$  on $Y$ to realize $\tau$.
The idea is to locate  branch points and twist-minimal hyperplanes using our identification of $Y$ with $\SP$, then show that these are metric invariants and use them to construct a new skewed blowup structure $(S^\Omega,\mathcal G)$ on $Y$.  To make this work we first need   to relate the geometry of $(Y,d)=(\SP,d_\F)$ to the combinatorial structure of $\SP$. The following proposition is the key.

\begin{proposition}\label{prop:straightening} Let $v$ be a twist-dominant generator of $\AG$.  The straightening map $s_\F\colon (\SP,\F)\to (\SP,\E)$ takes  axes  for $v$ in $(\USP,d_\F)$  to axes for $v$  in $(\USP,d_{\E})$ and the minset of $v$ to the minset of $v$, where the actions are given by  the collapse maps $c_\pi^\F$ and $c_\pi=c_\pi^\E$ respectively.   Moreover,  $s_\F$ maps branch points for $v$  in  $(\USP,d_\F)$ to branch points for $v$ in $(\USP,d_\E)$.  The same holds  if we replace $c_\pi^\E$  and $c_\pi^\F$  by any untwisted markings $h$ on $(\SP,\E)$ and $h'=h \circ s_F$ on $(\SP,\F)$.
\end{proposition}
\begin{proof}
 First assume the markings are standard collapse maps.
Since $v$ is twist-dominant,  each characteristic cycle for $v$ in both $(\SP,\F)$ and  $(\SP,\E)$ is a geodesic that is the image of an axis by Lemma \ref{lem:convex}.  The full minset $\Min(v)\subset \USP$ is the convex hull of the lifts of these characteristic cycles, and since $s_\F$ identifies these, it also takes the minset for $v$ in $(\USP,d_\F)$ to the minset for $v$ in $(\USP,d_\E)$.
The last statement about branch points follows from the fact that the straightening map induces a homeomorphism on links.

For a more general  untwisted  marking $h$, factor $h$ as $\sigma \circ h_0$ where $h_0 \circ c_{\pi}^{-1} \in U^0(\AG)$ and $\sigma$ is a graph automorphism.  Since $U^0(\AG)$ preserves twist-dominant generators up to conjugacy, the axes and minset of $v$ with respect to $h_0$ are just translates of the axes and minset with respect to $c_{\pi}$, so the argument above still applies.  For the graph automorphism, $\sigma(v)=w$ for some other twist-dominant generator $w$, so applying the proposition to $w$ gives the same result.
\end{proof}

For the standard metric $d_\E$, Lemma~\ref{lem:CharCycles} (1) gives a decomposition of the minset of a  generator $v$ with respect to the marking $c_\pi$ as  $\Min(v) \iso \alpha_v \times \UH_v$, and hence a projection $\pr_v\colon \Min(v)\to \alpha_v$.  This projection can be viewed either as the nearest-point (orthogonal) projection, or as collapsing hyperplanes whose labels  are adjacent to   $v$.  If $v$ is twist-dominant, then by the proposition above, the straightening map takes axes of $v$ in $(\USP,d_\F)$ to axes of $v$ in $(\USP,d_\E)$, and likewise minsets to minsets.  Thus, we can define an analogous projection in $(\USP, d_\F)$ by  ``straightening-projecting-unstraightening", i.e.
$$\pr_v^\F=s_\F^{-1}\circ \pr_v\circ s_\F$$
(see Figure~\ref{fig:eight}).  While this is no longer a nearest-point projection, it is again obtained by collapsing all hyperplanes whose labels  are adjacent to   $v$.  That is, for any paralleotope in the minset, $\pr_v^\F$ collapses every edge $e_A$ with $\max(A)\neq \{v\}$ to a point.

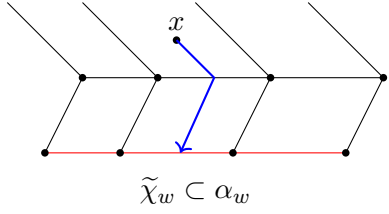
\begin{figure}
\begin{center}
  \begin{tikzpicture}
  \coordinate (e0) at (6,0); \coordinate (f0) at (7,0); \coordinate (g0) at (8.5,0); \coordinate (h0) at (10,0);
  \coordinate (e1) at (6.5,1); \coordinate (f1) at (7.5,1); \coordinate (g1) at (9,1); \coordinate (h1) at (10.5,1);
  \coordinate (x) at (7.75,1.5);
   \draw [red] (e0) to (h0);
 \node (av) at (8,-.5) {$\widetilde\chi_w\subset \alpha_w$};
 \vertex{(e0)}; \vertex{(f0)}; \vertex{(g0)}; \vertex{(h0)};
  \vertex{(e1)}; \vertex{(f1)}; \vertex{(g1)}; \vertex{(h1)};
  \draw (e1) to (h1);
 \draw (e0) to (e1) to (5.5,2);
  \draw (f0) to (f1) to (6.5,2);
   \draw (g0) to (g1) to (8,2);
    \draw (h0) to (h1) to (9.5,2);
  \vertex{(7.75,1.5)};\node [above] () at (x) {$x$}; \draw[blue,thick, ->] (7.75,1.5) to (8.25,1) to (7.8,0);
  \end{tikzpicture}
  \caption{The projection map $\pr_w^\F$ for 
  $w$ twist-dominant  }\label{fig:eight}
\end{center}
\end{figure}

\begin{proposition}\label{prop:tdom}  Let $\F$ be an allowable parallelotope structure  on $\SP$.  Let $v$ be twist-dominant and suppose $\tau\colon v\mapsto vw$ is an elementary twist. Then  there is an isometry $i: (\SP,d_\F)\to (\SPp,d_\Fp)$ that  realizes $\tau\circ\varphi$ for some $\varphi\in U^0(\AG)$ satisfying  $\tau\circ\varphi=\varphi\circ\tau$.
\end{proposition}

\begin{proof}  Since $v$ and $w$ commute, there is a vertex $x\in \SP$ which is a terminal vertex for edges $e_v$ and $e_w$.  Let $\chi_v$ and $\chi_w$ be characteristic cycles for $v$ and $w$ containing $e_v$ and $e_w$, $\tilde x$ a lift of $x$ to $\USP$ and $\tilde\chi_v, \tilde \chi_w$ lifts starting at  $\tilde x$ of these characteristic cycles.  Since both $v$ and $w$ are twist-dominant,  $\tilde\chi_v$ and $\tilde \chi_w$ are contained in axes $\alpha_v$ and $\alpha_w$ through $\tilde x$, and the product of these axes is a subcomplex of $\USP$ isometric to $\Eu^2,$ with stabilizer  the subgroup $\langle v,w\rangle\iso\Z^2$ of $\AG$.   The parallelogram in  $\USP$ spanned by $\tilde x, v\tilde x, w\tilde x$ and $vw\tilde x$ is a fundamental domain $D$ for this action (see Figure~\ref{fig:two}).
\begin{figure}.  
\begin{center}
  \begin{tikzpicture}[scale=.65].
   \draw  (0,0) to (5,0) to (4,4) to (-1,4) to (0,0);
 \draw [red] (-.25,1) to (4.75,1);
  \draw [red] (-.5,2) to (4.5,2);
    \draw (-.75,3) to (4.25,3);
    \draw (3,0) to (2,4);
       \draw (4,0) to (3,4);
       \redvertex{(0,0)};\node [below left] (x) at (0,0) {$\tilde x$};
         \redvertex{(5,0)};  \node [below right] (x) at (5,0) {$w\tilde x$};
            \redvertex{(4,4)};  \node [above right] (x) at (4,4) {$vw\tilde x$};
               \redvertex{(-1,4)};  \node [above left] (x) at (-1,4) {$v\tilde x$};
       \redvertex{(2.25,3)};       \redvertex{(3.75,1)};
 \draw[->] (4,0) to (4.5,0);
  \node () at (4.5,-.35) {$e_w$}; 
   \draw[->] (-.75,3) to (-.875,3.5);
  \node () at (-1.25,3.5) {$e_v$}; 
 \node () at (-.75, 1.5) {$e_{P_2}$}; 
  \node () at (-.6, .5) {$e_{P_1}$}; 
       \node () at (2.5,-.5)  {$\tilde\chi_w$};
   \end{tikzpicture}
  \caption{Fundamental domain $D$ for $\langle v, w\rangle\iso \Z^2$ on $\alpha_v\times\alpha_w\iso \Eu^2\subset \USP$.  The red dots and lines are the projections of all branch points for $v$ and $w$ onto $D$.}\label{fig:two}
\end{center}
\end{figure}
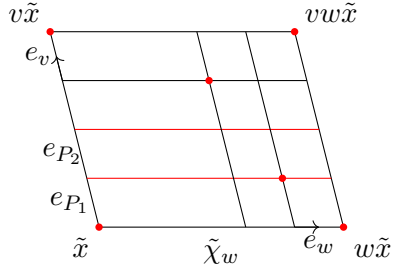
Define a map $p\colon(\USP,d_\F)\to \alpha_v\times \alpha_w$ by $p=(s_\F)\inv\circ p^\perp \circ s_\F$, where $p^\perp$ is nearest-point projection in $(\USP,d_\E)$.  We will be most interested in the  restriction of $p$ to $\Min(w),$ projecting $\Min(w)$ onto $\alpha_v\times\alpha_w$.

\medskip
{\bf Claim.} Let $\Br(v)$ be the set of branch points for $v$ and $\Br(w)$ the set of branch points for $w$.  Then $p(\Br(v))$ consists of lines parallel to $\alpha_w$ and isolated points, and $p(\Br(w))$ consists of isolated points.  The isolated points are vertices of $\SP$.
\medskip

\begin{proof}[Proof of claim.]
There is a branch point for $w$ at a vertex $x\in \chi_w\times H_w\subset \SP$ if and only if there is an edge $e_A$ at $x$ with $[A,w]\neq 1$.  If $x$ is a branch point for $w$ and $x\in\chi_v\times H_v$, then $x$ is also a branch point for $v$.   If $x$ is a branch point for $v$ but not for $w$, then all edges $e_A$ at $x$ that do not  are adjacent to 
$v$ must commute with $w$.  In this case every point of $x\times \chi_w$ is a branch point for $v$.
\end{proof}

 Let  $x\in \SP$ be a terminal vertex for edges $e_v$ and $e_w$ as above and $\tilde x$ a lift to $\tilde\chi_v\times\tilde\chi_w$.
If $w$ is central, then $\Br(w)$ is empty.  In this case, the characteristic cycle  for $w$ consists of the single edge $e_w$ and the only vertices on $\tilde\chi_w$ are the $w$-translates of $\tilde x$, but these are not branch points.    The same is true  for $\tilde\chi_v$  if $v$ and $w$ are both central.

Let $B=\Br(v)\cup \Br(w)$.  Note that the decomposition of $\tilde\chi_v\times\tilde\chi_w$ into parallelograms is completely determined by  $p(B)$ $\cup \{\tilde x\}$. This is because each edge of this decomposition is on a lift of a characteristic cycle $\chi_w$ or $\chi_v$, and each endpoint of this edge corresponds to a branch point in some (parallel) axis for $v$ or $w$  or to a translate of $\tilde x$.

 We are now ready to replace the action of $v$ by the action of $\tau(v)=vw$.
 Since $v\leq_tw$ the centralizer of $v$ is equal to the centralizer of $vw$, so $\Min(v)=\Min(vw)$ and $\Br(v)=\Br(vw).$ Thus replacing $v$ by $vw$ does not change $B$ nor the projections of branch points onto the plane $\alpha_v\times\alpha_w$.
 Replacing the fundamental domain $D$ of $\alpha_v\times\alpha_w$ by a new fundamental domain $D'$ with vertices $\tilde x, vw\tilde x, vw^2\tilde x$ and $w\tilde x$, these projections determine a decomposition of $D'$ into parallelograms.   The decomposition of $\alpha'_v=\alpha_{vw}$, the axis for $\tau(v)$,  is in one-to-one correspondence with the decomposition of $\alpha_v$, since in both cases, the vertices are projections of points in $p(B)$ parallel to an axis $\alpha_w$.   But the decomposition of $\alpha_w$ will change since vertices are now projections of $p(B)$ parallel to the new axis  $\alpha'_{v}$, instead of the old axis $\alpha_v$.  So for example, two points in $p(B)$ could project to the same point under one of these projections and to distinct points under the other.

 \begin{figure}
\begin{center}
  \begin{tikzpicture}[scale=.65]
 \draw [densely dotted] (0,0) to (5,0) to (4,4) to (-1,4) to (0,0);

\draw [blue]  (0,0) to (5,0) to (9,4) to (4,4) to (0,0);
 \draw [red] (1,1) to (6,1);
  \draw [red] (2,2) to (7,2);
     \draw [blue] (3,3) to (8,3);
        \draw [blue] (2.75,0) to (6.75,4);
      \draw [blue] (4.25,0) to (8.25,4);
         \redvertex{(0,0)};\node [below left] (x) at (0,0) {$\tilde x$};
         \redvertex{(5,0)};  \node [below right] (x) at (5,0) {$w\tilde x$};
            \redvertex{(4,4)};  \node [above right] (x) at (4,4) {$vw\tilde x$};
                \redvertex{(9,4)};  \node [above right] (x) at (9,4) {$vw^2\tilde x$};
               \redvertex{(-1,4)};  \node [above left] (x) at (-1,4) {$v\tilde x$};
       \redvertex{(2.25,3)};    \redvertex{(7.25,3)};     \redvertex{(3.75,1)};

  \end{tikzpicture}
  \caption{Skewing $D$. The branch points for $v$ and $w$ are the same as the branch points for $vw$ and $w$.}\label{fig:three}
\end{center}
\end{figure}
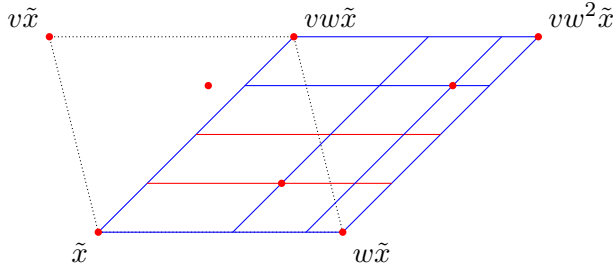

We  claim that  $D'$ together with its decomposition is part of a skewed $\G$-complex structure $(\Sa^\Omega,\Fp)$ on $(Y,d)$.
To prove this we need to do two things.  The first is to complete the new parallelogram decomposition of $\alpha_v\times\alpha_w$ to a parallelotope decomposition  of  all of $Y$.  The second is to find a compatible set $\Omega$  of partitions corresponding to this decomposition, i.e. a parallelotope structure $\Fp$  on $\Sa^\Omega$ making $(\Sa^\Omega,d_\Fp)$ isometric to $(Y,d)=(\SP,d_\F)$.\\

\noindent{\em Parallelotope decomposition.} We have changed the decomposition of the axis $\alpha_w$  into edges. As collateral damage, we have also changed the decomposition of any characteristic cycle with a lift that intersects $\alpha_w$.  However the endpoints of the intersection interval are images of branch points for $w$, so are still vertices in the new decomposition, i.e. this segment of the characteristic cycle is the only thing we have changed.  (In particular if the intersection is a single point we have not changed this characteristic cycle at all.)

 If $w$ commutes with $u\in  \max(A)$ for some label $A$ then the decomposition of every product subcomplex $\chi_w\times\chi_u$ of $\SP$ is affected by changing the decomposition of $\alpha_w$.
 If $u$ also commutes with $v$, then this is not a problem because then the new decomposition of $\alpha_v\times\alpha_w$ extends to a decomposition of  $\alpha_v\times \alpha_w\times e_A\subset \USP$.

If $[u,w]=1$ but $[u,v]\neq 1$ it may happen that some partition $\WP$ that splits $v$ also splits $u$, so that $\chi_u\times\chi_w$ overlaps $\chi_v\times\chi_w$ in the band $e_\WP\times \chi_w$.  We have changed the decomposition of this band.  However, notice that $u\leq_fv$ so $u$ cannot be twist-dominant  and moreover $u\leq_t w$.  Since $u$ is twist-minimal we can compensate for what we have done by using the band $e_u\times \alpha_w\subset \tilde\chi_u\times \alpha_w$ to skew the characteristic cycle for $u$ back to the original endpoint of $\tilde\chi_u$ (see Figure~\ref{fig:six}). We do not change the angle with $\alpha_w$ in any other band, so preserve the condition of allowability for the new parallelotope structure.\\

 \begin{figure}
\begin{center}
  \begin{tikzpicture}[scale=.75]
 \draw  (0,0) to (5,0) to (4.5,2) to (-.5,2) to (0,0); 
 \draw (4.5,2) to (4.5,4) to (-.5,4) to (-.5,2);
 \draw[->] (-.5,2.75) to (-.5,3.5);
 \node () at (-1,3.5) {$e_u$}; 
 \node () at (-.75, 1.5) {$e_{P_2}$}; 
  \node () at (-.6, .5) {$e_{P_1}$}; 
   \draw [red] (-.25,1) to (4.75,1);
  \draw [red] (-.5,2) to (4.5,2);
    \draw (-.5,2.75) to (4.5,2.75); %
\draw(3,0) to (2.5,2) to (2.5,4);
       \draw(4,0) to (3.5,2) to (3.5,4);
       \node () at (-2,2) {$\tilde\chi_u$};
\draw [blue]  (0,0) to (5,0) to (7,2) to (2,2) to (0,0); 
 \draw [red] (1,1) to (6,1);
  \draw [red] (2,2) to (7,2);
        \draw [blue] (2.75,0) to (4.75,2); 
      \draw [blue] (4.25,0) to (6.25,2); 

         \redvertex{(0,0)};\node [below left] (x) at (0,0) {$\tilde x$};
         \redvertex{(5,0)};  \node [below right] (x) at (5,0) {$w\tilde x$};
      \redvertex{(4.5,4)};  \node [above right] (x) at (4.5,4) {$uw\tilde x$};
               \redvertex{(-.5,4)};  \node [above left] (x) at (-.5,4) {$u\tilde x$};
       \redvertex{(3.75,1)};
\begin{scope}[xshift=10cm]
  \draw [densely dotted] (0,0) to (5,0) to (4.5,2) to (-.5,2) to (0,0); 
  \draw [densely dotted] (4.5,2) to (4.5,4) to (-.5,4) to (-.5,2);  
\draw[blue]  (4.5,4) to (-.5,4);  
   \draw [red] (-.25,1) to (4.75,1); 
  \draw [red] (-.5,2) to (7,2); 
    \draw (-.5,2.75) to (7,2.75); 
\draw [blue]  (0,0) to (5,0) to (7,2) to (2,2) to (0,0);
\draw[blue] (2,2) to (2,2.75) to (-.5,4);
\draw[blue] (7,2) to (7,2.75) to (4.5,4);
 \draw [red] (1,1) to (6,1); 
  \draw [red] (2,2) to (7,2); 
        \draw [blue] (2.75,0) to (4.75,2) to (4.75,2.75) to (2.25,4);
      \draw [blue] (4.25,0) to (6.25,2) to (6.25,2.75) to (3.75,4);
         \redvertex{(0,0)};\node [below left] (x) at (0,0) {$\tilde x$};
         \redvertex{(5,0)};  \node [below right] (x) at (5,0) {$w\tilde x$};
      \redvertex{(4.5,4)};  \node [above right] (x) at (4.5,4) {$uw\tilde x$};
               \redvertex{(-.5,4)};  \node [above left] (x) at (-.5,4) {$u\tilde x$};
       \redvertex{(3.75,1)};
\end{scope}
  \end{tikzpicture}
  \caption{ Skewing $e_u$ back in $\tilde\chi_u\times \alpha_w$ when $e(\WP_1), e(\WP_2)$ are in $\chi_u$}\label{fig:six}
\end{center}
\end{figure}
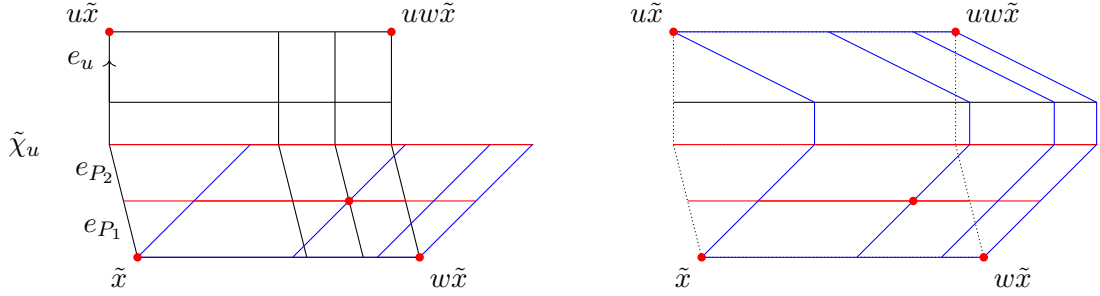

\noindent{\em Blowup structure.} We need to find a set of partitions $\Omega$ corresponding to our new parallelotope decomposition.
In particular we need to show that the new decomposition of $\alpha_w$ comes from a set of \GW\ partitions that split $w$. Recall from Section~\ref{sec:CharCycles} that the partitions $\{\WP_1,\ldots,\WP_k\}$ that split $w$ are nested, i.e. their $w$-sides $P_i$ satisfy $P_0=\{w\}\subset P_1\subset P_2\subset\ldots\subset P_k\subset P_{k+1}=\overline P_0\setminus\{w\inv\}$, and if $\WR$ is any other partition in $\bPi$ that  is not adjacent to   
$w$, its non-$w$ side  $\overline R$ is contained in some piece $dP_i=P_i\setminus P_{i-1}$ of the nest.

 Since $v$ is twist-dominant, the partitions splitting $v$ are also nested, say $Q_0=\{v\}\subset Q_1\subset Q_2\subset\ldots\subset Q_\ell\subset Q_{\ell+1}=\overline Q_0\setminus\{v\inv\}$, and the pieces $dQ_j=Q_j\setminus Q_{j-1}$ are unions of $v$-components of $\G^\pm$.  (Recall from Section~\ref{sec:GWpartitions} that a $v$-component is a connected component of $\G^\pm\setminus \lk^\pm(v)\setminus \{v,v^{-1}\}$ and that each side of a partition based at $v$ is a union of $v$-components plus $v$ or $v\inv$.)  Since $\st(v)\subseteq \st(w)$,  these $v$-components are unions of $w$-components plus possibly some elements of $\lk(w)$.  Thus  the intersection of a set of $v$-components with a set of $w$-components is a set of $w$-components.  In particular each intersection $I_{ij}=dP_i \cap dQ_j$  is a union of $w$-components.

Each vertex $\re_{ij}$ of $\chi_v\times\chi_w$ is a region that contains the consistent set
$$\{\overline P_1,\ldots\overline P_{i-1},P_i,\ldots,P_k, \overline Q_1,\ldots, \overline Q_{j-1},Q_j,\ldots, Q_\ell\}.$$
Partitions that  are not adjacent to  
$w$ also  are not adjacent to  
$v$, so have sides $\overline R_i$ that fit into both nests (the sides that don't contain $v$ or $w$), and $\re_{ij}$ must also contain the consistent set
$$\mathcal S_{ij}=\{\overline P_1,\ldots\overline P_{i-1},P_i,\ldots,P_k, \overline Q_1,\ldots, \overline Q_{j-1},Q_j,\ldots, Q_\ell,  R_1,\ldots,  R_m\}$$
The remaining partitions in $\bPi$  are all adjacent to  $w$.

If   $I_{ij}=dP_i\cap dQ_j$ contains some outermost $\overline R_s$ or a vertex $u$ outside all of the $\overline R_s$, we can can use this to extend $\mathcal S_{ij}$ to a region incident to an edge labeled $\WR_s$ or $u$ as we did in Section~\ref{sec:tdomCC}.  This region is  is a branch point for $w$ in some parallel copy of $D,$ and projects to $\re_{ij}$.

On the other hand, suppose $I_{ij}$ contains no $\overline R_s$ or outermost vertex $u$.  Then no extension of $\mathcal S_{ij}$ produces a region incident to an edge labeled $\WR_s$ or $u$.  Since every edge that branches off $\Min(w)$ has such a label, no such region gives a branch point for $w$, i.e. $\re_{ij}$ is not in the image of $\Br(w)$.

Identifying $(\alpha_v \times \alpha_w) = (\alpha_{vw} \times \alpha_w)$ we get a new fundamental domain $D'$ and a new map $\pr'_w \colon D' \to \alpha_w\cap D'$ which projects along $vw$-axes.    Using $\pr'_w$, project  those $\re_{ij}$ that are images of branch points for $w$ to an ordered set of points
$(\dsx_1,\ldots,\dsx_n)$ on $\alpha_w\cap D'$.

Let $I(\dsx_k)$ be the union of the $I_{ij}$ such that $\pr'_w(\re_{ij})=\dsx_k$. Let $P'_1=\{w\}\cup I(\dsx_1), P'_2=P'_1\cup I(\dsx_2)$  etc.  Each $P_i'$ is a side of a valid \GW\ partition $\WP_i'$ based at $w$, since each $I_{ij}$ is a union of $w$-components.

Let $\Omega$ be the collection of $\G$-Whitehead partitions obtained from $\bPi$ by replacing $\WP_1, \dots \WP_k$ by $\WP'_1, \dots \WP'_n$.  To see that the $\Omega$ partitions are pairwise compatible, we need only check that $\WP'_i$ is compatible with $\WR_j$ for all $i,j$.  We know that the side $\overline R_j$  lies in some $I_{st}$, and hence in some $I(\dsx_k)$.  So by definition, $\overline R_j \subset P'_k \setminus P'_{k-1}$ and it follows  that $\WR_j$  is compatible with $\WP'_i$ for all $i$.\\

\medskip\noindent{\em Marking change.}  Finally, we calculate the effect of replacing the structure $(\SP,\F)$ on $Y,$ with its marking $c_\pi^\F,$ by the new structure $(\SPp,\Fp)$ and marking $c_\omega^\Fp$.

\begin{lemma}\label{lem:markingchange} Suppose $v$ is twist-dominant and let $\tau\colon v\mapsto vw$ be an elementary twist. The composite map $c_\omega^\Fp\circ (c_\pi^\F)^{-1} : \SaG \to \SaG$ is of the form $\tau\circ \varphi$, where $\varphi\in U^0(A_\G)$ and $\tau\circ\varphi=\varphi\circ\tau$.
\end{lemma}

\begin{proof}

 Let  $\mu=c_\omega^\Fp\circ (c_\pi^\F)^{-1}  \colon \SaG \to \SaG$.
  The corner point $\widetilde x$ of the fundamental domains $D$ and $D'$ described above is a terminal vertex of edges $\widetilde e_v$ and $\widetilde e_w$ in $\widetilde{\Sa}^\Omega$ as well as in $\USP$.    Let $x$ be its image in $\SP$, and
for each $u\in V$, let $\xi_u$ be an edge path which goes from $x$ to an $e_u$ edge in $\EP$ , across $e_u$, and then back to $x$ in $\EP$. Note that $\xi_u$  crosses a single $e_u$ edge and all other edges are labeled by partitions. We choose $\xi_u$ to have minimal length among all such paths.  $\xi_u$ represents the homotopy class $(c_\pi^\F)^{-1}(u)\in \pi_1(\SP,x)$.

Lift $\xi_u$ to a path $\widetilde{\xi}_u$ based at $\widetilde{x}$. The endpoint $\widetilde{y}$ of $\widetilde{\xi}_u$ is then $u\cdot \widetilde{x}$ with respect to the $c_\pi^\F$-marking. Since $\xi_u$ was taken to be minimal, $\widetilde{\xi}_u$ is a combinatorial geodesic (i.e. it crosses each hyperplane in $\USP$ at most once), and our choice of $x$ means $\widetilde{x}$ and $\widetilde{y}$ are vertices in  $(\widetilde{\Sa}^\Omega,\Fp)$.  Any minimal length edge path $\widetilde{\eta}_u$ in $(\widetilde{\Sa}^\Omega,\Fp)$ between $\widetilde{x}$ and $\widetilde{y}$ consists of edges that cross a hyperplane which separates $\widetilde{x}$ and $\widetilde{y}$. To calculate $\mu(u)$, it is enough to know the $\Omega$-labels of hyperplanes that are crossed by $\widetilde{\eta}_u$. The only hyperplanes and labels that change as we go from $(\SP,\F)$ to $(\SPp,\Fp)$ are those with $\max=w$.  Thus, $\widetilde{\eta}_u$ crosses one hyperplane labeled $u$, and all other hyperplanes are either labeled by partitions or by $w$.

It follows that $\mu(u)=w^{n_u}uw^{m_u}$ for some $n_u,m_u\in \Z$. In particular, the twist-component of $\mu$ is a product of elementary twists by $w$. By construction, a $c_\pi^\F$-axis for $v$ maps to a  $c_\omega^\Fp$-axis for $vw$, so we know that $\mu(v)=vw$. If $u\neq v$ but $u\leq_t w$, either $u$ is twist-dominant, so the axis for $u$ has not changed, or we have sheared the $e_u$ edge so that a $c_\pi^\F$-axis for $u$ maps to a  $c_\omega^\Fp$-axis for $u$. Thus, $\mu(u)=u$.  This proves that the twist component $\tau$ of $\mu$ is just $\tau\colon v\mapsto vw$. Therefore we can write $\mu=\tau\circ \varphi$, where $\varphi$ is a product of folds and partial conjugations by $w$. Thus, $\varphi\in U^0(A_\G)$ and since $\tau$ is a twist by $w$, $\tau\circ \varphi=\varphi\circ \tau$, as desired.
\end{proof}

This completes the proof of Proposition~\ref{prop:tdom}.
\end{proof}

We next make some observations about changing the order of elementary twists, folds and partial conjugations.

\begin{definition} Let $\tau\colon v\mapsto vw$ be an elementary twist.  If $v$ is twist-dominant we say $\tau$ is a {\em $TD$ twist}, and if $v$ is twist-minimal we say $\tau$ is a {\em $TM$ twist}.
\end{definition}

\begin{lemma}\label{lem:TDcommuting} Let $\tau\colon v\mapsto vw$ be an elementary twist.
\begin{enumerate}
\item Let $\varphi$ be a partial conjugation or an elementary fold. Then either $\varphi$ commutes with $\tau$ or
$\tau\varphi=\alpha \varphi\tau $
where $\alpha$ is a partial conjugation, an elementary fold, or an elementary $TM$ twist by $w$ that commutes with both $\varphi$ and $\alpha$.
\item  If $\tau$ is a $TD$ twist and $t$ is a $TM$ twist, then either $t$ commutes with $\tau$ or $\tau t=\alpha t \tau $ where $\alpha$ is an elementary $TM$ twist by $w$ that commutes with both $\tau$ and $t$.
\end{enumerate}
\end{lemma}

 \begin{proof} 
(1) First suppose $\varphi$ is conjugates a component $C$ of $\G\setminus \st(u)$ by $u$.   Since $v$ and $w$ are connected by an edge in $\G,$ $\varphi$ commutes with $\tau$ unless $u=v$, in which case $\varphi\tau$ agrees with $\tau\varphi$ except that $\varphi\tau$ conjugates $C$ by  $vw$ instead of $v$.   Since $\st(v)\subset \st(w)$, $C$ is a union of components $C_i$ of $\G\setminus \st(w)$  plus some elements of $\st(w)$, so we can correct this by  partially conjugating the $C_i$ by $w\inv$.

Next suppose $\varphi$ is a right fold $\rho_{xy}\colon x\mapsto xy$ or left fold $\lambda_{xy} \colon x\mapsto yx$.  It cannot be that $w=x$ since that would mean
$v\leq_tw\leq_fy$.  Therefore $\tau$ commutes with $\varphi$ unless $v=y,$ in which case
  $$\varphi\tau\alpha=\tau\varphi$$
 where $\alpha: x \mapsto xw$ if $\varphi$ is a right fold, or $\alpha: x \mapsto wx$ if $\varphi$ is a left fold. Note that $\alpha$ may be either a fold if $[x,w]\neq 1$ or a twist  if $[x,w]=1$.  Since $x\leq_fy$, this implies $x$ cannot be twist-dominant, so if $\alpha$ is a twist, it is a $TM$ twist.  In either case, since $v$ commutes with $w$,  $\alpha$ commutes with both $\tau$ and $\varphi$.

(2)  Let $t\colon x\mapsto xy$ be a TM twist.  Then $x\neq w$ since $w$ is twist-dominant, so $t$ commutes with $\tau$ unless $v=y$.
  If $v=y$, then $x$ must commute with $v$ and hence also with $w$.  In this case, $\tau t=\alpha t \tau$ where $\alpha: x\mapsto xw$, which is a TM twist commuting with both $\tau$ and $t$.  
\end{proof}

 Recall that $\Out^0(\AG)$ is the subgroup of $\Out(\AG)$ generated by folds, twists, partial conjugations and inversions.  By checking the generators, it is not hard to see that graph automorphisms normalize $\Out^0(\AG)$, hence it is a normal subgroup.

\begin{corollary} \label{cor:factorization} Let $\langle TM \rangle$ denote the subgroup of $\Out^0(\AG)$  generated by $TM$ twists and let $G$ be the subgroup generated by $U^0(\AG)$ and  $\langle TM \rangle$. 
\begin{enumerate}
\item Any element $g \in G$ can be factored as 
$g=t_1 \circ \phi_1 =  \phi_2 \circ t_2 $ where $\phi_i \in U^0(\AG)$ and $t_i \in \langle TM \rangle$.
\item $TD$ twists normalize $G$, hence any element  of $\Out^0(\AG)$ can be factored as a product of an element of $\langle TD \rangle$, an element of $\langle TM \rangle$, and an element of $U^0(\AG)$  in any order.   The  $U^0(\AG)$ and $\langle TM \rangle$ factors may depend on the choice of order, but the $\langle TD \rangle$ factor remains unchanged. 
\end{enumerate}
\end{corollary} 

\begin{proof} First note that inversions normalize the subgroup of $\Out^0(\AG)$ generated by folds, twists, and partial conjugations.  Thus any inversion can be moved past any twist.  For (1), it remains to consider the case where $t_1= \tau$ is a single $TM$-twist and $\phi_1 = \varphi_1 \cdots \varphi_n$ is a product of folds and partial conjugations.  Applying part (1) of Lemma \ref{lem:TDcommuting} repeatedly gives 
$$ t_1 \circ \phi_1  = \tau \varphi_1 \cdots \varphi_n = (\varphi_1 \alpha_1\cdots \varphi_n\alpha_n)\tau$$
where each $\alpha_i$ is either the identity, a partial conjugation, an elementary fold, or a $TM$ twist by the same element $w$.  In particular, all of the $\alpha_i$'s commute with each other.  If all $\alpha_i$ lie in $U^0(\AG)$ we are done, but if one or more $\alpha_i$ is a twist, then we must apply the lemma again to move these twists to the right.  Since $\alpha_i$ commutes with the other $\alpha_j$'s, only moving it past the $\varphi_j$ terms can introduce new factors and these, too, will commute with each other and with the $\alpha_i$. Repeating this process, we can move all of the newly introduced $TM$ twists to the right  to obtain a new factorization $\tau \phi_1 = \phi_2 t_2$ as desired.

For (2), the fact that $TD$-twists normalize $G$ follows immediately from Lemma \ref{lem:TDcommuting} since $\alpha$ always lies in $G$.   So for any $h \in \Out^0(\AG)$, we can write $h = g_1 \circ \ t = t \circ g_2$ where $t \in \langle TD \rangle$ and $g_i \in G$. By part (1), we can factor $g_i$ into an element of $U^0(\AG)$ and an element of $\langle TD \rangle$ in either order.  By part (2) of Lemma \ref{lem:TDcommuting}, we can also switch the order of the $TM$ and $TD$ twists if desired.
\end{proof}

 We can now  complete the proof that $\Theta$ is surjective. 
\begin{proof}[Proof of Proposition~\ref{prop:surjective}]

 By definition, a point in $\OG$ is a space $(Y,d)$ isometric to a skewed $\G$-complex $(\SP, d_\F)$, together with a homotopy equivalence $f:Y \to \SaG$.  For the purpose of this proof, we will identify $(Y,d)$ with $(\SP, d_\F)$.  Then a point in the fiber
$\Theta\inv(Y,d,f)$ is given by a skewed $\G$-complex $(X,\Fp)$, an untwisted homotopy equivalence $h\colon X\to \SaG$  and an isometry $i\colon (\SP, d_\F) \to (X,d_{\Fp})$ such that $h\simeq f\circ i$.  If we also choose a combinatorial isometry of $X$ with some blowup $\SPp$ then the picture is 
 \begin{center}\begin{tikzcd}[]
 \SPp\iso X \arrow[d,"c_\omega^\Fp"]  \arrow[r, "i"] \arrow[dr, "h"] & \SP \arrow[d, "f"] \arrow[dr,"c_\pi^\F"]\\
 \SaG&\SaG & \SaG\arrow[l,"\phi"]
 \end{tikzcd}
 \end{center}
 where $h \circ(c_\omega^\Fp)\inv \in U(\AG)$.  To prove the Proposition, we must find such an $(\SPp, \Fp, h)$. 

 Let $\phi = f \circ (c_\pi^\F)\inv$.  Since graph automorphisms normalize $\Out^0(\AG)$, we can write $\phi=\phi'\circ \gamma$, where $\phi'\in \Out^0(\AG)$ and $\gamma$ is a graph automorphism.  Then replacing $\SP$ by $\mathbb S^{\gamma(\Pi)}$ as in the proof of Lemma \ref{lem:Gcollapse}, we may assume $\phi\in \Out^0(\AG)$.  By composing $c_\pi^\F$ with an isometry of $\SaG$, we can change the collapse map as in the proof of Lemma \ref{lem:Gcollapse}, thereby removing $\gamma$.  Without loss of generality, we therefore assume $\phi\in \Out^0(\AG)$. By Corollary~\ref{cor:factorization}, we can factor $\phi$  as $\phi=\eta \circ t_1 \circ t_2$ where $\eta \in U^0(\AG)$,  $t_1$ is a product of $TD$ twists, and $t_2$ is a product of  $TM$ twists.  Elements of $U^0(\AG)$ act on the left on both $\tS$ and $\OG$ and the action commutes with $\Theta$, so the fiber over $(\SP,d_\F,f)$ is isomorphic to the fiber over $\eta\inv(\SP,d_\F,f)=(\SP,d_\F,\eta\inv f)$.  Thus we may assume that $\phi = t_1 \circ t_2$.
 Moreover, by Lemma~\ref{lem:t-minimal} we can realize $t_2$ by a change of parallelotope structure on $\SP$ (which changes the collapse map, but not the metric on $\SP$), so we may assume that  $t_2=\id$ and write $\phi=\tau_1\circ\tau_2\circ\ldots\circ \tau_k$, a product of elementary $TD$ twist.

By Proposition~\ref{prop:tdom}  we can find elements $\varphi_i \in U^0(\AG)$ and   a sequence of skewed blowup structures on  $Y$ realizing the compositions $\tau_i\circ\varphi_i$.  Composing these gives
 \begin{center}
 \begin{tikzcd}[column sep=width("bbbbbbbbbb")]
\SPp \arrow[d,"c_\omega^\Fp"] & \SP \arrow[d, "c_\pi^\F"] \arrow[l, "i"]  \arrow[dr,"f"]\\
\SaG &\SaG \arrow[l, "\tau_1\varphi_1\ldots\tau_k\varphi_k"] \arrow[r, "\phi"]& \SaG
 \end{tikzcd}
 \end{center}
By Corollary~\ref{cor:factorization}, we can rewrite
$$\tau_1\varphi_1\ldots\tau_k\varphi_k=t^\prime\circ \varphi \circ (\tau_1\ldots\tau_k) = t^\prime\circ \varphi \circ \phi $$
where $t^\prime$ is a product of $TM$ twists and $\varphi \in U^0(\AG)$.  By changing the parallelotope structure on $\SPp$ we may again arrange that $t^\prime=\id$, so the diagram above becomes
\begin{center}
 \begin{tikzcd}[column sep=width("bbbbbbbbbb")]
\SPp \arrow[d,"c_\omega^\Fp"]  & \SP \arrow[d, "c_\pi^\F"] \arrow[l, "i"] \arrow[dr,"f"]\\
\SaG &\SaG \arrow[l, "\varphi\circ \phi"] \arrow[r, "\phi"]& \SaG
 \end{tikzcd}
 \end{center}
 Setting $h= \varphi\inv\circ c^{\mathcal G}_\omega =  f\circ i^{-1}$, we have $h \circ (c_\omega^\Fp)\inv = \varphi\inv \in U^0(\AG)$, so
 $(\SPp, \mathcal G, h)$ is the desired point in the fiber. 
\end{proof}

 \subsection{Structure of fibers}
 \subsubsection{Finding twist-minimal hyperplanes}

  In this section we show that the set of twist-minimal hyperplanes in a marked twisted $\G$-complex depends only on the underlying metric and the marking, i.e. on the projection to $\OG$.

\begin{lemma}\label{lem:tminhyperplanes} Let $[X,\F,h]$ and $[X',\F',h']$ be two points in the fiber over $[Y,d,f]$. The images in $Y$ of twist-minimal hyperplanes in $(X,\F)$ and $(X',\F')$ are the same (both set theoretically and pointwise) and their carriers have the same width.
 \end{lemma}

\begin{proof}  Since $(X,\F,h)$ and $(X',\F',h')$ both project to $(Y,d,f)$, we can identify
 $(X,d_\F) \cong (Y,d) \cong (X',d_{\F'})$, that is, we consider $(X,\F)$ and $(X',{\F'})$ to be two different skewed $\G$-complex  structures on the same underlying space $Y$.  Using this identification, we have $h= f = h'$, so $f$ is untwisted in both of these structures.

Recall that in Section \ref{sec:tminimal} we defined the sets $\Sing_h(H)$  and $\max_h(H)$ for a hyperplane in a rectilinear $\G$-complex with an untwisted marking $h$.  The same definitions can be used for a hyperplane $H$ in a skewed $\G$-complex $(X, \F)$ providing $H$ is convex, that is, $v \in \Sing_h(H)$ if an axis $\alpha_{v}$  crosses some lift of $H$ in $(\widetilde X, \F)$, where the action is given by the isomorphism $h_* \colon\pi_1(X)\iso \pi_1(\SaG)=\AG$.   In addition, if $s_\F$ is the straightening map and $h_\F=hs_\F^{-1}$,
 then the induced map $\tilde s_\F$ on the universal cover is equivariant with respect to the markings determined by $h$ and $h_\F$. Thus if some lift of $H$ separates $x$ from $vx$ in $(X, \F)$, then the same holds after straightening.  In other words, 
 $\Sing_h(H) = \Sing_{h_\F}(s_\F(H))$.  Thus, by Corollary \ref{cor:tmin_characterize} and Lemma \ref{lem:convex}, a hyperplane $H$ in $(X, \F)$ is twist-minimal if and only if any lift $\UH$ is convex and $[\max_h(H)]$ is twist-minimal.

Assume $H$ in $(X, \F)$ is twist-minimal and let $v \in \max_h(H)$.  Then some lift $\UH$ lies in $\Min_h(v)$ and we can decompose $\Min_h(v)$ as a (not necessarily orthogonal) product  $\Min_h(v) = \alpha_{v} \times \UH$. We would like to apply Proposition \ref{prop:min_h}, but that proposition was proved only in the context of rectilinear $\G$-structures, so we first must straighten $(X, \F)$.
For twist-minimal elements, the straightening map $s_\F$ need not take axes to axes or minsets to minsets, but as observed above, it does take the carrier of $\UH$  to the carrier of a hyperplane $\UH'=s_\F(\UH)$ that also has maximal element $v$.  Hence the minset of $v$ in the straightened structure $(X, \E, h_\F)$  decomposes as $\alpha'_{v} \times \UH'$, where $\alpha'_v$ is an axis for $v$ with respect to the marking $h_\F$.  In particular, the straightening  map between carriers extends to a homeomorphism between these two minsets.  It follows that $s_\F^{-1}$ maps branch points in $\kappa(\UH)$ to branch points in $\kappa(\UH')$.   By Proposition \ref{prop:min_h}, $\kappa(\UH')$ contains branch points on both components of its boundary, so the same holds for $\kappa(\UH)$.  The position of $\UH$  is determined by the projection of these branch points on $\alpha_v$  via the projection map $\pr_v^\F=s_\F^{-1}\circ \pr_v\circ s_\F$.    Moreover, since $\UH$ is the convex hull of the $A_{\lk(v)}$-orbit of a point on $\alpha_{v}$, the projection map is determined by the CAT(0) metric and the marking $h$, independent of the choice of point on $\alpha_{v}$.

Since $\Min_h(v)$,  the projection map to $\alpha_{v}$, and the branch locus of $v$ depend only on the CAT(0) metric and the marking, they are the same for $(\widetilde X, \F)$ and $(\widetilde X', \F')$.
 \end{proof}

\subsubsection{Shearing $(\SP,\F)$}
Now let $(Y,d,f)$ be an arbitrary point of $\OG$. By Proposition~\ref{prop:surjective} the fiber $\Theta\inv(Y,d,f)$  is non-empty, so we may fix a point $(\SP,\F,h_0)$ in this fiber and use an isometry $(Y,d)\iso (\SP,d_\F)$ to identify $Y$ with $\SP$ and $f$ with the untwisted marking $h_0$.  After acting by the untwisted subgroup $U(\AG)$ we may further assume $f=c_\pi^\F$.

If $(X,\Fp,h)$ is any other point in the same fiber, then there is an isometry $i\colon (X,d_\Fp)\to (Y,d)$ with $f\circ i\simeq h$.  Using this isometry to identify $(X,d_\Fp)$ with $(Y,d)$ and $h$ with $f$, we can view $(X,\Fp)$ as a different decomposition of the same underlying metric space $Y$ into   (unlabeled) parallelotopes.  We say $(X,\Fp)$ is a {\em $\G$-complex structure on $Y$}.   To understand the topology of the fiber, we will compare an arbitrary $\G$-complex structure $(X,\Fp)$ with our given structure $(\SP,\F)$.

The action of $\AG$ on universal covers is given by $f$ in both cases, so the axes, minsets and branch points are the same. However, while $f$ is untwisted with respect to both structures, it is a $\G$-collapse map only for $(\SP,\F)$ where it is in fact the standard collapse map.

 By Lemma~\ref{lem:tminhyperplanes} the set of twist-minimal hyperplanes and their carriers are the same in both structures. Let $H$ be a twist-minimal hyperplane and $\UH$ a lift of $H$ to $\USP$.
The carrier $\kappa(\UH)$ has two boundary components, $\bdry_0$ and $\bdry_1$. Let $x_0$ be a branch point in $\bdry_0$; then each of $(\USP,\F)$ and $(\widetilde X,\Fp)$ must have an edge dual to $\UH$ with one endpoint at $x_0$.  In $(\SP,\F)$ hyperplanes are labeled, so we have $\UH=\UH_A$ for some $A\in \bPi\cup V$ with $\max(A)$ twist-minimal, and we label this edge $e_A$.  In the skewed $\G$-complex structure $(X,\Fp)$ the edge does not have a  label, so we will just call it $e_H$.

 By Lemma \ref{lem:convex}, $\UH$ is convex.  The elements of  $\UL(A)$ are twist-dominant and commute with each other, so $\UH$ contains a subspace of the form $e_A \times \Eu^+_A$ where $\Eu^+_A$ is an affine space generated by axes of elements in $\UL(A)$.
By definition of an allowable parallelotope structure, the edge $e_A$ was obtained from an orthogonal edge by rotating in the direction of $\Eu^+_A$.  The same applies to $e_H$, since $\UL(H) =\UL(A)$.  Letting $t(e_A)$ and $t(e_H)$ be the endpoints of $e_A$ and $e_H$ in $\bdry_1$, it follows that the subspaces $t(e_A) \times \Eu^+_A$ and $t(e_H) \times \Eu^+_A$ agree.   So the difference $s_A=t(e_H)-t(e_A)$ is vector in the vector space $U^+_A$ spanned by the axes of $\UL(A)$.  (See Figure~\ref{fig:shearvector}.)

 Finally, note that in defining $s_A$, we began by choosing an isometry $i\colon (X,d_\Fp)\to (Y,d)$.  While this isometry need not be unique, for any other such isometry $j$, we have  $$j^{-1} \circ i = (j^{-1}  \circ f^{-1})\circ (f \circ i) \simeq h^{-1} \circ h = id.$$ Recall that $Y$ decomposes as an orthogonal product $Y=Y_0\times T_{\mathcal{Z}(\AG)}$ where $T_{\mathcal{Z}(\AG)}$ is a torus of dimension equal to the rank of the center $\mathcal{Z}(\AG)$. It follows from the work of Bregman \cite{Breg} that the only isometries of $Y$ that are homotopic to the identity are translations of the central torus $T_{\mathcal{Z}(\AG)}$.  Such a translation has no effect on the relative position of $e_H$ and $e_A$, so $s_A$ is independent of the choice of $i$. 

\begin{figure}
\begin{center}
  \begin{tikzpicture}[xscale=.65, yscale=.5]
 \coordinate (x0) at (7.6,1); \coordinate (teA) at (5,6);\coordinate (teH) at (8,5); 
 \vertex{(x0)};   \node [below] () at (x0) {$x_0$}; 
  \vertex{(teA)};  \node[above] () at (teA) {$t(e_A)$}; 
 \draw [thick, blue] (x0) to (teA); \node (eA) at (5.5,3.5) {$e_A$}; 
 \vertex{(teH)};   \node[above right] () at (teH) {$t(e_H)$}; 
  \draw [thick, blue] (x0) to (teH); \node (e) at (8.5, 3.5) {$e_H$}; 
   \draw [thick, red, ->] (teA) to (teH);\node at (7,5.7) {$s_A$};
\coordinate (bfl) at (-1,0); \coordinate (bfr) at (10,0); \coordinate (brl) at (2,3); \coordinate (brr) at (12,3);  
\coordinate (tfl) at (-1,4); \coordinate (tfr) at (10,4); \coordinate (trl) at (2,7); \coordinate (trr) at (12,7); 
 \draw (bfl) to (bfr) to (brr) to (brl) to (bfl); 
\draw (tfl) to (tfr) to (trr) to (trl) to (tfl); 
  \node (kappa) at (1.5,2.25) {$e_A\times \Eu^+_A\subset\kappa(\UH)$}; 
   \node (U) at (2.5,5) {$t(e_A)\times \Eu^+_A$}; 
 \draw [gray] (bfl) to (tfl);
   \draw [gray] (bfr) to (tfr);
      \draw [gray] (brl) to (trl);
       \draw [gray] (brr) to (trr);
  \end{tikzpicture}
\end{center}
\caption{Shear  of $e_H$ with respect to $e_A$}\label{fig:shearvector}
\end{figure}
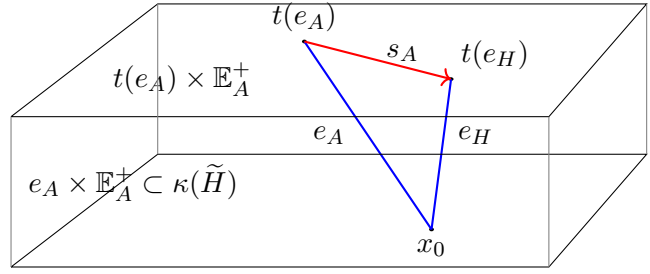

\begin{definition}  The vector $s_A=t(e_H)-t(e_A)\in U^+_A$
 is the {\em shear} of $e_H$ relative to $e_A$.
 \end{definition}

 \begin{definition}  A \emph{shearing} of $(\SP,\F)$ is a choice of vector  $s_A\in U^+_A$ for every hyperplane $H_A$, subject to the condition that  if $ \max(A)$ is twist-dominant, then $s_A=0$.  \end{definition}

We now observe that two $\G$-complex structures $(X,\Fp)$ and $(X',\Fp')$ in the fiber that define  the same shearing are the same.

\begin{proposition}\label{prop:uniqueness} Two points $[X,\Fp,h]$ and $[X',\Fp',h']$ in the fiber over $(\SP,d_\F,c^\F_\pi)= (Y,d,f)$ are the same if and only if they define the same shearings $\{s_A\}$ and $\{s'_A\}$ of $(\SP,\F)$.
\end{proposition}

\begin{proof} If $[X,\Fp,h]=[X',\Fp',h']$ then there is a combinatorial isometry $i\colon (X,\Fp)\to(X',\Fp')$ with $h'\simeq h\circ i$, (i.e. an isomorphism of cube complexes $X\iso X'$ which restricts to an isometry on each parallelotope), so the fact that corresponding edges have the same shearing is clear.

For the converse, suppose that $i\colon (X,d_\Fp)\to (X',d_{\Fp'})$ is an isometry of underlying metric spaces such that $h'\simeq i\circ h$.  Lift $i$ to an equivariant isometry $\tilde i\colon(\widetilde X,d_\Fp) \to (\widetilde X', d_{\Fp'})$.  By Lemma~\ref{lem:tminhyperplanes}, the CAT(0) metric  and the marking completely determine the twist-minimal hyperplanes, as well as the width of their carriers.  Hence $\tilde i$ maps each twist-minimal hyperplane $\UH$ to a twist-minimal hyperplane $\tilde i(\UH)$. The assumption on shearings now implies that the image of an  edge dual to $\UH$ is parallel to any  edge dual to $\tilde i(\UH)$ in $(\widetilde X',\Fp')$. To show that $\tilde i$ is a combinatorial isometry, we will show that it also maps twist-dominant hyperplanes in $(\widetilde X,\Fp)$ bijectively to twist-dominant hyperplanes in $(\widetilde X',\Fp')$.

 Suppose $v$ is twist-dominant and does not lie in the center of $\AG$.  Then by Propositions \ref{prop:min_h}  and \ref{prop:straightening}, the hyperplanes split by $v$ are completely determined by the projection maps $\pr_v^\Fp,~\pr_v^{\Fp'}$. Thus, to show that $i$ preserves these hyperplanes, it suffices to show that these two projection maps agree.
In both cases, the projection map may be thought of as performing a hyperplane collapse along all hyperplanes $\UH\subseteq \Min(v)$ whose maximal equivalence class commutes with $v$, where the collapse map takes the dual edges $e_H$ to a point.  Since $i$ takes twist-minimal hyperplanes to twist-minimal hyperplanes preserving the shearing and length of their dual edges, $\pr_v^\Fp$ and $\pr_v^{\Fp'}$ agree on twist-minimal hyperplanes $\UH\subseteq \Min(v)$.  If the maximal element $w \in \lk(v)$ is twist-dominant, the dual edge to $\UH$ lies along an axis $\alpha_w$  by Lemma \ref{lem:convex}. The entire axis is collapsed to a point under either of these projections. Since $i$ takes axes of twist-dominant generators to axes of twist-dominant generators, $\pr_v^\Fp$ and $\pr_v^{\Fp'}$ also agree along twist-dominant $\UH\subseteq \Min(v)$.  We conclude that the two projection maps are the same and hence determine the same twist-dominant hyperplanes.

When $A_\G$  has nontrivial center, $X$ and $X'$ decompose as (non-orthogonal) products with a locally convex torus endowed with a flat metric.  In each case, the parallelotope structure on the torus consists of a single parallelotope with opposite faces identified. In particular, any edge in the 1-skeleton of this torus is the image of an axis of some central element. As $i$ is an isometry and  $h'\simeq i\circ h$, the torus factors in $X$ and $X'$ agree as marked, metric tori.  Thus we may write $X=Z\times T, ~X'=Z'\times T$, where $Z,~Z'$ are subcomplexes, and $i$ maps every edge of the $T$-factor in $X$ parallel to an edge of the $T$-factor of $X'$. The above argument now shows that the combinatorial structure on $Z$ and $Z'$ must also agree, and that for every edge $e$ in the 1-skeleton $X$, $i(e)$ differs by translation in $T$ from an edge in the 1-skeleton of $X'$.  Since the 1-skeleton of $X$ is connected, $i$ differs from a combinatorial isometry by some fixed translation in $T$. Since any translation is isotopic to the identity, post-composing $i$ with the inverse of this translation gives a combinatorial isometry $i'\colon (X,\Fp)\rightarrow (X',\Fp')$ which still satisfies $h'\simeq h\circ i'$.
\end{proof}

\begin{corollary}\label{embedding}  The composite map $ \SG  \hookrightarrow  \tS  \xrightarrow{\Theta} \OG $ that forgets the cube complex structure on $[X,\Fp, h]$, is an embedding.
\end{corollary}

\begin{proof} Suppose $[X,\Fp,h]$, $[X',\Fp',h']$ are two rectilinear $\G$-complexes in the fiber over $[Y,d,h]\in \OG$. Then there is an isometry $i\colon (X',d_{\Fp'})\rightarrow (X,d_{\Fp})$ such that $h'\simeq i\circ h$. Since $\Fp,\Fp'$ are rectilinear, no shearing of edges  dual to twist-minimal hyperplanes is allowed, so by Proposition \ref{prop:uniqueness}, $[X, \Fp,h]=[X',\Fp',h']$ in $\tS$, and hence also in $\SG$.
\end{proof}

\subsubsection{Zero-sum shearings.} We now want to show that given any shearing of $(\SP,\F)$ satisfying a certain {\em zero-sum condition}, there is a skewed $\G$-complex structure $(X,\Fp)$ on $Y$ with that shearing.  Together with Proposition~\ref{prop:uniqueness} this gives us a characterization of all points in the fiber, which we can then use   to   prove that the fiber is contractible.

Let $v$ be a twist-minimal vertex of $\G$.  For any $A \subset \bPi \cup V$ with $ \max(A) \geq_f v$, we have $\UL(A) \subseteq \UL(v) \cup \UF(v)$, so the shearing vector $s_A \in U_A^+$ decomposes as
$$s_A = \ell^v_A + f^v_A$$ where the first factor lies in the subspace spanned by axes of $\UL(A) \cap \UL(v)$ and the second by the axes of $\UL(A) \cap \UF(v)$.  Note that if $v \in  \max(A)$, then $f_A^v=0$.

Now let $\chi_v$ a characteristic cycle for $v$ in $\SP$.  Let $H_{A_1},\ldots,H_{A_k}$ be the hyperplanes crossed by $\chi_v$ and orient the dual edges to be consistent with the orientation of $e_v$.  For all $i$ we have $ \max(A_i)\geq_f v$, so $s_{A_i} = \ell^v_{A_i} + f^v_{A_i}$.   Viewing all of the $\ell_{A_i}^v$ as vectors in the subspace of $U_A^+$  spanned by axes of $\UL(v)$, we can 
define $\ell_v$ to be the sum
$$\ell_v=\sum_i\ell_{A_i}^v.$$

\begin{definition} A shearing $\{s_A\}$ of $(\SP,\F)$ is a {\em zero-sum shearing} if $\ell_v=0$ for all  twist-minimal $v$.
\end{definition}

\begin{proposition}\label{prop:FiberZero}  If the images of $[X,\Fp,h]$ and $[\SP,d_\F,c^\F_\pi]$ in $\OG$ are equal, then $(X,\Fp)$ differs from $(\SP, \F)$ by a zero-sum shearing.
\end{proposition}
\begin{proof}
If  $[X,\Fp,h]$ and $[\SP,d_\F,c^\F_\pi]$ have the same image in $\OG$, then there is an isometry $i\colon (X, d_{\Fp})\to (\SP,d_\Fp) $  such that $h\simeq  c_\pi^\F\circ i$. Any such isometry lifts to an equivariant isometry on universal covers that takes minsets to minsets, axes to axes and twist-minimal hyperplanes to twist-minimal hyperplanes.
Let $u$ be twist-minimal and let  $\chi_u\subseteq \SP$ be a characteristic cycle for $u$ beginning at a vertex in the image of the branch locus $\Br(u)$.  Let $\eta_u$ be a minimal length edge path in $(X,\Fp)$ homotopic to $i^{-1}(\chi_u)$. Then $\eta_u$ and $\chi_u$ cross the same twist-minimal hyperplanes and lift to homotopic paths in  $\widetilde{X}\cong\USP$ with endpoints on some axis for $u$. Thus $\eta_u$ is a characteristic cycle for $u$ in $(X,\Fp)$.  Since only twist-minimal hyperplanes contribute to the total shearing along $\eta_u$, we conclude that $\ell_u = 0$.
\end{proof}

Conversely, we claim that any zero-sum shearing corresponds to a point in the fiber.

\begin{proposition}\label{prop:zerosum} Let $(Y,d,h)$ be the image of $(\SP, \F, c_\pi^\F)$ in $\OG$.  Any zero-sum shearing of $(\SP,\F)$ is realized by some skewed $\G$-complex structure $(X,\Fp)$ on $Y$  such that $h$ is untwisted with respect to this structure and hence $(X, \Fp,h)$ represents a point in the fiber over $(Y,d,h)$.
\end{proposition}

Before proving this proposition we deduce the following important corollary, which characterizes the fiber in terms of zero sum shearings.

\begin{corollary}\label{cor:fibers} $\Theta(X,\Fp,h)=\Theta(\SP,\F,c_\pi^\F)$ if and only if $(X,\Fp)$  differs by a zero-sum shearing  from $(\SP,\F)$ and $h\simeq c_\pi^\F\circ i$ for some isometry $i: X \to \SP$.
\end{corollary}
\begin{proof} If $\Theta(X,\Fp,h)=\Theta(\SP,\F,c_\pi^\F)$, then there exists an isometry $i\colon(X,d_{\Fp})\rightarrow (\SP,d_\F)$ such that $h\simeq c_\pi^\F\circ i$, and   $ (X,\Fp,h)$  differs by a zero-sum shearing  from $(\SP,\F,c_\pi^\F)$ by Proposition \ref{prop:FiberZero}. Conversely,  by Proposition \ref{prop:zerosum}, if $(X,\Fp)$ differs by a zero-sum shearing  from  $(\SP,\F)$ then $(X,\Fp)$ is a skewed $\Gamma$-complex with an isometry $i\colon(X,d_{\Fp})\rightarrow (\SP,d_\F)$ such that $c_\pi^\F\circ i$ is untwisted.  Since $h\simeq  c_\pi^\F\circ i$, we conclude that $\Theta(X,\Fp,h)=\Theta(\SP,\F,c_\pi^\F)$.
\end{proof}

The proof of Proposition \ref{prop:zerosum} will occupy the rest of this subsection.  As in the proof of surjectivity, we need to find a new decomposition of  $Y$ into parallelotopes, a corresponding skewed blowup structure $(\Sa^\Omega,\Fp),$ and then  determine the change of marking $c^\Fp_\omega\circ (c^\F_\pi)\inv.$

For each $A_i$ appearing in the characteristic cycle for $v$, $\ell^v_{A_i}$ decomposes into a sum of components lying along axes for $w \in \UL(v)$.  The zero-sum condition, $\ell_v=0$, implies that the components of  $\ell^v_{A_i}$ along the axis for each $w$ also sum to zero.  This means that we can achieve any zero-sum shearing by ordering the twist-dominant elements $w_i$, then first performing all shears towards $w_1$, then $w_2$, and so on.   At each stage, we will verify that the resulting parallelotope structure is a skewed $\Gamma$-complex with an untwisted marking.  At the final stage, we arrive at a skewed $\Gamma$-complex $(X,\Fp)$ that differs from $(\SP, \F)$ by the original zero-sum shearing.  This will prove the proposition.

\medskip\noindent{\bf Parallelotope decomposition.}
Assume we are shearing towards a single twist-dominant element $w$.  We will define the new paralleotope decomposition by determining the hyperplanes dual to the parallelotopes. The twist-minimal hyperplanes in the structure $(\SP,\F)$ will remain hyperplanes in the new decomposition, and we will eventually identify these with the twist-minimal hyperplanes in a new skewed $\G$-complex structure $(X,\Fp)$. The twist-dominant hyperplanes in $(X,\Fp)$ with maximal element  $w$ will be  defined using a projection of $\Min(w)$ to an axes for $w$.  The remaining twist-dominant hyperplanes will remain unchanged.

Choose a basepoint $x_0$ in $\Min(w)$ which is the terminal vertex of an edge labeled $w$ and  a branch point for $w$ in the structure $(\USP,\F)$.  Let $\alpha_w$ be the axis through $x_0$, viewed as a copy of the real line, based at $x_0$.
 We already have one  projection $\pr_w^\F=s_\F\inv \pr_ws_\F$ from $\Min(w)$ to $\alpha_w$ defined using the skewed blowup structure $(\SP,\F)$ on $Y$. The image of the branch locus $\Br(w)$  under this projection is a set of isolated points dividing   $\alpha_w$ into edges, and the inverse image of the midpoints of these edges are the hyperplanes $H_A$ with $ \max(A)=w$.  The image of any edge $e_A\in \Min(w)$ with $ \max(A)\neq w$ is a vertex of $\alpha_w$, while every axis for $w$ is sent isomorphically to $\alpha_w$.

Since $w$ is twist-dominant the subspace $\Min(w)$ is a subcomplex of $(\USP,\E)$, and therefore also of $(\USP,\F)$  by Proposition \ref{prop:straightening}.      We define a new projection map on edges of this subcomplex as follows.   Every (oriented) axis for $w$ can be identified with the real line $\R$ and this identification is unique up to translation.  Thus, segments of an axis can be viewed as vectors in $\R$ (up to translation). We first associate such a vector $r_A$ to each oriented edge $e_A$ in $\Min(w)$.   If $\max(A)=w$, then $e_A$ lies in an axis for $w$ and we let $r_A$ be the corresponding vector in $\R$.   If $\max(A) \neq w$, then the shearing of $e_A$ is given by a vector  $s_A\in U^+_A$.  Since we are only shearing toward $w$, $s_A$ lies along an axis for $w$ and we let $r_A= -s_A$.  Note that if $w\not\in \UL(A)$, then by definition of an allowable shearing, $r_A=0$.

Now define the new projection map $\pr'_w: \Min(w) \to \alpha_w$ as follows.  For any vertex $y$ in $\Min(w)$, choose a minimal edge path $e_{A_1} \dots e_{A_k}$ from $x_0$ to $y$ and set
$\pr'_w(y)= x_0 + \sum r_{A_i}$. Since the vectors $r_A$ depend only on the label $A$ and the orientation of $e_A$, this is independent of choice of path and two vertices connected by an edge $e_A$ will project to points that differ by the vector $r_A$.   Extending this map linearly on each parallelotope gives the desired projection.

We remark that $pr'_w$ can also be viewed as the map which collapses every hyperplane $\UH_A$ in $\Min(w)$ that does not split $w$.  The collapse is performed by identifying the hyperplane carrier with the product $e'_A \times \UH_A$ where $e'_A$ is the sheared version of $e_A$ (that is, an interval parallel to $e_A + s_A$)  and collapsing every copy of $e'_A$ to a point.

Now let $v$ be any generator that commutes with $w$ and $\chi_v$ a characteristic cycle for $v$ in the structure $(\SP,\F)$.   Then $\chi_v$ lifts to a path $p=e_{A_1} \dots e_{A_k}$ in $\Min(w)$ .   Since we are only allowing shearing in the direction of $w$, for each edge $e_{A_i}$ in $p$, we have $r_{A_i} = s_{A_i}=\ell^v_{A_i}$, so the zero-sum shearing condition says that $\sum r_{A_i} = 0$, or in other words,
the two endpoints $y$ and $vy$ of $p$ project to the same point under $\pr'_w$.  It follows that $\pr'_w$ is equivariant under the action of $A_{\st(w)}$.

Next observe that if an edge $e_A$ in $\Min(w)$ is contained in the branch locus $\Br(w)$ then $\max(A)$ must commute with some $u \notin \lk(w)$.  Thus $w \notin \UL(A)$ and hence $\pr'_w$ maps $e_A$ to a single point.  It follows that $\pr'_w$ takes each connected component of the branch locus to a single point. We declare these projection points to be the new vertices of $\alpha_w$;  note that $x_0$ is one of these vertices. This subdivides $\alpha_w$ into a new set of edges.  The inverse image under $\pr'_w$ of these edges form the carriers of the new hyperplanes that split $w$.

Now consider the hyperplane structure on $\widetilde Y$ consisting of the original hyperplanes which do not split $w$, together with the new hyperplanes that split $w$. These determine a new (equivariant) parallelotope structure $(\widetilde X,\Fp)$:  the maximal parallelotopes in $(\widetilde X,\Fp)$  are maximal intersections of carriers of these hyperplanes.

More explicitly, parallelotopes in $(\USP,\F)$ containing no edges $e_A$ with $\max(A) \leq_t w$, remain unchanged in $(\widetilde X,\Fp)$.  In particular, this is true for all parallelotopes not contained in  $\Min(w)$.   The $(\widetilde X,\Fp)$-structure on $\Min(w)$ consists of parallelotopes whose edges either lie in an axis for $w$ and project under $\pr'_w$ to a single edge in $\alpha_w$, or are parallel to $e_A + s_A$  in some $\kappa(\UH_A)$ and project to a single point in $\alpha_w$.  By the equivariance of $\pr'_w$, this descends to a parallelotope structure on the image of $\Min(w)$ in $\SP$.

 It remains to check that this new parallelotope structure is allowable in the sense of Definition~\ref{def:allowable}.  To see this, note that an allowable metric on a single parallelotope $\mathbf{c}$, as defined in  Definition~\ref{def:pre-allowable}, depends on the intersections of linear subspaces $K_i$ associated to edges $e_i$ emanating from a fixed vertex.  In our current terminology, if $e_i = e_A$, then $K_i$ is the subspace spanned by $e_A$ together with $U_A^+$. Since this subspace remains unchanged after shearing, the resulting metric on $\mathbf c$ is still allowable, so condition (1) of the definition is satisfied.
Condition (2), that if $\mathbf{c}'$ is a face of $\mathbf c$, then the metric on $\mathbf{c}'$ is the restriction of the metric on $\mathbf c$, is obvious.  For condition (3), note that if
$\max(A)=\{v\}$ is twist-dominant, then both $e_A$ and $e_v$ lie in the image $\overline{\alpha}_v$ of an axis for $v$ and neither of these edges are allowed to shear.  Thus if $B$ is adjacent to $A$, then any change in angle between $e_B$ and $e_A$ or $e_v$ must result from a shearing of the edge $e_B$.  This can only occur if $B$ is twist-minimal, in which case $H_B$ is locally convex and contains $\overline{\alpha}_v$.  It follows that any shearing of $e_B$ will change the angles between $e_B$ and any edge in $\overline{\alpha}_v$ by the same amount.

\medskip\noindent{\bf Blowup structure.} We have found a new decomposition $(X,\Fp)$ of $Y$ into parallelotopes.  The next thing to show is that $(X,\Fp)$ is a $\G$-complex, i.e.
we need to find a new set of partitions $\Omega$ such that $(X,\Fp)= (\Sa^\Omega,\Fp)$. The only difference between $\bPi$ and $\Omega$ will be the partitions that split our twist-dominant generator $w$.

Since $w$ is twist-dominant, $\Min(w)$ is a (convex) subcomplex of $(\USP,\F)$ by Proposition~\ref{prop:straightening}.  If $x$ and $y$ are vertices of $\Min(w)$ which are branch points for $w$, then there are edges $e_A$ adjacent to $x$ and $e_B$ adjacent to $y$ with $[A,w]\neq 1$ and $[B,w]\neq 1$.
Choose a lift of $\EP$ to $\USP$;  we will abuse notation by calling this $\EP$ as well. Let $M_w$ be the intersection of $\Min(w)$ with $\EP$. Since $\EP$ and $\Min(w)$ are both convex in the straightened version $(\USP,\E)$, their intersection is connected.
Since $\EP$ contains all vertices of $\SP$, every branch vertex in $\Min(w)$ has a unique translate in $M_w=\EP\cap \Min(w)$.

\begin{lemma} \label{lem:SameProj}Suppose $e_A$ and $e_B$ are edges branching off of $M_w$ at vertices $x$ and $y$ respectively.  If $A$ is a partition, let $A^\times$ denote the side of $A$ that doesn't contain $w$, and if $A$ is a vertex $v$, let $A^\times=v$ if $e_v$ terminates at $x$, and $A^\times=v\inv$ if $x$ is the initial vertex of $e_v$; define $B^\times$ similarly.
 If $a \in A^\times$ is maximal in $A$ and $b \in B^\times$ is maximal in $B$, and $a,b$ lie 
 in the same  $w$-component of $\G^\pm$, then  $x$ and $y$ project to the same point of $\alpha_w$ under $\pr_w'$.
\end{lemma}
\begin{proof} Since $M_w$ is a connected subcomplex we may connect $x$ and $y$ by a minimal-length edge-path $e_{A_1},\ldots, e_{A_r}$ lying in this intersection.
We claim that  $\max(A_i)\not\leq_t w$ for all $i$, so each $e_{A_i}$ collapses to a point under  $\pr'_w$.  Thus $\pr'_w$ maps the entire path to a point, showing $\pr'_w(x)=\pr'_w(y)$.

We argue by contradiction, so let $a_i\in\max(A_i)$ and suppose $a_i\leq_t w$ for some $i$.   Since $e_{A_i}\subset \Min(w)$, we have $[a_i,w]=1$.  If  $[a,a_i]=1$ then $a_i\leq_t w$ implies $[a,w]=1$, so $e_A\subset \Min(w)$, contradicting our hypothesis.   Thus we have $[a, a_i]\neq 1$ for all $i$, and similarly $[b,a_i]\neq 1$.

 The lift of the hyperplane $H_{A_i}$ containing $e_{A_i}$  separates $M_w$ into two components.   Since $a_i$ does not commute with either $a$ or $b$,  the endpoints of edges labeled $e_A$ and  those labeled $e_B$   lie in different components  (where orientation matters if $A$ or $B$ is a generator). In terms of partitions, the  sides  of $A$ and  $B$ that don't contain $a_i$ sit in different sides of the partition $A_i$. Since $a_i\in\lk(w)$ and $A^\times$ does not contain $w$ it  does not  contain $a_i$ either, and similarly $B^\times$ does not contain $a_i$.  Thus $A^\times$ and $B^\times$ are in different sides of $A_i$.
But each side of $A_i$  is a union of $a_i$-components plus $a_i$ or $a_i\inv$, and, since $a_i\leq_tw$, each $a_i$-component is a union of $w$-components plus possibly some elements of $\lk(w)$.  Since $a\in A^\times$,  $b\in B^\times$ and neither is in $\lk(w)$, this contradicts the hypothesis that $a$ and $b$ are in the same $w$-component.
\end{proof}

We now form the new partitions splitting $w$ in the same way we did in Section~\ref{sec:Surjectivity}.  To each branch point $\re\in \Br(w)\cap \EP$, associate the union $I(\re)$ of the sets $A^\times$ for  edges $e_A$ incident to $\re$ but not in $\Min(w)$.     The new projection $\pr'_w$ sends  $\Br(w)\cap M_w$ to an ordered set of points
$(\dsx_1,\ldots,\dsx_n)$ on $\alpha_w\cap \EP$, and to each $\dsx_i$ we associate the union $I(\dsx_i)$ of the $I(\re)$ with $\pr'(\re)=x_i$.  Let $P'_1=\{w\}\cup I(\dsx_1), P'_2=P'_1\cup I(\dsx_2)$  etc. By Lemma~\ref{lem:SameProj} each $I(\dsx_i)$ is a union of $w$-components, so each $P_i'$ is a side of a valid \GW\ partition $\WP_i'$ based at $w$.

Let $\Omega$ be the collection of $\G$-Whitehead partitions obtained from $\bPi$ by replacing $\WP_1, \dots \WP_k$ in $\bPi$ by $\WP'_1, \dots \WP'_n$.  To see that the $\Omega$ partitions are pairwise compatible, we need only check that each $\WP'_i$ is compatible with those $\WR\in\bPi$ that  are not adjacent to  
 $w$.  The side $R^\times$  is contained in some outermost $Q^\times$ in some piece $dP_i$.  The partition $\WQ$ cannot  be adjacent to  
 $w$, so there is an edge $e_\WQ$ at a branch point $\re\in M_w$, and $Q^\times\subset I(\re)$.  Since $I(\re)\subset I(\dsx_i)$ for some $i$.   it follows  that $\WR$  is compatible with $\WP'_i$.\\

\noindent{\bf Marking change.}
The blowup structure $(X=\SPp,\Fp)$ defined above comes with a collapse map $c_\omega^\Fp\colon \SPp\rightarrow \SaG$. We now analyze the change in marking induced by the difference between $c_\omega^\Fp$ and the original collapse map $c_\pi^\F$ from $(\SP,\F)$.

\begin{lemma}\label{lem:SingleShearMC} Suppose $(X=\SPp,\Fp, c_\omega^\Fp)$ is a zero-sum shearing of $(\SP,\F,c_\pi^\F)\in \tS$ which differs only in the direction of a twist-dominant generator $w$. Then the composite map $c_\omega^\Fp\circ( c_\pi^\F)^{-1} : \SaG \to \SaG$ is untwisted.
\end{lemma}
\begin{proof}
Let  $\mu=c_\omega^\Fp\circ( c_\pi^\F)^{-1}  : \SaG \to \SaG$. Observe that the only hyperplanes which change from $\SP$ to $\SPp$ are those with $ \max=\{w\}$.  Following exactly the same argument as in the proof of Lemma \ref{lem:markingchange}, for each $v\in V$ we have that $\mu(v)=w^{n_v}vw^{m_v}$ for some $n_v,m_v\in \Z$. Thus, nontrivial twists can occur only for $v\leq_t w$.  If such a $v$ is twist-dominant, then the characteristic cycle for $v$ is the same in both $\SP$ and $\SPp$, so $\mu(v)=v$.  If $v$ is twist-minimal, the fact that $(\SPp,\Fp)$ is a zero-sum shearing implies that a characteristic cycle for $v$ with respect to $(\SPp,\Fp)$ has the same endpoints as a characteristic cycle for $v$ with respect to $(\SP,\F)$, hence in this case $\mu(v)=v$ as well. Therefore, $\mu$ is a untwisted.
\end{proof}

 The proof of Lemma~\ref{lem:SingleShearMC} shows that $\mu$ acts trivially on vertices $v \leq_t w$, and one might be tempted to conclude that it shows that the action of $\mu$ is entirely trivial.  However, if $z \leq_f v \leq_t w$, then a characteristic cycle for $z$ may have an edge $e_A$ that also lies in a characteristic cycle for $v$.  In this case, a zero sum sheering of the characteristic cycle for $v$ in the direction of $w$ may result in $\mu$ acting as a non-trivial fold of $w$ onto $z$. 

\begin{remark}\label{rmk:TMinFixed}As observed in the proof of Lemma~\ref{lem:SingleShearMC}, the only hyperplanes that change in a zero-sum shearing are twist-dominant hyperplanes. In particular, the set of twist-minimal hyperplanes which split a particular $v$ does not vary among all zero-sum sheerings.
\end{remark}

We now finish the proof of Proposition \ref{prop:zerosum}:
\begin{proof} Order the twist-dominant generators $w_1,\ldots, w_n$. We perform an arbitrary zero-sum shearing as a sequence of single generator zero-sum shearings.  By the discussion above, we obtain a sequence of skewed blowups \[(\SP,\F)=(\Sa^{\Omega_0},\Fp_0), (\Sa^{\Omega_1},\Fp_1),\ldots, (\Sa^{\Omega_n},\Fp_n)=(\SPp,\Fp)\]
where for $1\leq i\leq n$, $(\Sa^{\Omega_i},\Fp_i)$ is obtained from $(\Sa^{\Omega_{i-1}},\Fp_{i-1})$ by a zero-sum shearing in the direction of $w_i$, and the change in marking $\mu_i$ is untwisted by Lemma \ref{lem:SingleShearMC}. The change in marking from $(\SP,\F)$ to $(\SPp,\Fp)$  is then a composition of untwisted automorphisms \[c_\omega^\Fp\circ (c_\pi^\F)^{-1}\simeq \mu_n\circ\cdots\circ \mu_1,\] hence untwisted as well.
\end{proof}

\subsection{Contractibility of $\OG$}

\subsubsection{Contractibility of fibers}
Let $\mathcal{H}_{min}$ denote the set of twist-minimal hyperplanes in $(\SP,\F,c_\pi^\F)$. Since these depend only on the metric $d_\F$ by Lemma \ref{lem:tminhyperplanes}, the set $\mathcal{H}_{min}$ is well-defined over the whole $\Theta$-fiber containing $(\SP,\F,c_\pi^\F)$.  Likewise, the twist-dominant axes remain the same throughout the fiber.
The dual edge to $H\in \mathcal{H}_{min}$ is allowed to shear in the direction of $\UL(H)$, and as above, we regard a given shearing $s_{H}$ as a vector in the vector space $U_H^+$.  (Here to emphasize the independence from $\bPi$, we use the notation  $s_{H}$ and $U_H^+$ rather than specifying a label $A$ and writing  $s_{A}$ and $U_A^+$.) We now describe the fiber containing $(\SP,\F,c_\pi^\F)$ as a linear subspace of \[\bigoplus_{H\in \mathcal{H}_{min}}U_H^+\]

Let $V_{min}$ denote the set of twist-minimal vertices.  For $v\in V_{min}$, the only edges which contribute to the shearing of $v$ are those dual to $H\in \mathcal{H}_{min}$ which split $v$. By Remark \ref{rmk:TMinFixed} and Corollary \ref{cor:fibers}, the set of twist-minimal hyperplanes that split $v$ does not change within the fiber. If $v\in \Sing(H)$, $H\in\mathcal{H}_{min} $, the contribution of $s_H$ to $l_v$ is $l_{H}^v$, where $l_{H}^v$ lies in the subspace of $U_H^+$ corresponding to $\UL(H)\cap\UL(v)$.  If $v\notin \Sing(H)$,  define $l_{H}^v\equiv 0$. We then identify $l_{H}^v$ with a vector in $U_v^+$ since $l_{H}^v$ lies in the span of axes in $\UL(H)\cap\UL(v)$. Thus, for each $v\in V_{min}$ we can think of $\ell_v$ as a linear map:

\begin{align*}l_v\colon\bigoplus_{H\in \mathcal{H}_{min}}U_H^+&\rightarrow U_v^+\\
\oplus s_{H}&\longmapsto \sum l_{H}^v
\end{align*}

Call the equations $\{l_v =0\mid v \in V_{min}\}$ the \emph{structure equations} for shearings of $(\SP,\F,c_\pi^\F)$. We now easily deduce the contractibility of the fibers from Corollary \ref{cor:fibers}.

\begin{theorem}\label{thm:contractible} The fibers of the map $\Theta: \tS \to \OG$ are contractible.  \end{theorem}

\begin{proof}   The space of solutions to the structure equations is the intersection $\cap_{v\in V_{min}}\ker l_v$, which is a linear subspace of $\bigoplus_{H\in \mathcal{H}_{min}}U_H^+$ and hence contractible. The preceding discussion shows that this subspace is in one-to-one correspondence with the set of zero-sum shearings of $(\SP,\F,c_\pi^\F)$.  Thus, by Corollary \ref{cor:fibers}, there is a bijection between the space of solutions and points in $\Theta^{-1}([\SP,d_\F,c_\pi^\F])$.  It is easy to see that this correspondence is a homeomorphism.  By Proposition \ref{prop:surjective}, every fiber of  $\Theta$ is a $U(A_\G)$-translate of one containing some $[\SP,\F,c_\pi^\F]$, so every fiber is contractible.
\end{proof}

\subsubsection{Contractibility of  $\OG$} We now finish the proof of Theorem \ref{thm:Theta}. By Theorem \ref{thm:contractible}, the fibers of $\Theta$ are contractible, but since they are not compact, $\Theta$ is not a proper map.  To conclude that $\OG$ is contractible, we will show that $\Theta$ is in fact a fibration.

\begin{proof}[Proof of Theorem \ref{thm:Theta}]   Since $\OG$ is paracompact (the equivariant Gromov-Hausdorff topology is metrizable),  it suffices to show that $\Theta$ is a fibration when restricted to sufficiently small neighborhoods $U \subset \OG$.

We begin by showing that for any point $y_0$ in $\OG$, and any lift $x_0 \in \tS$ of $y_0$, there exists a neighborhood $U$ and a section $s: U \to \Theta^{-1}(U)$
with $s(y_0)=x_0$.   Say $x_0=[X_0,\F_0,h_0]$ and $y_0=[Y_0,d_0,h_0]$, so that $(X_0,d_{\F_0})$ is isometric to $(Y_0,d_0)$.   By Proposition \ref{prop:surjective}, it suffices to consider the case when $h_0=c_0$ is a collapse map.

Consider the fiber over a point $y=[Y,d,c]$ in a small neighborhood $U$ of $y_0$. To define $s(y)$, we must choose a $\G$-complex structure $(X,\F)$ on $(Y,d)$.  For any such $\F$, the twist-minimal hyperplanes with $v$ as a maximal element are determined by the projection of the branch locus  $\Br(v)$  on an axis for $v$.  If $(Y,d,c)$ is close to $(Y_0,d_0,c_0)$  in the equivariant Gromov-Hausdorff topology, then these branch loci must also be close, and hence likewise their projections on an axes for $v$.  However, these projections can change in three ways as we move from $[Y_0,d_0,c_0]$ to $[Y,d,c]$.
\begin{itemize}
\item The distance between a pair of projection points may expand or contract.  This will affect the width of the carrier of the hyperplane separating these projection points.
\item  One projection point can split into multiple points.  This will require introducing new twist-minimal hyperplanes.
\item Two or more projection points may coalesce, causing the corresponding hyperplanes to merge.
\end{itemize}
Shrinking $U$ if necessary, we may avoid the coalescing of projection points and allow only changes of the first two types.  Moreover, for $U$ sufficiently small, the new twist-minimal hyperplanes will have carriers of width less than half that of the old twist-minimal hyperplanes, and thus (by abuse of notation) we may consider the set of twist-minimal hyperplanes in $\F_0$ to be a subset of those in $\F$.  Then the marking $c : X \to \SaG$ will correspond to collapsing the newly added hyperplanes, composed with the straighten-collapse map corresponding to $c_0$.

 The collection of twist-minimal hyperplanes at $[Y,d,c]$ is completely determined by the metric $d$.  The axes of twist-dominant generators are determined by the marking $c$.  As seen in Proposition \ref{prop:uniqueness}, once we have determined the twist-minimal hyperplanes, the shearing of their dual edges together with the branch locus completely determines $\F$.  Suppose $H$ is a new hyperplane, not coming from a hyperplane in $\F_0$.  We are free to choose the shearing on the dual edge by any vector in $U_H^+$. Different choices will only affect the determination of twist-dominant hyperplanes. Therefore we choose the dual edge to be orthogonal to $H$, of length equal to the width of $\kappa(H)$. If $H$ corresponds to a  twist-minimal  hyperplane in $\F_0$ which is collapsed by $c_0$, we leave the shearing unchanged  (that is, the angle between the dual edge and the axes of $\UL(H)$), but adjust the length of the dual edge to take account of the change in the width of the hyperplane carrier.  Finally, for twist-minimal hyperplanes not collapsed by $c_0$, namely those labeled $H_v$, we adjust the shearing so that the new characteristic cycle lifts to a path whose endpoints lie on an axis for $v$.  This determines a parallelotope structure $\F$ on $(Y,d)$
with the property that the shearing along twist-minimal hyperplanes satisfies the zero-sum condition relative to any skewed $\G$-complex in the fiber over $y$. Hence by Corollary \ref{cor:fibers}, $[X, \F, c]$ also lies in this fiber.
Set $s(y) = [X, \F, c]$.  Since the construction of $\F$ depends only on the metric $d$ and lift $[Y,\F_0,c_0]$, the map $s$ is well-defined and continuous.

Now let $Z$ be any space. Suppose  $f_t: Z  \to U$ is a homotopy and let $\hat f_0: Z \to \Theta^{-1}(U)$ be a lift of $f_0$.
We can lift $f_t$ to a homotopy $g_t =s\circ f_t : Y \to \Theta^{-1}(U)$, but $g_0$ need not agree with the given lift $\hat f_0$.  We can correct this by concatenating $g_t$ with a homotopy $h_t$ from $\hat f_0$ to $g_0$ which projects at all times $t$ to the map $f_0$.  To do this, use the fact that the fibers of $\Theta$ are convex subspaces of some Euclidean space, so the straight-line homotopy in each fiber from $\hat f_0(y)$ to $g_0(y)$ to gives such a homotopy $h_t$.   Then, up to reparameterizing the interval, $h_t$ followed by $g_t$ is a lift of $f_t$.

This shows that $\Theta$ is a fibration.  Since we have already proved that the fibers are contractible (Theorem \ref{thm:contractible}), we conclude that $\Theta$ is a homotopy equivalence.  By Corollary \ref{cor:tScontractible}, $\tS$ is contractible, so the same holds for $\OG$.
\end{proof}

\bibliographystyle{plain}
\bibliography{OSbib}

\end{document}